\documentclass[12pt]{amsart}

\usepackage{amsmath}
\usepackage{amsthm}
\usepackage{amssymb}
\usepackage{soul}
\usepackage{caption}
\usepackage[curve]{xypic}
\usepackage{cite}
\usepackage{enumitem}
\usepackage{color}
\usepackage{url}
\usepackage[usenames,dvipsnames]{xcolor}
\usepackage{graphicx}
\usepackage{tikz}
\usepackage{tikz-cd}
\usepackage{hyperref} 
\usepackage{verbatim}

\setlist{itemsep=4pt, topsep=4pt}

\makeatletter

\def\chapter{%
  \if@openright\cleardoublepage\else\clearpage\fi
  \thispagestyle{plain}\global\@topnum\z@
  \@afterindenttrue \secdef\@chapter\@schapter}

\def\@chapter[#1]#2{\refstepcounter{chapter}%
  \ifnum\c@secnumdepth<\z@ \let\@secnumber\@empty
  \else \let\@secnumber\thechapter \fi
  \typeout{\chaptername\space\@secnumber}%
  \def\@toclevel{0}%
  \ifx\chaptername\appendixname \@tocwriteb\tocappendix{chapter}{#2}%
  \else \@tocwriteb\tocchapter{chapter}{#2}\fi
  \chaptermark{#1}%
  \addtocontents{lof}{\protect\addvspace{10\p@}}%
  \addtocontents{lot}{\protect\addvspace{10\p@}}%
  \@makechapterhead{#2}\@afterheading}
\def\@schapter#1{\typeout{#1}%
  \let\@secnumber\@empty
  \def\@toclevel{0}%
  \ifx\chaptername\appendixname \@tocwriteb\tocappendix{chapter}{#1}%
  \else \@tocwriteb\tocchapter{chapter}{#1}\fi
  \chaptermark{#1}%
  \addtocontents{lof}{\protect\addvspace{10\p@}}%
  \addtocontents{lot}{\protect\addvspace{10\p@}}%
  \@makeschapterhead{#1}\@afterheading}
\newcommand\chaptername{Chapter}

\def\@makechapterhead#1{\global\topskip 7.5pc\relax
  \begingroup
  \fontsize{\@xivpt}{18}\bfseries\centering
    \ifnum\c@secnumdepth>\m@ne
      \leavevmode \hskip-\leftskip
      \rlap{\vbox to\z@{\vss
          \centerline{\normalsize\mdseries
              \uppercase\@xp{\chaptername}\enspace\thechapter}
          \vskip 3pc}}\hskip\leftskip\fi
     #1\par \endgroup
  \skip@34\p@ \advance\skip@-\normalbaselineskip
  \vskip\skip@ }
\def\@makeschapterhead#1{\global\topskip 7.5pc\relax
  \begingroup
  \fontsize{\@xivpt}{18}\bfseries\centering
  #1\par \endgroup
  \skip@34\p@ \advance\skip@-\normalbaselineskip
  \vskip\skip@ }
\def\appendix{\par
  \c@chapter\z@ \c@section\z@
  \let\chaptername\appendixname
  \def\thechapter{\@Alph\c@chapter}}

\newcounter{chapter}

\newif\if@openright

\makeatother

\makeatletter 
\def\@cite#1#2{{\m@th\upshape\bfseries%
[{#1\if@tempswa{\m@th\upshape\mdseries, #2}\fi}]}}
\makeatother 

\theoremstyle{plain}
\newtheorem{thm}{Theorem}[section]
\newtheorem{cor}[thm]{Corollary}
\newtheorem{ass}[thm]{Standing Assumption}
\newtheorem{prop}[thm]{Proposition}
\newtheorem{lem}[thm]{Lemma}
\newtheorem{sublem}[thm]{Sublemma}
\theoremstyle{definition}
\newtheorem{defn}[thm]{Definition}
\newtheorem{war}[thm]{Warning}

\newtheorem{ex}[thm]{Example}

\newtheorem{quest}[thm]{Question}

\newtheorem{conj}[thm]{Conjecture}
\theoremstyle{remark}
\newtheorem{rem}[thm]{Remark}

\numberwithin{equation}{subsection}
\renewcommand{\bold}[1]{\medskip \noindent {\bf #1 }\nopagebreak}

\newcommand{\nc}{\newcommand}
\newcommand{\rnc}{\renewcommand}
\newcommand{\e}{\varepsilon}

\nc\bA{\mathbb{A}}
\nc\bB{\mathbb{B}}
\nc\bC{\mathbb{C}}
\nc\bD{\mathbb{D}}
\nc\bE{\mathbb{E}}
\nc\bF{\mathbb{F}}
\nc\bG{\mathbb{G}}
\nc\bH{\mathbb{H}}
\nc\bI{\mathbb{I}}
\nc{\bJ}{\mathbb{J}} 
\nc\bK{\mathbb{K}}
\nc\bL{\mathbb{L}}
\nc\bM{\mathbb{M}}
\nc\bN{\mathbb{N}}
\nc\bO{\mathbb{O}}
\nc\bP{\mathbb{P}}
\nc\bQ{\mathbb{Q}}
\nc\bR{\mathbb{R}}
\nc\bS{\mathbb{S}}
\nc\bT{\mathbb{T}}
\nc\bU{\mathbb{U}}
\nc\bV{\mathbb{V}}
\nc\bW{\mathbb{W}}
\nc\bY{\mathbb{Y}}
\nc\bX{\mathbb{X}}
\nc\bZ{\mathbb{Z}}
\nc\cA{\mathcal{A}}
\nc\cB{\mathcal{B}}
\nc\cC{\mathcal{C}}
\rnc\cD{\mathcal{D}}
\nc\cE{\mathcal{E}}
\nc\cF{\mathcal{F}}
\nc\cG{\mathcal{G}}
\rnc\cH{\mathcal{H}}
\nc\cI{\mathcal{I}}
\nc{\cJ}{\mathcal{J}} 
\nc\cK{\mathcal{K}}
\rnc\cL{\mathcal{L}}
\nc\cM{\mathcal{M}}
\nc\cN{\mathcal{N}}
\nc\cO{\mathcal{O}}
\nc\cP{\mathcal{P}}
\nc\cQ{\mathcal{Q}}
\rnc\cR{\mathcal{R}}
\nc\cS{\mathcal{S}}
\nc\cT{\mathcal{T}}
\nc\cU{\mathcal{U}}
\nc\cV{\mathcal{V}}
\nc\cW{\mathcal{W}}
\nc\cY{\mathcal{Y}}
\nc\cX{\mathcal{X}}
\nc\cZ{\mathcal{Z}}
\nc\bfA{\mathbf{A}}
\nc\bfB{\mathbf{B}}
\nc\bfC{\mathbf{C}}
\nc\bfD{\mathbf{D}}
\nc\bfE{\mathbf{E}}
\nc\bfF{\mathbf{F}}
\nc\bfG{\mathbf{G}}
\nc\bfH{\mathbf{H}}
\nc\bfI{\mathbf{I}}
\nc{\bfJ}{\mathbf{J}} 
\nc\bfK{\mathbf{K}}
\nc\bfL{\mathbf{L}}
\nc\bfM{\mathbf{M}}
\nc\bfN{\mathbf{N}}
\nc\bfO{\mathbf{O}}
\nc\bfP{\mathbf{P}}
\nc\bfQ{\mathbf{Q}}
\nc\bfR{\mathbf{R}}
\nc\bfS{\mathbf{S}}
\nc\bfT{\mathbf{T}}
\nc\bfU{\mathbf{U}}
\nc\bfV{\mathbf{V}}
\nc\bfW{\mathbf{W}}
\nc\bfY{\mathbf{Y}}
\nc\bfX{\mathbf{X}}
\nc\bfZ{\mathbf{Z}}

\newcommand{\bk}{{\mathbf{k}}}

\nc{\dmo}{\DeclareMathOperator}
\nc{\wt}{\widetilde}
\rnc{\Re}{\operatorname{Re}}
\rnc{\Im}{\operatorname{Im}}
\rnc{\span}{\operatorname{span}}
\dmo{\rank}{rank}
\dmo{\End}{End}
\dmo{\Hom}{Hom}
\dmo{\Jac}{Jac}
\dmo{\Id}{Id}
\dmo{\Ann}{Ann}
\dmo{\Area}{Area}
\dmo{\CP}{\bC P^1}
\dmo{\rk}{rk}
\dmo{\rel}{rel}
\dmo{\ra}{\rightarrow}
\dmo{\TwistC}{\mathrm{Twist}(\bfC, \cM)}
\dmo{\Twist}{\mathrm{Twist}}
\dmo{\Star}{Star}
\dmo{\Sch}{Sch}

\rnc{\Col}{\operatorname{Col}}
\nc{\Cold}{\operatorname{Col}^{doub}}
\nc{\ColdB}{\operatorname{Col}^{doub}_{\bfB}}

\nc{\what}{\widehat}
\nc{\whatc}[1]{\widehat{#1}^c}

\nc{\ColOne}{\Col_{\bfC_1}}
\nc{\ColOneX}{\ColOne(X,\omega)}
\nc{\ColdX}{\operatorname{Col}^{doub}\left( X, \omega \right)}

\nc{\ColTwo}{\Col_{\bfC_2}}
\nc{\ColTwoX}{\ColTwo(X,\omega)}

\nc{\ColOneTwo}{\Col_{\bfC_1, \bfC_2}}
\nc{\ColOneTwoX}{\ColOneTwo(X,\omega)}

\nc{\MOne}{\cM_{\bfC_1}}
\nc{\MTwo}{\cM_{\bfC_2}}
\nc{\MOneTwo}{\cM_{\bfC_1, \bfC_2}}

\dmo{\For}{\cF}

\nc{\GL}{\mathrm{GL}^+(2, \bR)}

\title{High rank  invariant subvarieties}
\author[Apisa]{Paul~Apisa}
\author[Wright]{Alex~Wright}
\dedicatory{Dedicated to Maryam Mirzakhani, \\ who contributed to key ideas in this project, \\ and whose vision continues to inspire us}

\begin{document}
\maketitle
\thispagestyle{empty}


\setcounter{tocdepth}{1} 
\tableofcontents



\section{Introduction}

\subsection{Main result.} Eskin, Mirzakhani, and Mohammadi showed that $GL^{+}(2, \bR)$ orbit closures of translation surfaces are properly immersed smooth sub-orbifolds, and Filip showed they are moreover algebraic varieties \cite{EM, EMM, Fi1}. Conversely, all irreducible closed $GL^{+}(2,\bR)$-invariant subvarieties of strata of translation surfaces, or invariant subvarieties for short, are $GL^{+}(2,\bR)$ orbit closures. 

Despite these strong structure theorems, as well as a great deal of older work preceding them and newer work building upon them, it remains a major open problem to classify invariant subvarieties. In this paper we give a significant portion of such a classification. 

\begin{thm}\label{T:main}
Let $\cM$ be an invariant subvariety of genus $g$ translation surfaces with $\rank(\cM)\geq \frac{g}2+1$. Then $\cM$ is either a connected component of a stratum, or the locus of all holonomy double covers of surfaces in a stratum of quadratic differentials.  
\end{thm}

%
%


Recall that the tangent space $T_{(X,\omega)}(\cM)$ to $\cM$ at a point $(X,\omega)$ is a subspace  of the relative cohomology group $H^1(X, \Sigma, \bC)$, where $\Sigma$ is the set of zeros of $\omega$. There is a natural map $p:H^1(X, \Sigma, \bC) \to H^1(X, \bC)$ from relative to absolute cohomology, and rank is defined by 
$$\rank(\cM) = \frac12 \dim p(T_{(X,\omega)}(\cM)).$$
This is a positive integer that is at most the genus $g$. We say that $\cM$ has \emph{high rank} if $\rank(\cM)\geq \frac{g}2+1$.

Rank is the most important notion of the size of an invariant subvariety. Even when an orbit closure is not explicitly known, its rank can often be bounded below either by finding cylinders \cite[Theorem 1.10]{Wcyl} or using analytic techniques \cite[Section 7]{MirWri2}. Rank also plays a central role in algebro-geometric descriptions of orbit closures \cite{Fi1}.  

In Theorem \ref{T:main} it is implicit that the surfaces do not have marked points, but a version with marked points follows immediately from known results, as in the proof of Theorem \ref{T:AAW} below.

In Corollary \ref{C:easy}, we will see that the conclusion of Theorem \ref{T:main} also holds if $\cM\subset \cH(2g-2)$ and $\rank(\cM)\geq \frac{g}2+\frac12$, which is a slight improvement in odd genus. 

\subsection{Context} 
Prior to the work of Eskin, Mirzakhani, and Mohammadi, McMullen classified orbit closures in genus 2 \cite{Mc5}; see also \cite{Mc, Ca}. More recently, orbit closures of full rank ($\rank(\cM)=g$) were classified by Mirzakhani and Wright \cite{MirWri2}; and orbit closures of rank at least 2 were classified in genus 3  primarily  by Aulicino and Nguyen \cite{NW, ANW, AN, aulicino2016rank}, and in hyperelliptic connected components of strata by Apisa \cite{Apisa2}. 

New orbit closures of rank 2 were recently constructed by Mukamel, McMullen and  Wright in \cite{MMW}, and, with Eskin, in \cite{EMMW}. 

See also \cite{Apisa:Rank1, Ygouf} for recent progress in rank 1 and \cite{EFW, BHM, LNW} for recent finiteness results.

\subsection{Hopes for a magic table.} To solve many dynamical and geometric problems about translations surfaces, including questions arising in the study of rational billiards, interval exchange transformations, and other applications, it is necessary to first know the orbit closure of the translation surface. Indeed, the orbit closure can be viewed as the arena for re-normalization, and the answers to many questions depend quantitatively on its geometry. 

Zorich wrote that his most ``optimistic hopes" for the study of translation surfaces were that a ``magic table" could be created, containing a list of orbit closures and their numerical invariants \cite[Section 9.3]{Z}.
Classification results like Theorem \ref{T:main} bring us closer to the complete realization of these hopes. 

\subsection{Mirzakhani's conjecture} 
Mirzakhani conjectured a classification of orbit closures of rank at least 2; see \cite[Remark 1.3]{ApisaWrightDiamonds} for a more detailed discussion. The new orbit closures constructed in \cite{MMW, EMMW}  disprove her conjecture, but we view Theorem \ref{T:main} as substantiating the vision behind it.

Mirzakhani participated in our early efforts to prove Theorem \ref{T:main}, and we are very grateful for the insights she contributed, especially in Sections \ref{S:degens} and \ref{S:DoubleDegen}, which we have marked as containing joint work with her. 

\subsection{Three key difficulties} 
Our approach to Theorem \ref{T:main} is inductive, with genus 2 as the base case. Using cylinder deformations, we can pass from a surface in $\cM$ to a surface  in a component $\cM'$ in the boundary of $\cM$. The first key difficulty is that it is not always possible to arrange for such codimension 1 degenerations $\cM'$ to have high rank. 

For this reason, we develop a notion of ``double degenerations". Given a collection of cylinders on a surface in $\cM$, in many cases this gives a degeneration of the surface in a codimension 2 component of the boundary of $\cM$. See Figure \ref{F:DDexample} for one of the most concrete examples when $\cM$ is a stratum.
\begin{figure}[h]\centering
\includegraphics[width=\linewidth]{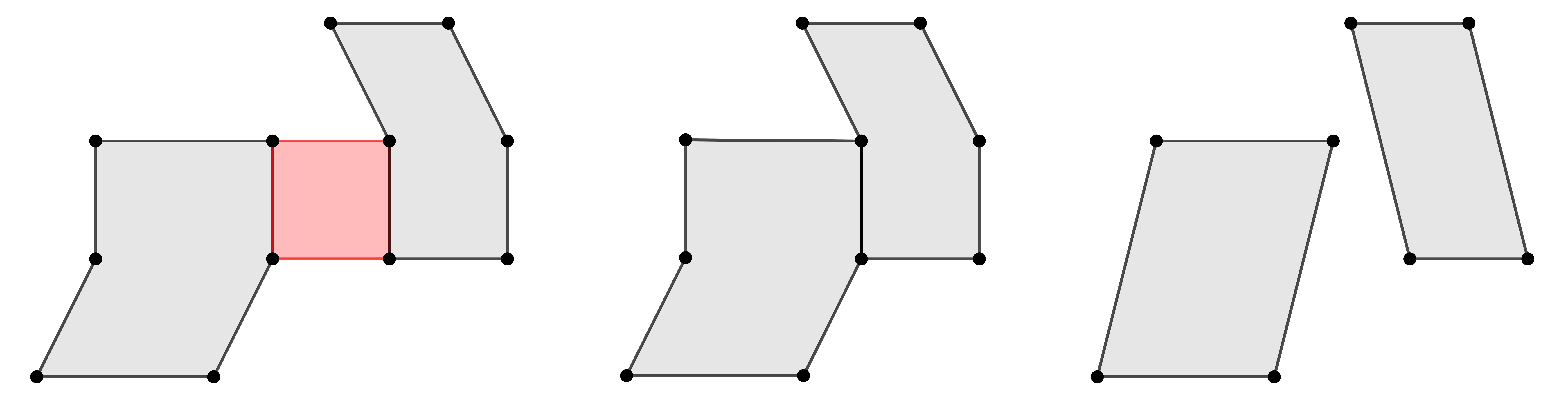}
\caption{From left to right: A surface in $\cH(4)$, with a vertical cylinder highlighted; the single degeneration of this cylinder is a surface in $\cH(1,1)$; and the double degeneration is a two component surface in $\cH(0)\times \cH(0)$. }
\label{F:DDexample}
\end{figure}
The second key difficulty is arranging for the component of the boundary containing a double degeneration to be high rank. 

Throughout, we must deal with the possibility that degenerations can disconnect the surface and create marked points.
Using a structure theorem from \cite{ChenWright}, we are able to easily show that the single degenerations we require are connected, but avoiding the possibility that double degenerations are disconnected is the third key difficulty in the paper. 

\subsection{Certificates}
Our proof that double degenerations exist is not explicit, and these degenerations themselves are surprisingly subtle and difficult to work with. So we were not able to rule out in general the possibility that the component of the boundary containing  a double degeneration satisfies the conclusion of Theorem \ref{T:main} even if $\cM$ does not. 

For this reason, we  rely on the main result of \cite{ApisaWrightGemini}, which shows that the failure of an invariant subvariety to satisfy the conclusion of Theorem \ref{T:main} must be witnessed by cylinders. Proceeding by contradiction, we consider a minimal dimensional counterexample $\cM$ to Theorem \ref{T:main}. The main result of \cite{ApisaWrightGemini} produces a collection of cylinders on a surface in $\cM$ which could not exist if $\cM$ satisfied the conclusion of Theorem \ref{T:main}. We think of these cylinders as being a certificate that $\cM$ does not satisfy the conclusion of Theorem \ref{T:main}. We are able to arrange for this certificate to persist on the double degeneration, thus showing that the associated codimension 2 boundary component of $\cM$ also does not satisfy the conclusion of Theorem \ref{T:main}. We arrange for this to contradict the fact that $\cM$ was chosen to be a minimal dimensional counterexample to Theorem \ref{T:main}.

\subsection{The Cylinder Degeneration Dichotomy} Most of the arguments in this paper are novel and broadly applicable to the classification problem, and we hope they will shed significant light on the structure of invariant subvarieties. 

We wish to highlight especially a general dichotomy for cylinder degenerations, proven in Section \ref{S:degens}, which arose from joint work with Mirzakhani. This result shows in particular that when cylinders degenerate to a collection of saddle connections, this collection is either acyclic or strongly connected when viewed as a directed graph. In Section \ref{S:DoubleDegen}, we show that in the acyclic case this graph can be contracted in a natural way, thus producing the double degeneration.

\subsection{Organization} In Section \ref{S:outline}, we reduce the proof of Theorem \ref{T:main} to two propositions. We then turn in Section \ref{S:boundary} to the boundary of invariant subvarieties, observing the ``symplectic compatibility" of the boundary and recalling the notion of primality from \cite{ChenWright}. 

Section \ref{S:CylinderDegenerations} develops the general theory of cylinder degenerations. This is followed in  Sections \ref{S:degens} and \ref{S:DoubleDegen} with the proof of the Cylinder Degeneration Dichotomy and the construction of double degenerations. 

In Section \ref{S:nested}, we prove that if a certain simple configuration of two cylinders appears on a surface in an invariant subvariety, then that subvariety satisfies the conclusion of Theorem \ref{T:main}. This allows us to avoid such simple configurations, which give rise to components of the boundary containing double degenerations that may not be high rank. 

In Sections \ref{S:hardest} and \ref{S:Goldilocks}, we produce the cylinders that will be double degenerated and then perform the double degeneration. These results are all phrased for invariant subvarieties of rank at least 2 or 3, with the high rank assumption only used to ensure prime degenerations are connected, via a short argument in Section \ref{S:Primality}.   

Sections \ref{S:Good} and \ref{S:RelMain} conclude the paper by handling the cases where only a single degeneration is required. 

\subsection{Prerequisites} 
We assume the reader is familiar with cylinder deformations and the What You See Is What You Get (WYSIWYG) partial compactification, as in \cite{Wcyl, MirWri, ChenWright}, but we do not assume familiarity with any other recent results. 

This paper is the culmination of a large project, which motivated and relies on \cite{ApisaWrightDiamonds} and \cite{ApisaWrightGemini}. In spite of this, here we have isolated the dependence on those two papers to a single statement, which appears as Theorem \ref{T:Geminal} below. So, it is not necessary or even helpful to have read those two papers.

\subsection{Next steps} 
The next step in this line of investigation would be the following. 

\begin{conj}[Almost High Rank Conjecture]
If $\cM$ is an orbit closure of surfaces of genus $g$  and 
$$\rank(\cM) = \frac{g}2+\frac12,$$
then $\cM$ is a locus of double covers of a component of a stratum of Abelian differentials, a locus of double covers of a hyperelliptic locus in a stratum of Abelian differentials, or the locus of all holonomy double covers of surfaces in a stratum of quadratic differentials. 
\end{conj}

Although there are many directions in which the classification problem can be explored, we think the following deserves special attention. 

\begin{quest}
Is every invariant subvariety of rank at least 3 trivial? More precisely, is every such subvariety a full locus of covers according to the definition given in \cite{ApisaWrightDiamonds}?
\end{quest}

Although we feel there is still insufficient evidence to upgrade this question to a conjecture, there are a number of reasons to hope for a positive answer. It would be interesting to investigate the special case of rank 3 orbit closures of dimension 6 in genus 6.

\bold{Acknowledgements.} 
We thank Maryam Mirzakhani for generously sharing her ideas, vision, and enthusiasm with us. 
We also thank Barak Weiss for helpful conversations, and Francisco  Arana-Herrera and Chris Zhang for helpful comments on an earlier draft. 

During the preparation of this paper, the first author was partially supported by NSF Postdoctoral Fellowship DMS 1803625, and the second author was partially supported by a Clay Research Fellowship,  NSF Grant DMS 1856155, and a Sloan Research Fellowship.

\section{Structure of the proof}\label{S:outline}


Here we show, using previous results, that Theorem \ref{T:main} follows from two technical results, namely Propositions \ref{P:NoRelMain} and \ref{P:RelMain} below. This gives the high level structure of our approach. The remainder of the paper will then develop the novel and broadly applicable results that we use to derive Propositions \ref{P:NoRelMain} and \ref{P:RelMain}.

We begin with the following observation. 

\begin{lem}\label{L:NoCover}
In an invariant subvariety $\cM$ of high rank, there exist Abelian differentials which are not translation covers of smaller genus Abelian differentials. 
\end{lem}

\begin{proof}
Suppose otherwise, and consider a surface with dense orbit in $\cM$ that is a translation cover. Let $g$ denote the genus of surfaces in $\cM$, and $h$ denote the genus of the codomain of the covering map. 

So $\cM$ is contained in an invariant subvariety of covers of surfaces in a genus $h$ stratum of Abelian differentials, and we see that $\cM$ has rank at most $h$. 

The Riemann-Hurwitz formula gives that 
$2-2g \leq 4-4h$, which is equivalent to $h \leq \frac{g}2 + \frac12$. This contradicts the assumption that $\cM$ has high rank.  
\end{proof}

In light of Lemma \ref{L:NoCover}, we can state the following special case of \cite[Theorem 1.1]{ApisaWrightGemini}. Recall that an orbit closure $\cM$ is said to be \emph{geminal} if for any cylinder $C$ on any $(X,\omega)\in \cM$, either 
\begin{itemize}
\item any cylinder deformation of $C$ remains in $\cM$, or 
\item there is a cylinder $C'$ such that $C$ and $C'$ are parallel and have the same height and circumference on $(X,\omega)$ as well as on all small deformations of $(X,\omega)$ in $\cM$, and any cylinder deformation that deforms $C$ and $C'$ equally remains in $\cM$.  
\end{itemize}
In the first case we say that $C$ is free, and in the second case we say that $C$ and $C'$ are twins. 

\begin{thm}\label{T:Geminal}
Any geminal invariant subvariety of high rank is a component of a stratum of Abelian differentials or the set of all holonomy double covers of surfaces in a stratum of quadratic differentials.
\end{thm}

\begin{proof}
Without the high rank assumption, \cite[Theorem 1.1]{ApisaWrightGemini} allows for two additional possibilities where $\cM$ consists of covers of lower genus Abelian differentials; but these are ruled out by Lemma \ref{L:NoCover}. 
\end{proof}

Our approach to Theorem \ref{T:main} will be to prove the following: 

\begin{thm}\label{T:AllGeminal}
Every high rank invariant subvariety consisting of surfaces without marked points is geminal.
\end{thm}

This immediately implies Theorem \ref{T:main}. 

\begin{proof}[Proof of Theorem \ref{T:main} assuming Theorem \ref{T:AllGeminal}]
Let $\cM$ be high rank. Theorem \ref{T:AllGeminal} gives that $\cM$ is geminal, and hence Theorem \ref{T:Geminal} implies that $\cM$ is a component of a stratum or the set of all holonomy double covers of surfaces in a stratum of quadratic differentials. 
\end{proof}

Theorem \ref{T:AllGeminal} in turn will follow from the next two results. We begin with the following definitions.

\begin{defn}
Let $p$ be a marked point on a translation surface $(X, \omega)$ in an invariant subvariety $\cM$. The point $p$ is said to be \emph{free} if $\cM$ contains all surfaces obtained by moving $p$ while fixing the rest of the surface.
\end{defn}

\begin{defn}
A  \emph{rel vector} is defined to be an element of the subspace $\ker(p) \cap T_{(X, \omega)}(\cM)$, where $p$ is the map from relative to absolute cohomology. The dimension of this subspace is called the \emph{rel of $\cM$}. We say $\cM$ \emph{has no rel} if the rel of $\cM$ is zero, and that it \emph{has rel} if the rel is positive. 
\end{defn}

\begin{prop}\label{P:NoRelMain}
Suppose that $\cM$ is an invariant subvariety of genus $g$ surfaces with no marked points, and assume that $\cM$ has high rank, is not geminal, and has no rel. Additionally assume $\rank(\cM)\geq 3$. 

Then the boundary of $\cM$ contains an invariant subvariety $\cM'$ that 
\begin{enumerate}
\item consists of connected surfaces,
\item has $\rank(\cM')=\rank(\cM)-1$,  
\item is not geminal,  
\item does not have free marked points, and
\item consists of surfaces of genus at most $g-2$.
\end{enumerate}
\end{prop}

\begin{prop}\label{P:RelMain}
Suppose $\cM$ is an invariant subvariety of genus $g$ surfaces without marked points, and assume $\cM$ has high rank, is not geminal, and has rel. 

Then the boundary of $\cM$ contains an invariant subvariety $\cM'$ that 
\begin{enumerate}
\item consists of connected surfaces,
\item has $\rank(\cM')=\rank(\cM)$,  
\item is not geminal, and 
\item does not have free marked points. 
\end{enumerate}
\end{prop}

We will also need the following special case of results in \cite{Apisa, ApisaWright}. If $\cM$ is an invariant subvariety of surfaces with marked points, we let $\cF(\cM)$ denote the corresponding invariant subvariety with marked points forgotten.

\begin{thm}\label{T:AAW}
If $\cM'$ is an invariant subvariety of high rank, possibly with marked points but without free marked points, and $\cF(\cM')$ is a component of a stratum of Abelian differentials or the locus of holonomy double covers of surfaces in a stratum of quadratic differentials, then $\cM'$ is geminal.

Moreover, either $\cM'=\cF(\cM')$ and $\cM'$ is a component of a stratum of Abelian differentials, or $\cM'$ is a quadratic double. 
\end{thm}

Following \cite{ApisaWrightDiamonds} and \cite{ApisaWrightGemini}, we say $\cM'$ is a quadratic double if $\cF(\cM')$ is the locus of holonomy double covers of surfaces in a stratum of quadratic differentials and the marked points in $\cM'$ occur in pairs exchanged by the holonomy involution or at fixed points of the holonomy involution, with no further constraints. We will sometimes use this terminology even when there are no marked points, since it is more concise. 

\begin{proof}
First assume $\cF(\cM')$ is a non-hyperelliptic component of a stratum of Abelian differentials. In this case, \cite{Apisa} implies, with no additional assumptions, that all marked points are free; since here there are no free marked points, we get that $\cM'=\cF(\cM')$ is a component of a stratum and hence geminal. 

Next assume $\cF(\cM')$ is a hyperelliptic component of a stratum of Abelian differentials of genus at least two. In this case, \cite{Apisa} implies, with no additional assumptions, that all marked points are free, or occur in pairs exchanged by the hyperelliptic involution, or are fixed by the hyperelliptic involution. Since here there are no free marked points, $\cM$ is a quadratic double. We also see that
every cylinder $C'$ on a surface in $\cM'$ is either fixed by the hyperelliptic involution, in which case it is free, or its image is another cylinder $C'$, in which case $C$ and $C'$ are a pair of twins. In particular, $\cM'$ is geminal.

Finally assume that $\cF(\cM')$ is a quadratic double. We first claim that the associated stratum of quadratic differentials is not hyperelliptic. Indeed, if $(X,q)$ is hyperelliptic with hyperelliptic involution $\tau$, then its holonomy double cover is a translation cover of the holonomy double cover of $(X,q)/\tau$; so the claim follows from Lemma \ref{L:NoCover}. 

Because the associated stratum of quadratic differentials is not hyperelliptic,  \cite{ApisaWright} gives that, without any additional assumptions, all marked points are free, or occur in pairs exchanged by the holonomy involution, or are fixed points for the holonomy involution. We see that $\cM'$ is geminal and a quadratic double as in the previous case. 
\end{proof}

\begin{proof}[Proof of Theorem \ref{T:AllGeminal} assuming Propositions \ref{P:NoRelMain} and \ref{P:RelMain}]
Suppose, in order to find a contradiction, that Theorem \ref{T:AllGeminal} is not true, and let $\cM$ be a  counterexample of minimal dimension.

First we claim that $\cM$ cannot have rank 2. Indeed, in that case the high rank assumption implies that the genus is 2, and the only rank 2 invariant subvarieties in genus 2 are strata: this follows from McMullen's classification in genus 2 \cite{Mc5}, and very short proofs using newer technology have also been given in \cite[Lemma 3.2]{MirWri2} and \cite[Lemma 5.14, Proposition 5.16]{Wsurvey}. So assume $\cM$ has rank at least 3. 

If $\cM$ has no rel, let $\cM'$ be the invariant subvariety given by Proposition \ref{P:NoRelMain}. Because the genus has decreased by at least two, and the rank has decreased by one, $\cM'$ is  high rank.

If $\cM$ has rel, let $\cM'$ be the invariant subvariety given by Proposition \ref{P:RelMain}. Because the rank has not changed and the genus has not increased, $\cM'$ is high rank.

Since $\cM'$ is in the boundary of $\cM$, it has smaller dimension than $\cM$. Hence $\cF(\cM')$ also has smaller dimension than $\cM$. By our minimality assumption, $\cF(\cM')$ must be geminal, and hence Theorem \ref{T:Geminal} gives that $\cF(\cM')$ is a component of a stratum of Abelian differentials or a quadratic double. 

Thus Theorem \ref{T:AAW} gives that in fact $\cM'$ is geminal, contradicting the statement of  Proposition \ref{P:NoRelMain} or Proposition \ref{P:RelMain}. 
\end{proof}

\section{The boundary of an invariant subvariety}\label{S:boundary}

This section recalls and establishes foundational results on the boundary of an invariant subvariety. We begin in Section \ref{SS:Prime} by recalling the structure of invariant subvarieties of multi-component surfaces from \cite{ChenWright}; in Sections \ref{SS:Rank} and \ref{SS:SymplecticCompatibility} we prove new  results showing that the symplectic form on first cohomology is as compatible as could be imagined with invariant subvarieties and their boundaries; and we end in Section \ref{SS:CompabilityApplications} with some preliminary applications. 

\subsection{Prime boundary components}\label{SS:Prime}
In this subsection, we will suppose that $\cM$ is an invariant subvariety in a product of strata of Abelian differentials $\cH_1 \times \cdots \times \cH_k$. Thus, we are assuming that the $k$ components of surfaces in $\cM$ are ordered or labelled, but one can apply the same discussion when the components are not ordered or labelled by lifting to a product of strata, as in \cite{ChenWright}. We start by recalling a concept introduced in \cite{ChenWright}.

\begin{defn}
$\cM$ is called $\emph{prime}$ unless, possibly after reordering the components, there is some $1 \leq s < k$, and invariant subvarieties $\cM' \subseteq \cH_1 \times \cdots \times \cH_s$ and $\cM'' \subseteq \cH_{s+1} \times \cdots \times \cH_k$ such that $\cM = \cM' \times \cM''$. 
\end{defn}

The study of invariant subvarieties of multi-component surfaces reduces to the prime case, due to the following observation from \cite[Corollary 7.10]{ChenWright}.

\begin{thm}[Chen-Wright]\label{T:PrimeDecomp}
$\cM$ can be written uniquely as a product of prime subvarieties. 
\end{thm}

The following result from \cite[Theorem 1.3]{ChenWright} strongly restricts the structure of prime invariant subvarieties, and will play a prominent role in the proof of Proposition \ref{P:NoRelMain}. Here $\pi_i$ denotes the projection onto the $i$-th factor.

\begin{thm}[Chen-Wright]\label{T:PrimeRank}
If $\cM$ is prime, then the following hold:
\begin{enumerate}
    \item\label{I:Prime:Proj} There is an invariant subvariety $\cM_i$ and a finite union $\cM_i'$ of proper invariant subvarieties of $\cM_i$ such that 
    \[ \cM_i - \cM_i' \subseteq \pi_i(\cM) \subseteq \cM_i. \]
    \item\label{I:Prime:Absolute} Locally in $\cM$, the absolute periods on any component of a multi-component surface in $\cM$ determines the absolute periods on all other components.
    \item\label{I:Prime:Rank} The rank of  $\cM_i$ is independent of $i$.
\end{enumerate}
\end{thm}

The main statement in Theorem \ref{T:PrimeRank} is the second statement; the third statement follows from the second, and the first statement is just a preliminary observation. 

\subsection{Clarifying the definition of rank}\label{SS:Rank}
In the multi-component case, rank is defined as in the single component case: given a point $(X, \omega)$ in $\cM$ with singularities $\Sigma$, it is still the case that $T_{(X, \omega)}(\cM)$ is a subset of $H^1( X, \Sigma; \mathbb{C})$. If $p$ denotes the projection to absolute cohomology, then we define
\[ \mathrm{rank}(\cM) := \frac{1}{2} \dim p( T_{(X, \omega)}(\cM) ). \]

The absolute cohomology group $H^1(X, \bR)$ has a natural symplectic form. Indeed, if the components of $X$ are $X_1, \ldots, X_k$, then $H^1(X, \bR) = \bigoplus_{i=1}^k H^1(X_i,\bR)$, and the sum of the symplectic forms on the $H^1(X_i,\bR)$ is a symplectic form on $H^1(X)$. We denote the symplectic form by $\langle \cdot, \cdot\rangle$. 

Our first goal in this section is to show $p( T_{(X, \omega)}(\cM) )$ is symplectic, generalizing
%
%
a result of Avila-Eskin-M\"oller \cite{AEM} to the multi-component case.

\begin{lem}\label{L:BoundarySymplectic}
If $(X, \omega) \in \cM$ and $\cM$ is prime, then $p( T_{(X, \omega)}(\cM) )$ is symplectic and its dimension is twice the rank of $\cM_i$ for any $i$.
\end{lem}
\begin{proof}
Consider the flat subbundle of $\cL$ of $T\cM$ whose fiber over a point $(X, \omega)$ in $\cM$ consists of elements $v$ such that $\langle p(v), p(w) \rangle = 0$ for all $w \in T_{(X, \omega)}(\cM)$. Clearly, $\cL$ contains $\ker(p)$ as a sub-bundle.

The symplectic pairing of the real and imaginary parts of $\omega$ is the area of $(X,\omega)$, and so in particular cannot be zero. Hence, the real and imaginary parts of $\omega$ are not in $\cL$, and hence $\cL$ is not all of $T\cM$.  

The statement of \cite[Theorem 7.13]{ChenWright}, which generalizes that of \cite[Theorem 5.1]{Wcyl}, gives precisely that any proper flat sub-bundle of $T\cM$ must be contained in $\ker(p)$. Hence $\cL = \ker(p)$. 

In other words, for every element $v$ of $T_{(X, \omega)}(\cM)$ that is not in the kernel of $p$, there is an element $w$ of $T_{(X, \omega)}(\cM)$ so that $\langle p(v), p(w) \rangle \ne 0$. This shows that $\langle \cdot, \cdot \rangle$ induces a nondegenerate skew-symmetric bilinear form on $p(T_{(X, \omega)}(\cM))$ and hence that $p(T_{(X, \omega)}(\cM))$ is symplectic as desired. 

The claim about dimension follows since, letting $(X_i, \omega_i)$ denote the $i$th component of $(X, \omega)$, $(\pi_i)_*: p(T_{(X, \omega)}(\cM)) \ra p(T_{(X_i, \omega_i)}(\cM_i))$ is an isomorphism by Theorem \ref{T:PrimeRank} \eqref{I:Prime:Absolute}. 
\end{proof}

It is now easy to generalize to the case when $\cM$ is not prime.

\begin{cor}\label{C:RankCharacterization}
If $(X, \omega) \in \cM$ and $\cM$ is any invariant subvariety of multi-component surfaces, then $p(T_{(X, \omega)}(\cM))$ is symplectic. 

Moreover, the rank of $\cM$ is half the dimension of a maximal dimensional symplectic subspace of $T_{(X, \omega)}(\cM)$ with respect to the form $\langle p(\cdot), p(\cdot) \rangle$. 
\end{cor}

\begin{proof}
By Theorem \ref{T:PrimeDecomp}, there are a collection of prime invariant subvarieties $\cM_1, \hdots, \cM_d$ such that $\cM = \cM_1 \times \cdots \times \cM_d$. This corresponds to a partition of $(X, \omega)$ as $(X_1, \omega_1) \sqcup \cdots \sqcup (X_d, \omega_d)$ where each $(X_i, \omega_i)$ is a union of connected components of $(X, \omega)$. We can write 
\[ p\left( T_{(X, \omega)}(\cM) \right) = p\left( T_{(X_1, \omega_1)} (\cM_1) \right) \oplus \cdots \oplus p\left( T_{(X_d, \omega_d)} (\cM_d) \right). \]
Lemma \ref{L:BoundarySymplectic} gives that each summand is symplectic, so the first claim holds. The second claim follows immediately from the first. 
\end{proof}

\subsection{Symplectic compatibility of the boundary}\label{SS:SymplecticCompatibility}

Following \cite{MirWri}, if $(X,\omega)$ is ``close" to a point $(Y,\omega_Y)$ in the boundary of the WYSIWYG compactification, we can view the tangent space $H^1(Y,\Sigma_Y)$ of the stratum of $(Y,\omega_Y)$ as a subspace of the tangent space $H^1(X,\Sigma_X)$ of the stratum of $(X,\omega)$. We now explain how the symplectic form on $H^1(X)$ is related to the symplectic form on $H^1(Y)$, showing that the induced bilinear form on $H^1(Y,\Sigma_Y)$ is the restriction of the corresponding bilinear form on $H^1(X,\Sigma_X)$. 

Let $X$ be an oriented surface. Let $f:X\to Y'$ be a map obtained as follows: First collapse a simple multi-curve $\gamma$ on $X$ to obtain a nodal surface. Then collapse a subset of the components of this nodal surface, so each component in this subset gets collapsed to a point. The result, $Y'$, is a collection of surfaces glued together at a collection of points. Thus we assume that each component of $X\setminus \gamma$ maps either to a point, or maps homeomorphically to a subset of $Y'$.

\begin{figure}[h]\centering
\includegraphics[width=0.9\linewidth]{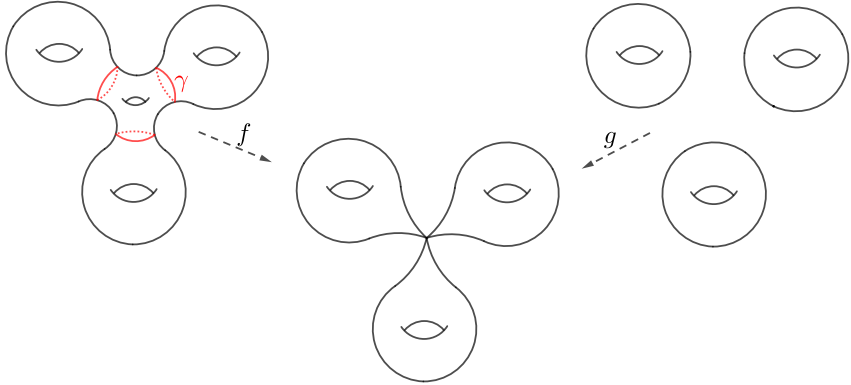}
\caption{An example of the collapse map $f$ and the gluing map $g$.}
\label{F:ColNorm}
\end{figure}

There is a possibly disconnected surface $Y$ equipped with a surjective gluing map $g:Y\to Y'$ that glues together finitely many points, as in Figure \ref{F:ColNorm}. The components of $Y$ correspond to the components of $X\setminus \gamma$ that don't map to a point under $f$. 

Now, let 
$$\Sigma_X\subset X,\quad \Sigma_{Y'}\subset Y', \quad\text{and}\quad \Sigma_Y \subset Y$$
be finite sets such that $f(\Sigma_X)=\Sigma_{Y'}$ and $\Sigma_Y=g^{-1}(\Sigma_{Y'})$, and such that $\Sigma_{Y'}$ contains $f(\gamma)$. The map 
$$g^*:H^1(Y',\Sigma_{Y'}) \to H^1(Y, \Sigma_{Y})$$ 
is an isomorphism, and the map
$$f^*\circ (g^*)^{-1}: H^1(Y,\Sigma_Y) \to H^1(X,\Sigma_X)$$
is an inclusion. So we can view $H^1(Y,\Sigma_Y)$ as a subspace of $H^1(X,\Sigma_X)$.  

If $(X, \omega)$ is a surface near a surface $(Y, \omega_Y)$ in the WYSIWYG boundary, this inclusion of  $H^1(Y,\Sigma_Y)$ into  $H^1(X,\Sigma_X)$ is the same one used \cite[Lemma 9.2]{MirWri}, which underlies the main results of \cite{MirWri}, \cite{ChenWright}. In this context we call $f$ a collapse map; it is produced in \cite[Proposition 2.4]{MirWri}, and can be seen to have the form above by factoring it through a  map to a nearby point in the Deligne-Mumford compactification.

Each component of $Y\setminus \Sigma_Y$ is homeomorphic to a subset of $X-(\gamma \cup \Sigma_X)$, and we pick the orientation on $Y$ corresponding to the orientation on $X$. Both $H^1(X,\Sigma_X)$ and $H^1(Y, \Sigma_Y)$ thus have natural bilinear forms, obtained by first mapping from relative cohomology to absolute cohomology, and then taking the symplectic pairing of the absolute cohomology classes. Working with relative rather than absolute cohomology means the bilinear form may be degenerate, but allows us to state the following. 

\begin{lem}\label{L:BilinearForm}
The bilinear form on $H^1(Y, \Sigma_Y)$ is the restriction of the bilinear form on $H^1(X,\Sigma_X)$. 
\end{lem}

\begin{proof}
There is a Poincare duality isomorphism from
$H^1(Y, \Sigma_Y)$  to $H_1(Y-\Sigma_Y)$, obtained by viewing $H^1(Y,\Sigma_Y)$ as the dual of $H_1(Y,\Sigma_Y)$, and defining an intersection number pairing between $H_1(Y,\Sigma_Y)$ and $H_1(Y-\Sigma_Y)$. This isomorphism sends the bilinear form on $H^1(Y, \Sigma_Y)$ defined above to the bilinear form on $H_1(Y-\Sigma_Y)$ defined by algebraic intersection number.

%

Similarly, $H^1(X,\Sigma_X)$ is isomorphic to $H_1(X-\Sigma_X)$.  The inclusion of $H_1(Y-\Sigma_Y)$ into $H^1(X,\Sigma_X)$ is induced by the homomorphism of $Y-\Sigma_Y$ to a subset of $X$, which does not affect the number of intersections between curves or their signs. 
\end{proof}

\subsection{Preliminary applications}\label{SS:CompabilityApplications}

In this subsection we will let $\cM'$ be a component of a boundary of an invariant subvariety $\cM$. 

\begin{cor}\label{C:RankofBoundary}
If $\cM'$ is a codimension $d$ boundary of $\cM$ then the rank of $\cM'$ is at least $\mathrm{rk}(\cM) - d$.
\end{cor}
\begin{proof}
By Corollary \ref{C:RankCharacterization}, rank is half the dimension of a maximal symplectic subspace of the tangent space. 

Pick a point $(X,\omega) \in \cM$ close to a point $(Y,\omega_Y)\in\cM'$. The main results of \cite{MirWri,ChenWright} allow us to view $T_{(Y,\omega_Y)}(\cM')$ as a subspace of $T_{(X,\omega)}(\cM)$. Lemma \ref{L:BilinearForm} gives that the bilinear form on $T_{(Y,\omega_Y)}(\cM')$ is the restriction of the bilinear form on $T_{(X,\omega)}(\cM)$.

There is a symplectic subspace of $S$ of $T_{(X,\omega)}(\cM)$ of dimension $2\rank(\cM)$. Basic linear algebra gives that $S\cap T_{(Y,\omega_Y)}(\cM')$ has dimension at least $2\rank(\cM)-d$. 
Any codimension $d$ subspace of a symplectic vector space contains a symplectic subspace of codimension $2d$. Hence $\cM'$ has rank at least $\rank(\cM)-d$.
\end{proof}

If $(X,\omega) \in \cM$ is close to a point $(Y,\omega_Y)\in \cM'$, then we can define the \emph{vanishing cycles}, using the notation of the previous subsection, as  
$$V=\ker(f_*: H_1(X,\Sigma) \to H_1(Y', \Sigma_{Y'})).$$
In this case, the annihilator $\Ann(V)\subset H^1(X,\Sigma)$ can be identified with $H^1(Y, \Sigma_{Y})$, and the main results of \cite{MirWri,ChenWright} identify $T_{(Y,\omega_Y)}(\cM')$ with $\Ann(V) \cap T_{(X,\omega)}(\cM)$.

\begin{lem}\label{L:RankReducing}
$\cM'$ has rank less than that of $\cM$ if and only if there is a vanishing cycle $v \in V$ that is nonzero as a functional on $T_{(X, \omega)}(\cM)$, but which is zero on $\ker(p)\cap T_{(X,\omega)}(\cM)$. 
\end{lem}
\begin{proof}
The characterization of rank in Corollary \ref{C:RankCharacterization} together with Lemma \ref{L:BilinearForm} give that $\rank(\cM)=\rank(\cM')$ if and only if 
$$p(\Ann(V) \cap T_{(X, \omega)}(\cM)) = p(T_{(X, \omega)}(\cM)).$$
Recall the following statement from linear algebra. 

\begin{sublem}
Let $T: W_1 \ra W_2$ be a linear map between finite dimensional vector spaces, and let $U$ be a subspace of $W_1^*$. Then $T(\Ann(U)) = T(W_1)$ if and only if every non-zero element of $U$ is non-zero on $\ker(T)$. 
\end{sublem}
%
%
%
%
%

The lemma  follows by setting $W_1 = T_{(X, \omega)}(\cM)$, $T = p$, and $U$ to be the image of $V$ in $W_1^*$.
\end{proof}

%
%
%

\section{Cylinder degenerations}\label{S:CylinderDegenerations}

\subsection{The twist space.}
In this section we will suppose that $\bfC$ is an equivalence class of generic cylinders on a surface $(X, \omega)$ in an invariant subvariety $\cM$. Recall that a cylinder is said to be \emph{$\cM$-generic}, or just \emph{generic} if $\cM$ is clear from context, if its boundary saddle connections remain parallel to the core curve of the cylinder in a neighborhood of $(X, \omega)$ in $\cM$. 

Following \cite[Definition 3.10]{ApisaWrightDiamonds}, we define cylinders to be open subsets, so that a cylinder does not include its boundary saddle connections. We can view $\bfC$ as a set of cylinders, but we also frequently view it as a subset of the surface.  

For each cylinder $C\in  \bfC$, let $\gamma_C$ denote its core curve, and let $h_C$ denote its height. Orient the $\gamma_C, C\in \bfC$, consistently, which means that their holonomies are all positive multiples of each other. By the Cylinder Deformation Theorem, the \emph{standard deformation} $$\sigma_{\bfC} = \sum_{C\in \bfC} h_C \gamma_C^*$$
belongs to $T_{(X, \omega)}(\cM)$; see \cite[Theorem 1.1]{Wcyl} for the original statement, \cite[Section 4.1]{MirWri} for the reformulation we are using here, and \cite{BDG} for a novel proof and generalization. Here  $\gamma_C^*$ denotes the cohomology class dual to $\gamma_C$, using the same duality used in Section \ref{SS:SymplecticCompatibility} and \cite[Section 4.1]{MirWri}. 

Sometimes the tangent space contains other elements of the same form. 

\begin{defn}
The twist space $\TwistC$ is the complex vector subspace of $T_{(X, \omega)}(\cM)$ consisting of vectors that can be represented as $\sum_{C\in \bfC} a_C \gamma_C^*$ with $a_C\in\bC$. 
\end{defn}

This space is so-named because $\TwistC$ is the complexification of the real subspace of  $T_{(X, \omega)}(\cM)$ tangent to deformations of $(X,\omega)$ in $\cM$ that twist or shear the cylinders of $\bfC$ while leaving the complement of $\bfC$ unchanged. 

For any $v\in T_{(X, \omega)}(\cM)$ small enough, there is a surface in $\cM$ whose periods coordinates differ from those of $(X,\omega)$ by $v$. We denote that surface by $(X,\omega)+v$. 

If $v\in \TwistC$ is small enough, then $(X,\omega)+v$ can be obtained from $(X,\omega)$ by deforming the cylinders in $\bfC$. In this case there is natural piecewise linear map 
$$ T_v: (X,\omega) \to (X,\omega)+v,$$
whose derivative is the identity off of $\bfC$, and which shears and dilates the cylinders in $\bfC$. Different cylinders in $\bfC$ may be sheared and dilated different amounts.

Recall the following result from \cite[Theorem 1.5]{MirWri}, which is also discussed in \cite[Lemma 6.10]{ApisaWrightDiamonds}. 

\begin{thm}[Mirzakhani-Wright]\label{T:CDTConverse}
$p(\TwistC)= \bC \cdot p(\sigma_\bfC)$.
\end{thm}

This theorem says that, up to purely relative cohomology classes (the kernel of $p$), the only cylinder deformations of $\bfC$ that remain in $\cM$ are multiples of the standard deformation. It can be viewed as a partial converse to the Cylinder Deformation Theorem. The following consequence, recorded in \cite[Corollary 1.6]{MirWri},  can be viewed as a partial generalization of the Veech dichotomy.   

\begin{cor}[Mirzakhani-Wright]\label{C:ConstantModuli}
If $\cM$ has no rel, then  $\cM$-parallel cylinders have rational ratios of moduli.  
\end{cor}

In particular, this implies that the ratio of moduli of $\cM$-parallel cylinders is locally constant in $\cM$.

\subsection{The definition of cylinder degenerations.} Thus, for any $v\in \TwistC$, the surface $(X,\omega)+tv$ is well defined and contained in $\cM$ for all $t>0$ sufficiently small. At the very least, this path is defined until the time $t_v$ when some cylinder in $\bfC$ reaches zero height. We assume from now on that the height of some cylinder in $\bfC$ decreases along this path, so $t_v<\infty$, and we denote by $\bfC_v$ the subset of cylinders whose heights go to zero as $t$ approaches $t_v$.  

\begin{ex}\label{E:Cv}
If $\bfC$ is horizontal and $v=\sum_{C\in \bfC} a_C \gamma_C^*$, then 
$$t_v = \min( -h_C / \Im(a_C) : \Im(a_C)<0),$$
and $C \in \bfC_v$ if and only if it realizes this minimum. 
\end{ex}

\begin{defn}
If $\bfC$ is an equivalence class of generic cylinders, and $v\in \TwistC$,  then the path $(X,\omega)+tv \in \cM, t\in [0,t_v)$ is called a \emph{collapse path} if it diverges in the stratum and if $\overline{\bfC}_v$ is not the whole surface. 
\end{defn}

The second assumption prevents the area from going to zero, and the first ensures that this path does in fact degenerate the surface. 

\begin{ex}
If $\bfC$ is horizontal and $v=-i \sigma_{\bfC} = -i \sum_{C\in \bfC} h_C \gamma_C^*$, then deforming in the $v$ direction vertically collapses all the cylinders in $\bfC$, so $\bfC_v=\bfC$. In this case $t_v=1$, and the path diverges in the stratum if and only if $\overline{\bfC}$ contains a vertical saddle connection.
\end{ex}

For each $t<t_v$, the map $$T_{tv} : (X,\omega)\to (X,\omega)+tv$$ 
has constant derivative on each cylinder in $\bfC$, and this derivative has a 1 eigenvector in the direction of the cylinders. For each cylinder in $\bfC_v$, the other eigenvalue of the derivative must be in $(0,1) \subseteq \mathbb{R}$, since the area of these cylinders decreases, and we will call the associated eigendirection the \emph{maximally contracted direction}. 

\begin{ex}\label{E:Derivative}
If $\bfC$ is horizontal and $v=\sum_{C\in \bfC} a_C \gamma_C^*$, the derivative of $T_{tv}$ on $C$ is 
$$\begin{pmatrix} 1 &  \frac{t\Re(a_C)}{h_C} \\ 0 & 1 + \frac{t\Im(a_C)}{h_C}  \end{pmatrix}.$$
\end{ex}

\begin{ex}
Figures \ref{F:MultivaluedGood} and \ref{F:MultivaluedBad} both illustrate collapse paths where $\bfC=\{C\}$ is a single horizontal cylinder, and $v=-ih_C \gamma_C^*$, so the degeneration path vertically collapses $\bfC$ and converges to a surface $\Col_v(X,\omega)$ where $\bfC$ is completely collapsed. In both examples, if we glue together the two points $z_1, z_2$ on the limit, it is possible to define a collapse map $\Col_v$. We can choose not to glue these points together, at the expense of allowing the map $\Col_v$ to be multivalued. 

\begin{figure}[h!]
\includegraphics[width=0.75\linewidth]{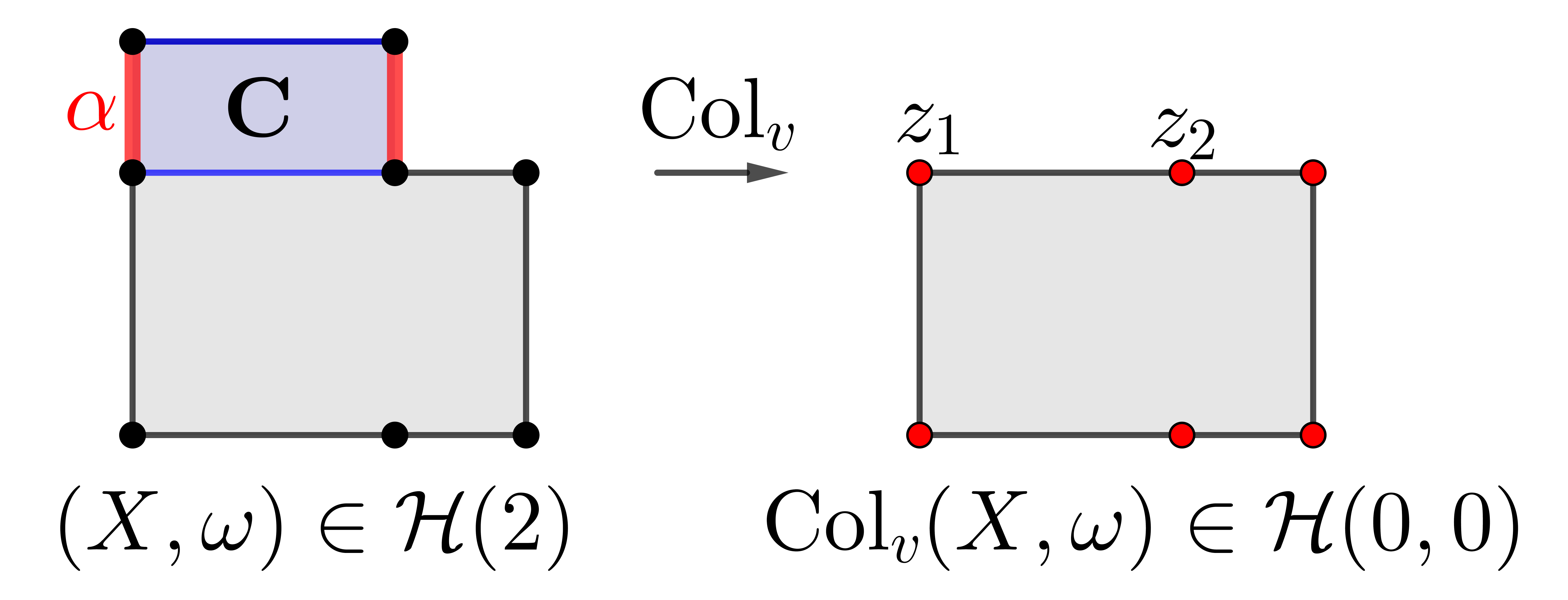}
\caption{For any point $p$ on the vertical segment $\alpha$, one may define $\Col_v(p)=\{z_1, z_2\}$.  }
\label{F:MultivaluedGood}
\end{figure}

\begin{figure}[h!]
\includegraphics[width=0.75\linewidth]{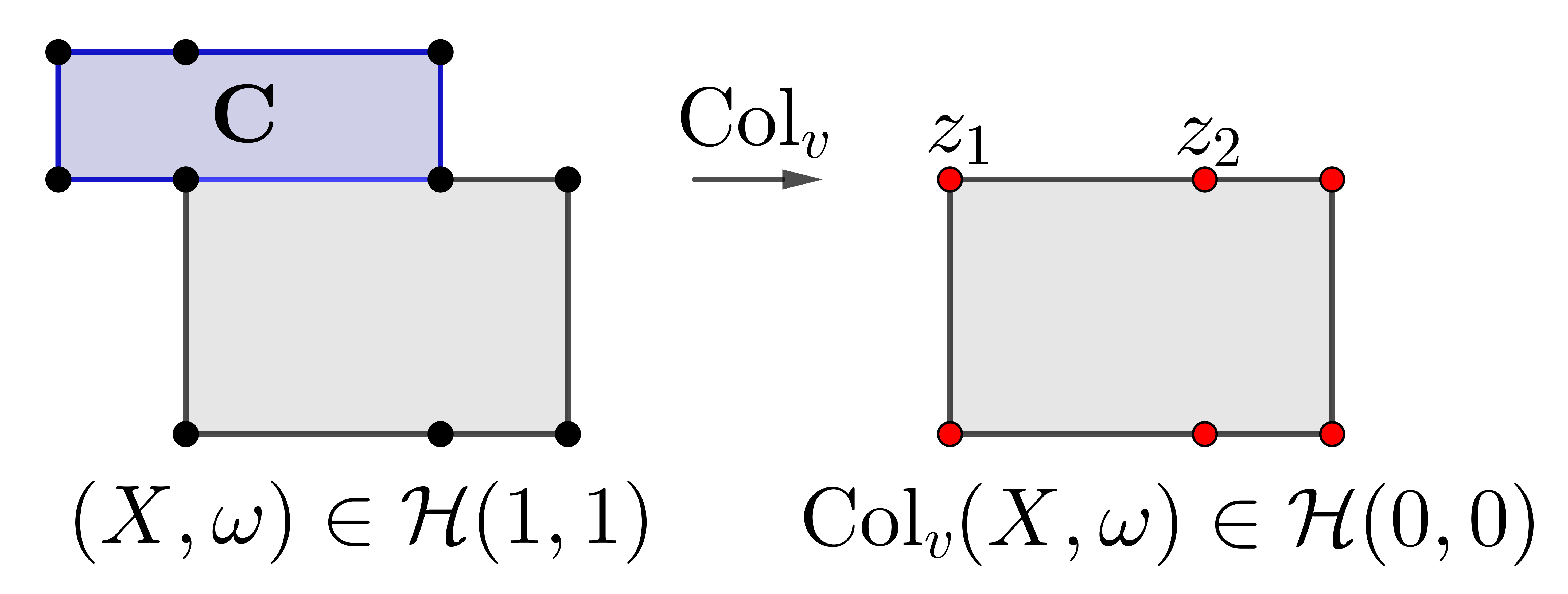}
\caption{For any point $p$ in the top left square, one may define $\Col_v(p)=\{z_1, z_2\}$.  }
\label{F:MultivaluedBad}
\end{figure}
Note that in Figure \ref{F:MultivaluedGood}, $\Col_v$ is only multivalued on single saddle connection, whereas in Figure \ref{F:MultivaluedBad}, $\Col_v$ is multivalued on an open set. 
\end{ex}

Recall from Section \ref{SS:SymplecticCompatibility} or \cite{ChenWright} that a point in the WYSIWYG partial compactification is obtained in two steps, first via a collapse and then ungluing a finite set of points. Because of the ungluing, multivalued collapse maps are typical in this context. 

In the next subsection, we verify the following.

\begin{lem}\label{L:Converge}
The collapse path converges as $t\to t_v$. The limit $\Col_v(X,\omega)$ is the image of PL map $$\Col_v:(X,\omega)\to \Col_v(X,\omega).$$
The derivative of $\Col_v$ is the identity off $\overline{\bfC}$; is constant and invertible on each cylinder of $\bfC\setminus \bfC_v$; and has kernel equal to the maximally contracted direction on each cylinder of $\bfC_v$. 

At a point $p$ where $\Col_v$ is multivalued, $\Col_v(p)$ is a finite subset of the singularities and marked points of $\Col_v(X)$.
$\Col_v$ is single-valued except possibly on a finite union of line segments. 
\end{lem}

All but the final claim in this lemma should be viewed as intuitive, and all of the claims may even be clear in any given example. However there is non-trivial content to the final claim that $\Col_v$ is only as badly multivalued as the example in Figure \ref{F:MultivaluedGood}; we will give a proof that the behaviour in Figure \ref{F:MultivaluedBad} is ruled out by the requirement that $\bfC$ consists of generic cylinders. 

\begin{defn}
The limit $\Col_v(X,\omega)$ will be called the cylinder degeneration corresponding to the collapse path $(X,\omega)+tv$. 
The component of the boundary of $\cM$ containing $\Col_v(X,\omega)$ will be denoted $\cM_v$.
\end{defn}

\subsection{The proof of Lemma \ref{L:Converge}.}  The discussion in this section, which is somewhat technical, is not used the remainder of the paper, except for one appeal to Lemma \ref{L:Loop2} in the proof of Lemma \ref{L:FreeMarkedPointAfterDegenerating}. This section can thus be skipped on a first reading of this paper. 

Throughout this subsection we will assume that $\bfC$ is an $\cM$-equivalence class of cylinders on a surface $(X,\omega)$ in an invariant subvariety $\cM$ and that $v\in \TwistC$. We will assume without loss of generality that the cylinders in $\bfC$ are horizontal. 

We begin with two general lemmas, which may be of some independent interest. The first does not require our assumption that $\bfC$ consists of generic cylinders. 

\begin{lem}\label{L:Loop1}
If $\bfC_v \ne \bfC$, then the imaginary part of the holonomy of any closed loop in $\overline{\bfC}_v$ is zero. 
\end{lem}

\begin{proof}
Let $\alpha$ be a closed loop in $\overline{\bfC}_v$. Recall that $\sigma_\bfC=\sum h_C \gamma_C^*$ denotes the standard deformation.

Without loss of generality, assume that $v$ is purely imaginary. By Theorem \ref{T:CDTConverse}, we can write $v$ as $v=c i \sigma_\bfC+r$, where $c\in \bR$, and $r\in \TwistC$ is purely relative. 

If  we write $r=\sum i r_C \gamma_C^*$, then the computation in Example \ref{E:Cv} shows that $C\in \bfC_v$ if and only if 
$$ t_v = -h_C / (r_C+ c h_C).$$

We now claim that $r_C$ is not zero when $C\in \bfC_v$. Indeed, the fact that $\bfC_v$ is neither empty nor all of $\bfC$ implies that $r$ is non-zero. Since $r$ is purely relative, this means that some $r_C$ are positive and some negative. On $(X, \omega) + tv$, the height of $C$ is $(1+tc+\frac{tr_C}{h_C}) h_C $ so we notice that the heights of cylinders with $r_C < 0$ have a smaller ratio to their original height than those of cylinders with $r_C >0$.   Since the cylinders in $\bfC_v$
reach height zero first, we have that $r_C < 0$ when $C \in \bfC_v$.


For all $C\in \bfC_v$, we compute that 
$$r_C = -\frac{c t_v+1}{t_v} h_C.$$
It is important that $r_C$ is a constant non-zero multiple of $h_C$, and that the constant $r_C/h_C$ does not depend on $C$. 

Hence, since $\alpha$ is contained in $\overline\bfC_v$, 
$$r(\alpha) = -i\frac{c t_v+1}{t_v} \sum h_C \gamma_C^*(\alpha).$$
Again since $\alpha$ is contained in $\overline\bfC_v$, the expression $\sum h_C \gamma_C^*(\alpha)$ computes the imaginary part of the holonomy of $\alpha$, and we get
$$r(\alpha) = -i\frac{c t_v+1}{t_v} \Im\left(\int_\alpha \omega\right).$$

Since $r$ is purely relative and $\alpha$ is a closed loop, by definition $r(\alpha)=0$. Hence the above equality gives $\Im(\int_\alpha \omega)=0$, as desired. 
\end{proof}

\begin{lem}\label{L:Loop2}
Assume that $\bfC$ consists of generic cylinders and that $\overline\bfC\neq (X,\omega)$. Then the imaginary part of the holonomy of any closed loop in $\overline{\bfC}$ that is disjoint from the singularities of $\omega$ is zero.  
\end{lem}

Compared to the previous lemma, the additional assumption that the cylinders in $\bfC$ are generic and that the loop avoids singularities will ensure that the loop continues to be a closed loop and continues to be contained in $\overline{\bfC}$ even after degenerations and perturbations. So the assumptions in Lemma \ref{L:Loop2} should be considered as a more robust version of the assumptions in Lemma \ref{L:Loop1}. 

\begin{proof}
Suppose not in order to deduce a contradiction. Let $\cM$ be a counterexample with smallest possible dimension. 

Since $\overline\bfC\neq (X,\omega)$, Smillie-Weiss \cite[Corollary 6]{SW2} gives the existence of a cylinder $D$ disjoint from $\bfC$. Perturbing if necessary, we can assume that $D$ is not parallel to $\bfC$. 

Let $\bfD$ be the equivalence class of $D$. A short argument using the Cylinder Deformation Theorem gives that $\bfD$ is disjoint from $\bfC$; compare for example to \cite[Propostion 3.2]{NW}.

We can now degenerate $\bfD$ using a standard deformation to obtain a surface in a smaller dimensional orbit closure $\cM'$. Since $\bfC$ and $\bfD$ are not parallel and hence not adjacent, each saddle connection on the boundary of $\bfC$ persists under this degeneration, and $\bfC$ continues to consist of generic cylinders. Similarly the assumption that $\bfC$ does not cover the whole surface continues to hold. 

Since we considered a minimal counterexample, it must be that $\bfC$ is not an $\cM'$ equivalence class; the degeneration must have caused cylinders to become generically parallel to those in $\bfC$. So $\bfC$ is actually a strict subset of an $\cM'$-equivalence class $\bfE$. In this case, taking $v=-i\sigma_\bfC$ to be the standard  deformation, which decreases heights of cylinders in $\bfC$, we have that $v\in \Twist(\bfE, \cM')$ and $\bfE_v=\bfC$. Thus Lemma \ref{L:Loop1} gives that the imaginary part of the holonomy of any relevant closed loop is in fact zero, which contradicts the assumption that $\cM$ is a counterexample to the claim.
\end{proof}

For any $C\in \bfC_v$, define a \emph{collapsed segment} of $C$ to be a maximal line segment contained in $\overline{C}$ in the direction of the maximally contracted direction, oriented in the upwards direction. We do not require the endpoints to be singularities or marked points.

\begin{lem}\label{L:OnlyOnce} 
Suppose that $(X, \omega) \ne \overline{\bfC}_v$ and that $\bfC$ consists of generic cylinders if $\bfC_v = \bfC$.  Any path in $\overline{\bfC}_v$ that is a concatenation of collapsed segments and that is disjoint from the singularities can only enter each cylinder in $\bfC_v$ at most once. 
\end{lem} 

\begin{proof}
Otherwise, we can form a closed loop $\alpha$ by starting in the middle of a cylinder $C$ in $\bfC_v$, traveling upwards along the  path until we return to $C$, and then travelling along the core curve of $C$. The imaginary part of the holonomy  of $\alpha$ is positive, contradicting Lemma \ref{L:Loop1} (when $\bfC_v \ne \bfC$) or Lemma \ref{L:Loop2} (when $\bfC_v = \bfC$). 
\end{proof}

The following immediate corollary of Lemma \ref{L:OnlyOnce} is the key to avoiding the behavior illustrated in Figure \ref{F:MultivaluedBad}. 

\begin{cor}\label{C:OnlyOnce}
Under the assumptions of Lemma \ref{L:OnlyOnce}, starting at any point in $\bfC_v$, it is possible to travel both upwards and downwards along a concatenation of collapsed segments and reach either a singularity or the boundary of $\bfC_v$.
\end{cor}

It is possible to define the surface $\Col_v(X,\omega)$ and the map $\Col_v$ without Corollary \ref{C:OnlyOnce}, but we will make use of it to simplify the discussion.

\begin{proof}[Proof of Lemma \ref{L:Converge}:]
Assuming $v$ defines a cylinder degeneration, we will define $\Col_v(X,\omega)$ in several steps, starting from $(X,\omega)$. First, we perform the part of the cylinder degeneration that does not in fact degenerate any cylinders. Namely, we consider $(X,\omega)+t_v v_{\bfC\setminus \bfC_v}$, where we define 
$$v_{\bfC\setminus \bfC_v} = \sum_{C \in \bfC \setminus \bfC_v } a_C \gamma_C^*$$
to be the part of $v=\sum a_C \gamma_C^*$ corresponding to the cylinders in $\bfC\setminus \bfC_v$. Then, we cut out the interior of $\overline\bfC_v$, to obtain the translation surface with boundary $$ ((X,\omega)+t_v v_{\bfC\setminus \bfC_v'}) \setminus \mathrm{int}(\overline{\bfC}_v).$$ 
We then identify pairs of points on the boundary if they are joined by a concatenation of collapsed segments that are disjoint from the singularities. Since we have avoided singularities, the result is a punctured surface, and we fill in the punctures to obtain $\Col_v(X,\omega)$. Any non-singular filled-in points are declared to be marked points. 

We can now define the map $\Col_v$ as follows. Off $\overline{\bfC}_v$, it is induced by the cylinder deformation 
$$T_{t_v v_{\bfC\setminus \bfC_v}}: (X,\omega) \to (X,\omega)+t_v v_{\bfC\setminus \bfC_v}.$$
For each point $p$ in $\overline\bfC_v$ not reachable from a singularity via a concatenation of collapsed segments, $\Col_v(p)$ is equal to the single point of $\Col_v(X,\omega)$ obtained by identifying the two points on the boundary of $\overline\bfC_v$ reachable from $p$ via a concatenation of collapsed segments. 

All other points $p$ are in the set where $\Col_v$ is potentially multivalued, and we can define $\Col_v(p)$ to be the set of limits of $\Col_v(p_n)$, for sequences $p_n$ of points where $\Col_v$ has previously been defined with $p_n\to p$. The construction gives that $\Col_v(p)$ is a subset of the singularities and marked points of $\Col_v(X,\omega)$. By construction $\Col_v$ is multivalued off the set reachable from singularities by travelling along collapsed segments, and Corollary \ref{C:OnlyOnce} gives that that set is a finite union of line segments. 

Since the derivative of $\Col_v$ is as described in Lemma \ref{L:Converge}, it remains only to show that the collapse path converges to $\Col_v(X,\omega)$. We sketch this now, explaining how to extend the strategy of \cite[Lemma 3.1]{MirWri} to this slightly more general situation. 

We will use the criterion for convergence from \cite[Definition 2.2]{MirWri}, which requires that we build maps 
$$g_t: \Col_v(X,\omega)\setminus U \to (X,\omega)+tv$$
that distort the flat metric very little, where $U$ is a small neighbourhood of the zeros and marked points of $\Col_v(X,\omega)$ and where $t$ is close to $t_v$. The criterion also requires that the complement of the image of $g_t$ be small in a precise sense, which, in particular, is satisfied when this complement is contained in a small neighborhood of the zeros. 

To this end, let $U$ be the union of the $L^\infty$ balls centered at the zeros and marked points on $\Col_v(X,\omega)$ of radius $\e>0$.

%
%

If we restrict the collapse path to start at time $t$, the discussion above gives PL collapse maps  
$$f_t: (X, \omega)+tv \ra \Col_v(X, \omega)$$ 
for all $t \in [0, t_v)$. Here we avoid our previous notation for collapse maps, since it does not specify the domain surface.

The map $f_t$ is an isometry when restricted to the complement of $\overline{\bfC}$. By Example \ref{E:Derivative}, the derivative of the restriction of $f_t$ to a cylinder $C \in \bfC - \bfC_v$ is $$\begin{pmatrix} 1 &  \frac{(t_v-t)\Re(a_C)}{h_C} \\ 0 & 1 + \frac{(t_v - t)\Im(a_C)}{h_C}  \end{pmatrix}.$$
Notice that this matrix is arbitrarily close to the identity when $t$ is sufficiently close to $t_v$. 


Let $U'$ be an $\epsilon$ neighborhood of $\Col_v(\bfC_v)$ in the $L^\infty$ metric. We begin by defining $g_t$ to be $f_t^{-1}$ on the complement of $U'$. Ultimately, we want to extend this function over points in $U' \backslash U$. 

The cylinders in $\bfC_v$ give rise to a finite collection $\Col_v(\bfC_v)$ of horizontal saddle connections on $\Col_v(X,\omega)$. For each such saddle connection $s$, consider a rectangle $R_s$ centered on $s$ of width $\ell_s-2 \epsilon$ and height $2\epsilon$, where $\ell_s$ is the length of $s$. Let $\gamma^{top}_s$ and $\gamma^{bot}_s$ denote the top and bottom line segments of this rectangle. See Figure \ref{F:OpenUp}.

\begin{figure}[h!]
\includegraphics[width=\linewidth]{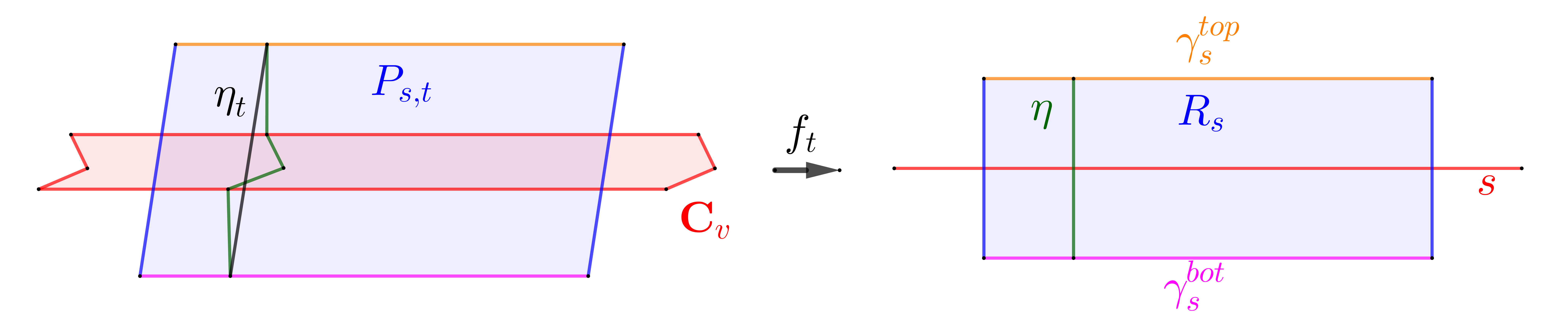}
\caption{$f_t^{-1}(\eta)$ is a sequence of line segments (shown on the left in green), which can be tightened to a straight line $\eta_t$. The red region on the left is a subset, $f_t^{-1}(s)\subset\bfC_v$.}
\label{F:OpenUp}
\end{figure}

Given a vertical line $\eta$ in $R_s$ joining a point in $\gamma^{top}_s$ to $\gamma^{bot}_s$ there is a corresponding sequence of lines $f_t^{-1}(\eta)$, which can be pulled tight, while fixing endpoints, to a geodesic representative $\eta_t$. Let $P_{s,t}$ be the union of all line segments formed this way. Since the height of the cylinders in $\bfC_v$ is proportional to $t-t_v$, which can be taken to be much smaller than $2\e$ (which is the height of $R_s$), $\eta_t$ is indeed a line segment and $P_{s,t}$ is a parallelogram. 


For each $s \in \Col_v(\bfC_v)$ extend $g_t$ to $R_s$ by defining $g_t|_{R_s}$ to be the linear map taking $R_s$ to $P_{s,t}$ that agrees with $g_t$ on $\gamma^{top}_s$ and $\gamma^{bot}_s$. It can be shown that the derivative of this linear map may be made arbitrarily close to the identity if $t$ is sufficiently close to $t_v$. 

We have now extended $g_t$ to a PL homeomorphism from $\Col_v(X, \omega) \setminus U$ to $(X, \omega) + tv$ whose derivatives, where defined, are arbitrarily close to the identity. It is also clear that, when $t$ is sufficiently close to $t_v$, the image of $g_t$ contains the complement of a $2\epsilon$ $L^\infty$ neighborhood of the zeros of $(X, \omega) + tv$.  

Since, where defined, the derivatives of $g_t$ were made arbitrarily close to the identity, it is possible to approximate $g_t$ by a smooth diffeomorphism whose derivative is also arbitrarily close to the identity and whose image also contains the complement of a $2\epsilon$ $L^\infty$ neighborhood of the zeros of $(X, \omega) + tv$. 
\end{proof}

\subsection{A slightly stronger genericity assumption}

Up until now, we have considered equivalence classes all of whose cylinders are generic. In this subsection, we introduce and discuss a slightly stronger notion of genericity for an equivalence class of cylinders. 

\begin{defn}
An equivalence class $\bfC$ of $\cM$-parallel cylinders on $(X,\omega)\in \cM$ is called \emph{$\cM$-generic}, or just \emph{generic} when $\cM$ is clear from context, if each cylinder in $\bfC$ is $\cM$-generic, and there does not exist a perturbation of $(X,\omega)$ in $\cM$ that creates new cylinders that are $\cM$-parallel to those in $\bfC$. \end{defn} 

\begin{ex}
Consider a surface $(X', \omega')$ with a pair of simple homologous cylinders $C, C'$. Let $\cM$ be the stratum containing $(X,\omega)$. Assume there are no other cylinders homologous to $C$ and $C'$.  Collapse $C'$ without degenerating the surface, to obtain a surface $(X, \omega)\in \cM$. The cylinder $C$ on $(X,\omega)$ is simple and hence $\cM$-generic. Since $(X,\omega)$ has no cylinders homologous to $C$, it is its own equivalence class. Thus $\{C\}$ is an equivalence class of generic cylinders. However, we can deform $(X, \omega)$ to make $C'$ reappear, so $\{C\}$ is not a generic equivalence class; see Figure \ref{F:GenericButNotGeneric}.
\begin{figure}[h]
\includegraphics[width=.8\linewidth]{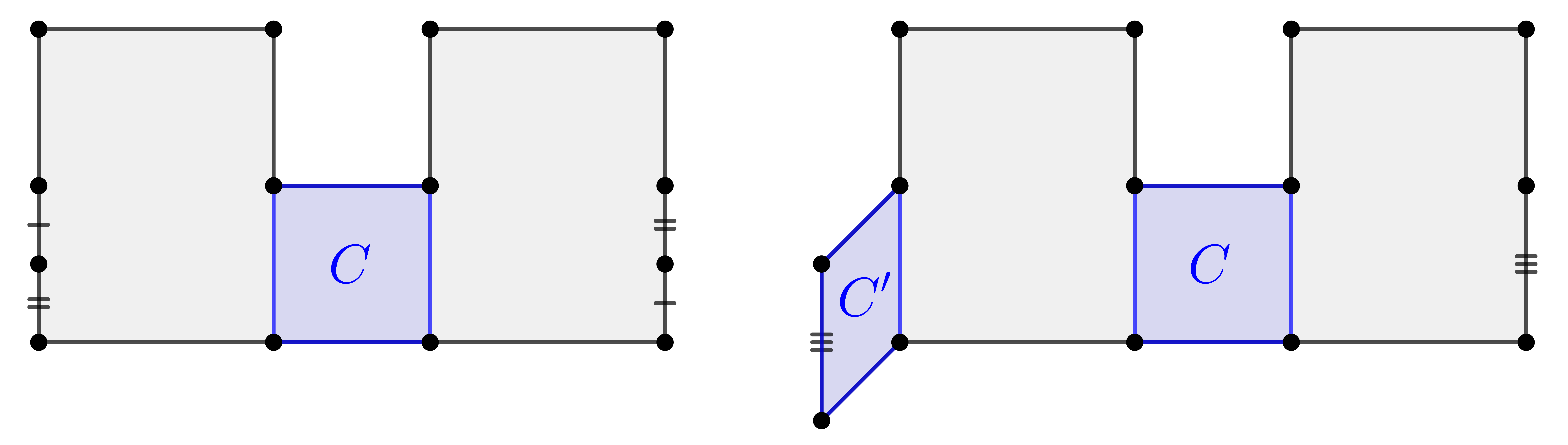}
\caption{Left: $(X,\omega)$. Right: $(X', \omega')$.}
\label{F:GenericButNotGeneric}
\end{figure}
\end{ex}

Despite the example, there are mild conditions which guarantee that an equivalence class of generic cylinders is a generic equivalence class. 

\begin{lem}\label{L:generic}
Let $\bfC$ be an equivalence class of generic cylinders on $(X,\omega)\in \cM$. Suppose either 
\begin{enumerate}
\item $\cM$ has no rel, or 
\item every saddle connection parallel to $\bfC$ is generically parallel to $\bfC$. 
\end{enumerate}
Then $\bfC$ is a generic equivalence class. 
\end{lem}

In particular, the second criterion of Lemma \ref{L:generic} implies that the set of surfaces in $\cM$ where all equivalence classes are generic is dense in $\cM$. The proof will use the following definition.

\begin{defn}
Recall that a \emph{cross curve} of a cylinder $D$ is a saddle connection contained in $\overline{D}$ that crosses the core curve exactly once. 
\end{defn}

\begin{proof}
First suppose $\cM$ has no rel. Then Corollary \ref{C:ConstantModuli} implies that $\bfC$ is generic. 

Next suppose that every saddle connection parallel to $\bfC$ is generically parallel to $\bfC$. Suppose, in order to find a contradiction, that there is a sequence of surfaces $(X_n, \omega_n)\in \cM$  that converge to $(X,\omega)$, and such that the cylinders of $\bfC$ persist on $(X_n,\omega_n)$, but where each $(X_n, \omega_n)$ has a cylinder $D_n$ that is generically parallel to $\bfC$ but not contained in $\bfC$. 

By the Cylinder Finiteness Theorem of \cite[Theorem 5.1]{MirWri} or \cite[Theorem 5.3]{ChenWright}, the circumference of the $D_n$ is uniformly bounded in $n$. After passing to a subsequence, $D_n$ converges to a collection of saddle connections on $(X,\omega)$ parallel to the core curves of $\bfC$.
%
%

If $\gamma_n$ are bounded length cross curves of $D_n$, then similarly we may assume that $\gamma_n$  converges to a collection of saddle connections parallel to $\bfC$.  Since $\gamma_n$ is not parallel to $\bfC$, at least one of these saddle connections is not generically parallel to $\bfC$. 

This contradicts the assumption, so the sequence must not exist, and we conclude that $\bfC$ is generic. 
\end{proof}

\subsection{The twist space decomposition}

\begin{lem}\label{L:SCinRelHomology}
In relative homology, the difference between any two cross curves of a cylinder $C$ that cross the cylinder in the same direction can be written as a linear combination of the boundary saddle connections of $C$. 

Let $\bfC$ be any collection of parallel cylinders on a translation surface. In relative homology, every saddle connection in $\overline{\bfC}$ can be written as a linear combination of boundary saddle connections and cross curves.
\end{lem}

\begin{proof}
The first statement follows from the definition of relative homology. 

Every saddle connection in $\overline{\bfC}$ is homotopic to a concatenation of cross curves of cylinders in $\bfC$ and boundary saddle connections of cylinders in $\bfC$.
\end{proof}

\begin{prop}\label{P:TwistDecomp} 
Let $\bfC$ be a generic equivalence class of cylinders on $(X,\omega)\in \cM$.

Fix an element $w \in T_{(X, \omega)}(\cM)$  that pairs non-trivially with the core curve of a cylinder in $\bfC$. Then any $\eta \in T_{(X, \omega)}(\cM)$ admits a unique decomposition 
$$\eta = aw + \eta_\bfC + \eta_{(X,\omega)\setminus \bfC}$$
where $a \in \mathbb{C}$, $\eta_\bfC \in \TwistC$, and $\eta_{(X,\omega)\setminus \bfC}\in T_{(X, \omega)}(\cM)$ evaluates to zero on every saddle connection in $\overline{\bfC}$. For any $C\in \bfC$, we have $a=\frac{\eta(\gamma_C)}{w(\gamma_C)}$.
\end{prop}
Recall that $\gamma_C$ denotes the core curve of $C$. 
\begin{proof}
With the specified value of $a$, the class $\eta - a w \in T_{(X, \omega)}(\cM)$  evaluates to zero on $\gamma_C$. Since every saddle connection in the boundary of $\bfC$ is generically parallel to $\gamma_C$, $\eta - a w$ evaluates to zero on the core curves of $\bfC$. For simplicity of notation, we will replace $\eta$ with $\eta - aw$ so that we may suppose that $\eta$ evaluates to zero on the core curves of $\bfC$. 

For each $C\in \bfC$, fix a cross curve $s_C$ of $C$, oriented so $\gamma_C^*(s_C)=1$.
Define $$\eta_{\bfC} = \sum_{C\in \bfC} \eta(s_C) \gamma_C^*.$$

\begin{lem}\label{L:Zero}
$\eta - \eta_\bfC$ is zero on each saddle connection in $\overline\bfC$.
\end{lem}
\begin{proof}
We have assumed that $\eta$ evaluates to zero on every saddle connection on the boundary of a cylinder of $\bfC$. By construction, the same holds for $\eta_\bfC$. Hence, $\eta - \eta_\bfC$ is zero on every such boundary saddle connection of a cylinder in $\bfC$ as well as on every cross curve in $\{ s_C \}_{C\in \bfC}$. The result now follows by Lemma \ref{L:SCinRelHomology}.
\end{proof}

\begin{lem}\label{L:InTM}
$\eta_\bfC\in T_{(X, \omega)}(\cM)$.
\end{lem}
\begin{proof}
Assume that the cylinders in $\bfC$ are horizontal. We will show that both the real and imaginary parts of $\eta_\bfC$ are in the tangent space. 

Since $\eta \in T_{(X, \omega)}(\cM)$ and since $\cM$ is defined by real linear equations, $\mathrm{Re}(\eta)$ and $\mathrm{Im}(\eta)$ both belong to $T_{(X, \omega)}(\cM)$. Therefore, for sufficiently small $\epsilon$, 
\[ (X', \omega') = (X, \omega) + i \epsilon\Re(\eta) \]
belongs to $\cM$. The standard deformation on $(X', \omega')$ is 
\[ \sigma_\bfC'=\sum_{C\in \bfC} \left(h_C + \e \mathrm{Re}(\eta)(s_C) \right)\gamma_C^*,\]
where $h_C$ continues to denote the height of $C$ on $(X,\omega)$. Here, crucially, we use that $\bfC$ is a generic equivalence class, rather than just an equivalence class of generic cylinders, to ensure that $\bfC$ is still an equivalence class on the deformation $(X', \omega')$. 

By the Cylinder Deformation Theorem, the standard deformation of $\bfC$ on $(X', \omega')$ belongs to $T_{(X', \omega')}(\cM)$. By parallel transport, we can also consider this vector as an element of $T_{(X, \omega)}(\cM)$. The standard deformation at $(X,\omega)$, namely $\sigma_\bfC = \sum_{C\in \bfC} h_C \gamma_C^*$, also belongs to $T_{(X, \omega)}(\cM)$. Hence 
\[\sigma_\bfC'-\sigma_\bfC =  \e\sum_{C\in \bfC} \mathrm{Re}(\eta)(s_C) \gamma_C^*\in T_{(X, \omega)}(\cM).\]
This shows that $\mathrm{Re}(\eta_{\bfC}) \in T_{(X, \omega)}(\cM)$. A similar argument shows that $\mathrm{Im}(\eta_{\bfC}) \in T_{(X, \omega)}(\cM)$ and hence that $\eta_{\bfC} \in T_{(X, \omega)}(\cM)$.
\end{proof}

Proposition \ref{P:TwistDecomp} now follows from Lemmas \ref{L:Zero} and \ref{L:InTM}; uniqueness is left as an exercise. 
\end{proof}

We now establish the following version of the twist space decomposition using $p( T_{(X,\omega)}(\cM))$ instead of $T_{(X,\omega)}(\cM)$. Recall that $T_{\Col_v(X, \omega)}(\cM_v)$ can be naturally identified with   $\Ann(V) \cap T_{(X, \omega)}(\cM)$, where $V$ is the space of vanishing cycles. It will also be convenient to use the following definition. 

\begin{defn}
A cylinder degeneration specified by $v\in \TwistC$ will be called rank-reducing if $\rank(\cM_v)<\rank(\cM)$, and rank-preserving if $\rank(\cM_v)=\rank(\cM)$.
\end{defn}

\begin{prop}\label{P:AbsoluteTwistDecomp}
Let $\bfC$ be a generic equivalence class of cylinders on $(X,\omega)\in \cM$ and let $v \in \TwistC$ define a rank-reducing cylinder degeneration. Then 
$$p(T_{(X,\omega)}(\cM))=p(\Ann(V) \cap T_{(X, \omega)}(\cM))\oplus p(\bC\cdot \sigma_\bfC).$$
\end{prop}

Before we give the proof, we note the following. 

\begin{lem}\label{L:VinCv}
Let $\bfC$ be an equivalence class on a surface in an invariant subvariety $\cM$. 
For any cylinder degeneration determined by $v\in \TwistC$, there is a basis of $V$ consisting of $1$-chains supported in $\overline\bfC_v$. 
\end{lem}

\begin{proof}
Letting $\Col_v$ be the PL collapse map defined in Lemma \ref{L:Converge}, $V$ is the kernel of $(\Col_v)_*$ on $H_1(X, \Sigma; \mathbb{C})$ where $\Sigma$ denotes the singularities of $\omega$. The claim now follows from the fact that for any simplicial map that is surjective on 2-chains, the kernel of first homology is generated by 1-chains that map to zero. 
%
%
%
%
%
%
%
\end{proof}

\begin{proof}[Proof of Proposition \ref{P:AbsoluteTwistDecomp}]
Consider $u\in T_{(X,\omega)}(\cM)$.

Let $\eta$ denote the cohomology class of the Abelian differential on $\Col_v(X, \omega)$, viewed as an element of $\Ann(V) \cap T_{(X, \omega)}(\cM)$ via the identification of this space and $T_{\Col_v(X,\omega)}(\cM_v)$.

We see that $\eta$ is nonzero on core curves of $\bfC$, since the Abelian differential on $\Col_v(X, \omega)$ is nonzero on nonzero, non-negative linear combinations of consistently oriented parallel saddle connections in $\Col_v(\bfC_v)$, and every core curve of a cylinder in $\bfC_v$ degenerates to such a sum. 

Using Proposition \ref{P:TwistDecomp}, write $$u = a\eta + u_\bfC + u_{(X,\omega)\setminus\bfC}.$$
Since $u_{(X,\omega)\setminus\bfC}$ is zero on all saddle connections in $\overline\bfC$, it is contained in $\Ann(V)$ by Lemma \ref{L:VinCv}. 

By Theorem \ref{T:CDTConverse}, there is a constant $c$ such that $p(u_\bfC)= c p(\sigma_\bfC)$, so we can write $$p(u) =   p(a\eta + u_{(X,\omega)\setminus\bfC})+ c p(\sigma_\bfC).$$
Since both $\eta$ and $u_{(X,\omega)\setminus\bfC}$ are in $\Ann(V) \cap T_{(X, \omega)}(\cM)$, the first summand is in $p(\Ann(V) \cap T_{(X, \omega)}(\cM))$.

Since this holds for all $u$, this shows that 
$$p(T_{(X,\omega)}(\cM))=p(\Ann(V) \cap T_{(X, \omega)}(\cM))+ p(\bC\cdot \sigma_\bfC).$$
Since $p(\bC\cdot \sigma_\bfC)$ is one dimensional, to show that this is a direct sum decomposition it suffices to show that $p(T_{(X,\omega)}(\cM))\neq p(\Ann(V) \cap T_{(X, \omega)}(\cM))$. That follows from Corollary \ref{C:RankCharacterization} and Lemma \ref{L:BilinearForm}, since the degeneration is rank-reducing. 
\end{proof}

\begin{cor}\label{C:RankMinus1}
Under the same assumptions as Proposition \ref{P:AbsoluteTwistDecomp}, if $\mathrm{rank}( \cM_v ) < \mathrm{rank}( \cM )$, then $\mathrm{rank}( \cM_v ) = \mathrm{rank}( \cM )-1$.
\end{cor}
\begin{proof}
This follows from Proposition \ref{P:AbsoluteTwistDecomp}, Lemma \ref{L:BoundarySymplectic}, and Corollary \ref{C:RankCharacterization}.  
\end{proof}

\section{A cylinder degeneration dichotomy (with Mirzakhani)}\label{S:degens}

Any collection of parallel saddle connections on a translation surface can be considered as a directed graph in two ways, by directing each of the saddle connections so they all point in the same direction. The two directed graphs obtained in this way differ by changing the direction of each edge. It will not matter in the subsequent analysis which of the two directed graphs are chosen. We will consider these graphs both as abstract graphs and as subsets of the surface. 
 
In this section, we will consider degenerations in which some cylinders become a collection of parallel saddle connections. We will show that the resulting directed graph of saddle connections is either strongly connected or satisfies a very restrictive property that, in particular, implies that it is acyclic. Recall that a directed graph is \emph{acyclic} if it contains no directed cycles, and it is \emph{strongly connected} if every edge is contained in a directed cycle, so this is a dichotomy between opposite extremes. 

\subsection{Statement} 
We will call a  collection of parallel saddle connections $\Gamma$ on a translation surface $(X',\omega')$ in an invariant subvariety $\cM'$ \emph{$\cM'$-rel-scalable} (or simply \emph{rel-scalable} when $\cM'$ is clear from context) if there is purely relative cohomology class $s\in T_{(X',\omega')}(\cM')$ which evaluates on each edge of $\Gamma$ to give the holonomy of the edge. In that case we say that $s$ \emph{certifies that $\Gamma$ is rel-scalable}. 

By definition, a cohomology class is purely relative if it evaluates to zero on any absolute homology class. Hence, if $\Gamma$ is rel-scalable, then the integral of $\omega'$ along every loop in $\Gamma$ is zero. This implies that $\Gamma$ is acyclic as a directed graph.

Fix an equivalence class of cylinders $\bfC$ on a translation surface $(X, \omega)$ in an invariant subvariety $\cM$. Pick a direction $v \in \TwistC$ in which to deform the cylinders  in $\bfC$ in order to degenerate the surface. We will continue to use the notation introduced in the previous section: 
 $\bfC_v\subset \bfC$ denotes the subset of the cylinders that reach zero height first along the degeneration path; $\Col_v(X,\omega)$ denotes the limit surface; and $\cM_v$ is the boundary invariant subvariety that contains it.

\begin{thm}[The Cylinder Degeneration Dichotomy]\label{T:GraphsFull}
Let $\bfC$ be a generic equivalence class of cylinders on a translation surface $(X, \omega)$ in an invariant subvariety $\cM$. Consider a cylinder degeneration given by $v \in \TwistC$.  
\begin{itemize}
\item If $\mathrm{rank}(\cM_v) < \mathrm{rank}(\cM)$, then $\bfC_v = \bfC$ and $\Col_v(\bfC)$ is $\cM_v$-rel-scalable. 
\item If $\mathrm{rank}(\cM_v) = \mathrm{rank}(\cM)$, then $\Col_v(\bfC_v)$ is strongly connected. 
\end{itemize}
%
%
\end{thm}

The following special case of Theorem \ref{T:GraphsFull} was initially discovered and proved jointly with Mirzakhani.

\begin{cor}\label{C:ExactGraphsRel0}
Let $\bfC$ be an equivalence class of $\cM$-generic cylinders on a translation surface $(X, \omega)$ in an invariant subvariety $\cM$. Consider a cylinder degeneration given by $v \in \TwistC$.  If $\cM$ has no rel, then $\bfC_v = \bfC$ and $\Col_v(\bfC)$ is $\cM_v$-rel-scalable. 
\end{cor}

\begin{proof}
If $\cM$ has no rel, then $\mathrm{rank}(\cM_v) < \mathrm{rank}(\cM)$, and Lemma \ref{L:generic} gives that every equivalence class of generic cylinders is a generic equivalence class. 
\end{proof}

Theorem \ref{T:GraphsFull}  evolved from the ``no rel" special case (Corollary \ref{C:ExactGraphsRel0}) discovered with Mirzakhani, so we feel it is appropriate to attribute Theorem \ref{T:GraphsFull} jointly to her. In the course of this evolution, our perspective has changed significantly, so the arguments we present will have a  different flavor than those we discussed with Mirzakhani.

\subsection{Proof of Theorem \ref{T:GraphsFull}} 

 Let $V$ denote the collection of vanishing cycles for the collapse from $(X, \omega)$ to $\Col_v(X, \omega)$. The first statement (on rel-scalability) of Theorem \ref{T:GraphsFull} will be derived from the following. The reader should keep in mind that relative cohomology is the dual of relative homology.

\begin{prop}\label{P:SameImage}
If $\mathrm{rank}(\cM_v) < \mathrm{rank}(\cM)$, then the restriction of the rel subspace of $T_{\Col_v(X, \omega)}(\cM_v)$ to the subspace of relative homology generated by saddle connections in $\Col_v(\bfC_v)$ is equal to the restriction of $T_{\Col_v(X, \omega)}(\cM_v)$ to this subspace. Moreover, $\bfC = \bfC_v$. 
\end{prop}

Thus, any change to $\Col_v(\bfC_v)$ that can be accomplished locally in $\cM_v$ can also be obtained with a rel deformation in $\cM_v$. 

Before we give the proof, we explain how to apply Proposition \ref{P:SameImage}. 

\begin{proof}[Proof of Theorem \ref{T:GraphsFull} in the rank-reducing case]
Let $\omega_0$ denote the holomorphic one-form that induces the flat structure on $\Col_v(X, \omega)$.

By Proposition \ref{P:SameImage}, there is a purely relative class $r\in T_{\Col_v(X, \omega)}(\cM_v)$ whose restriction to  the subspace of relative homology generated by saddle connections in $\Col_v(\bfC_v)$ is the same as the restriction of $[\omega_0]\in T_{\Col_v(X, \omega)}(\cM_v)$. Since $\omega_0$ evaluated on any saddle connection gives its holonomy, this means that $r$ evaluated on any saddle connection in $\Gamma$ gives its holonomy. The existence of such an $r$ is the definition of rel-scalability, so this, together with the fact that $\bfC = \bfC_v$ proves the result. 
\end{proof}

\begin{proof}[Proof of Proposition \ref{P:SameImage}]
Let $w\in T_{\Col_v(X, \omega)}(\cM_v)$. We wish to find a purely relative class $r\in T_{\Col_v(X, \omega)}(\cM_v)$ which is equal to $w$ on saddle connections in $\Col_v(\bfC_v)$. We start with finding some relative class $r_0$, which can be thought of as an initial guess for $r$, and then we will show how to modify $r_0$ to get a class $r$ with the desired restriction. 

\begin{lem}\label{L:NewRel}
There is a purely relative class $r_0\in T_{\Col_v(X, \omega)}(\cM_v)$ that, when viewed using the inclusion $T_{\Col_v(X, \omega)}(\cM_v)\subset T_{(X,\omega)}(\cM)$, is non-zero on core curves of cylinders in $\bfC$. Moreover, $\bfC_v = \bfC$.
\end{lem} 

Since core curves are absolute cohomology classes, $r_0$ is not purely relative in $T_{(X,\omega)}(\cM)$. So one might say that $r_0$ becomes purely relative in the passage from $\cM$ to $\cM_v$. 

\begin{proof}
Recall that $\sigma_{\bfC}$ denotes the standard deformation in $\bfC$. 

\begin{sublem}\label{SL:PairWithCoreCurve}
For any $w\in T_{(X,\omega)}(\cM)$, if $\langle \sigma_\bfC, w\rangle \neq 0$, then $w(\gamma)\neq 0$ for each core curve $\gamma$ of $\bfC$. 
\end{sublem}
\begin{proof}[Proof of Sublemma.]
Given any vector space $V$ with a non-degenerate bilinear form $\langle \cdot, \cdot \rangle$, the dual vector space $V^*$ is endowed with a dual bilinear form. The fundamental property is that $$v\mapsto \phi_v=\langle \cdot, v\rangle$$ sends the bilinear form to the dual bilinear form.  We have that $\langle \phi_v, \psi_w \rangle = \langle v, w \rangle = \psi_w(v)$. 

In our situation, we apply this to $V=p(T_{(X,\omega)}(\cM))$ with its symplectic form.
The standard deformation $\sigma_\bfC$ is dual to a linear combination $\sum h_i \gamma_i$ of the core curves of $\bfC$. From $\langle w, \sigma_\bfC\rangle\neq 0$ we then get that $w$ is non-zero on $\sum h_i \gamma_i$. Since all the $\gamma_i$ give the collinear functionals on $T_{(X, \omega)}(\cM)$, it is non-zero on all of them.
\end{proof}

By Proposition \ref{P:AbsoluteTwistDecomp},  since $\mathrm{rank}(\cM_v) < \mathrm{rank}(\cM)$, 
\[ p(T_{(X,\omega)}(\cM))=p(\Ann(V) \cap T_{(X, \omega)}(\cM))\oplus p(\bC\cdot \sigma_\bfC). \]
In particular, there is an element $r_0$ of $\Ann(V) \cap T_{(X, \omega)}(\cM)$ so that 
\begin{enumerate}
\item $\langle r_0, \sigma_{\bfC} \rangle \ne 0$, and 
\item $r_0$ pairs trivially with every element of $\Ann(V) \cap T_{(X, \omega)}(\cM)$.
\end{enumerate}
Using the identification of $T_{\Col_v(X, \omega)}(\cM_v)$ with $\Ann(V) \cap T_{(X, \omega)}(\cM)$, the second property shows that $r_0$ is rel in $T_{\Col_v(X, \omega)}(\cM_v)$. Together with Sublemma \ref{SL:PairWithCoreCurve}, the first shows that $r_0$, when viewed as an element of $T_{(X, \omega)}(\cM)$, pairs with the core curve of every cylinder in $\bfC$ nontrivially.

Finally, if $\bfC \ne \bfC_v$, then there is a cylinder $C$ in $\bfC$ that remains a cylinder on $\Col_v(X, \omega)$. Since $r_0$ pairs nontrivially with the core curve of every cylinder in $\bfC$, the same is true for $\Col_v(C)$. But this contradicts the fact that $r_0$ is rel and hence must evaluate to zero on any closed curve. This concludes the proof of Lemma \ref{L:NewRel}. 
\end{proof}

Since the $r_0$ supplied by Lemma \ref{L:NewRel} is non-zero on core curves of $\bfC$, there is constant $c$ such that $w-cr_0$ is zero on core curves of $\bfC$. (All the core curves are generically parallel, so being zero on one core curve implies being zero on all of them.) 

For notational convenience, define $\eta = w-cr_0$. By Proposition \ref{P:TwistDecomp}, we can decompose  $\eta$ as
\[ \eta = \eta_\bfC + \eta_{(X, \omega) \backslash \bfC}, \]
where $\eta_{\bfC} \in \TwistC$ and $\eta_{(X, \omega) \backslash {\bfC}}$ pairs trivially with any saddle connection contained in $\overline{\bfC}$. 

By definition, $\eta$ is in $T_{\Col_v(X, \omega)}(\cM_v)$, which we identify with $\Ann(V) \cap T_{(X, \omega)}(\cM)$. By Lemma \ref{L:VinCv}, $\eta_{(X, \omega) \backslash {\bfC}}$ is also in $\Ann(V)$, and hence so is $\eta_\bfC=\eta-\eta_{(X, \omega) \backslash {\bfC}}$.

\begin{lem}
$\eta_\bfC$ is purely relative in $T_{(X,\omega)}(\cM)$. 
\end{lem}
\begin{proof}
Since $\eta_\bfC \in \TwistC$, we use Theorem \ref{T:CDTConverse} to write it as 
$$\eta_\bfC= a \sigma_\bfC + \tau,$$ 
where $a\in \bC$ and $\tau$ is a purely relative element of $\TwistC$. We must show that $a=0$. 

The result now follows from Proposition \ref{P:AbsoluteTwistDecomp}, since $p(\eta_\bfC) = ap(\sigma_\bfC)$ and $\eta_\bfC \in \Ann(V) \cap T_{(X, \omega)}(\cM)$.
\end{proof}

We can now conclude. We have that $r_0$ is purely relative, and 
$$w = (c r_0 + \eta_\bfC) + \eta_{(X,\omega)\backslash\bfC},$$ 
where the first summand is purely relative and the second is zero on all saddle connection in $\overline{\bfC}$, and both summands are in $\Ann(V)$. Since all saddle connections in $\Col_v(\bfC)$ arise from saddle connections in $\overline{\bfC}$, and since $\eta_{(X,\omega)\backslash\bfC}$ is zero on saddle connections in $\overline\bfC$, we get that $w$ and $c r_0 + \eta_\bfC$ are equal on $\Col_v(\bfC)$, concluding the proof of Proposition \ref{P:SameImage}. 
\end{proof}

It remains only to consider the rank-preserving case.

\begin{proof}[Proof of Theorem \ref{T:GraphsFull} in the rank-preserving case]
For concreteness, assume that $\bfC$ is horizontal, and direct $\Col_v(\bfC_v)$ so the edges point in the positive real direction. 

We will verify that $\Col_v(\bfC_v)$ is strongly connected using the following criterion: A directed graph is strongly connected if and only if the edges can be assigned positive weights such that for each vertex the sum of the incoming edge weights is equal to the sum of the outgoing edge weights.

%
%

To each edge of $\Col_v(\bfC_v)$, we assign a positive weight as follows. Pick any point $p$ on the edge, and define the weight to be the imaginary part of the holonomy of $\Col_v^{-1}(p)$. More informally, this can be described as the sum of the heights of the cylinders degenerating to this edge. It does not depend on the choice of the point $p$.  

Let $q$ be any singularity in $\Col_v(\bfC_v)$ and let $\alpha_q$ denote the clockwise oriented boundary of any small embedded flat disk centered at $q$ that contains no other singularities or marked points of $\Col_v(X, \omega)$, as in Figure \ref{F:alphaq}. The following two lemmas complete the proof. 

\begin{figure}[h]
\includegraphics[width=0.75\linewidth]{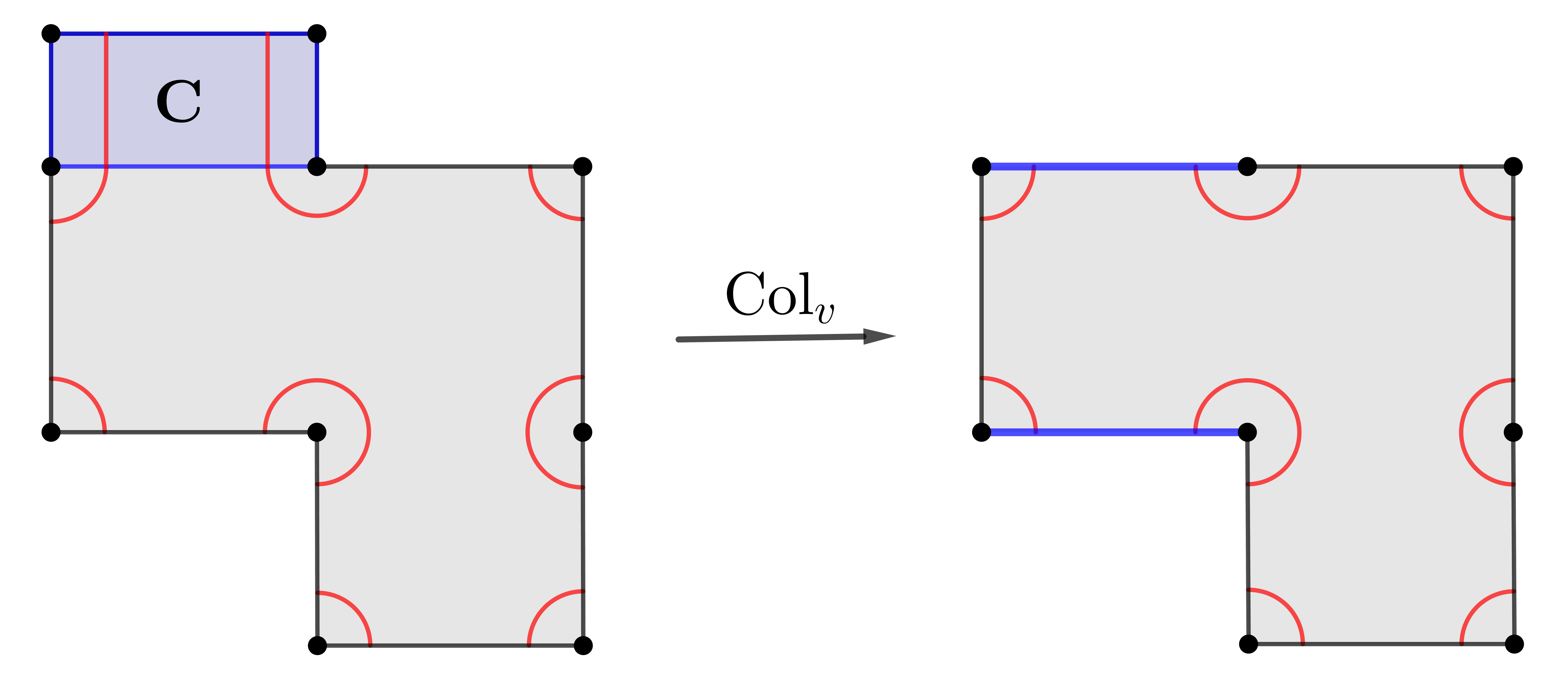}
\caption{An example $\Col_v^{-1}(\alpha_q)$ (left) and $\alpha_q$ (right).}
\label{F:alphaq}
\end{figure}

\begin{lem}\label{L:InMinusOut}
The imaginary part of the holonomy of the loop $\Col_v^{-1}(\alpha_q)$ is equal to sum of the weights of the edges leaving $q$ minus the sum of the weights of the edges entering $q$. 
\end{lem}
\begin{proof}
Let $P_q$ (resp. $N_q$) denote the set of points of positive (resp. negative) intersection of $\alpha_q$ with the saddle connections in $\Col_v(\bfC_v)$. The imaginary part of the holonomy of $\Col_v^{-1}(\alpha_q)$ is 
\[ \Im\left( \sum_{p \in P_q} \int_{\Col_v^{-1}(p)} \omega - \sum_{p \in N_q} \int_{\Col_v^{-1}(p)}\omega\right),\]
as can be seen by using smaller and smaller embedded discs to define $\alpha_q$. 

This quantity is precisely the sum of the incoming edge weights minus the sum of the outgoing edge weights for $q$.
\end{proof}

\begin{lem}\label{L:HolZero}
The holonomy of the loop $\Col_v^{-1}(\alpha_q)$ is zero.
\end{lem}

\begin{proof}
Notice that $\Col_v^{-1}(\alpha_q)$ is a vanishing cycle. Since $\Col_v^{-1}(\alpha_q)$ is a closed loop, any class in $\ker(p) \cap T_{(X, \omega)}(\cM)$ evaluates to zero on it. 

Since $\rank(\cM)=\rank(\cM_v)$, Lemma \ref{L:BilinearForm} and Corollary \ref{C:RankCharacterization} imply that 
$$p\left( T_{(X, \omega)}(\cM) \right) = p\left( \Ann(V) \cap T_{(X, \omega)}(\cM) \right),$$
where $V$ continues to denote the vanishing cycles. This shows that every element of $T_{(X, \omega)}(\cM)$ evaluates to zero on $\Col_v^{-1}(\alpha_q)$. Since the cohomology class of $\omega$  is in $ T_{(X,\omega)} (\cM)$, we get the result. 
\end{proof}

At each vertex $q$ of the graph $\Col_v(\bfC_v)$, Lemmas \ref{L:InMinusOut} and \ref{L:HolZero} together show that the incoming weights equal the outgoing weights, establishing the criterion for being strongly connected. 
\end{proof}

\begin{rem}
In the rank-preserving case, the assumption that $\bfC$ is generic is only used to ensure that the PL map $\Col_v$ is defined (by Lemma \ref{L:Converge}). In fact that only requires that $\bfC$ consists of generic cylinders, and likely even that could be weakened. 
\end{rem}

\section{Double degenerations (with Mirzakhani)}\label{S:DoubleDegen}

Using Theorem \ref{T:GraphsFull}, when $\rank(\cM_v)<\rank(\cM)$ we will prove the existence of a degeneration of $\Col_v(X,\omega)$ that contracts the graph $\Col_v(\bfC)$. When $\cM$ has no rel, this degeneration will be canonically defined, but even in this case we remark in Warning \ref{W:DD} that it is more subtle than may be expected. 

The degeneration will involve deforming in the direction of a vector that verifies rel-scalability. The end result, given in Definition \ref{D:DD}, will be called a \emph{double degeneration}, keeping in mind that it results from the two step process of first degenerating $(X,\omega)$ to obtain $\Col_v(X,\omega)$ and then degenerating $\Col_v(X,\omega)$ to contract the graph $\Col_v(\bfC)$.

The main ideas in this section were initially discovered in the special case when $\cM$ has no rel jointly with Mirzakhani. 

\subsection{Rel}\label{SS:Rel}

In this subsection and the next, we start by recalling some well-known material; compare to  \cite[Section 6]{BSW}, \cite[Section 4.2]{KZ}, \cite[Section 8.1]{EMZboundary} and \cite{McM:nav, McM:iso, MW2, Wol}. 

Let $(X,\omega)$ be a translation surface and let $\Sigma=\{z_1, \ldots, z_s\}$ be the set of zeros of $\omega$. Consider a vector $\lambda = (\lambda_1, \ldots, \lambda_s)\in \bC^s$. 

For each $i$, we can consider the set of directions at $z_i$ that point in the same direction as $\lambda_i$; if the cone angle is $2\pi k_i$ then there are $k_i$ such directions. For example, if $\lambda_i$ is positive and real, these are the east pointing directions at $z_i$. 

Define $\Star_i(\lambda)\subset (X,\omega)$ to be the union of the $k_i$ line segments of length $|\lambda_i|$ leaving $z_i$ in the direction of $\lambda_i$.   We will assume that all these line segments are embedded. We will moreover assume that the different  $\Star_i(\lambda)$ are disjoint, and define
$$\Star(\lambda)=\cup_i\Star_i(\lambda).$$

Under these assumptions, we now define a surface $\Sch_\lambda(X,\omega)$,  which lies in the same stratum as $(X,\omega)$. The notation is chosen to emphasize the similarity to classical Schiffer variations. 
The surface  $\Sch_\lambda(X,\omega)$ is constructed from $(X,\omega)$ via the following two step process illustrated in Figure \ref{F:RelReglue}. 
\begin{enumerate}
\item First cut each $\Star_i(\lambda)$. For each $i$, we thus obtain a degenerate  $2k_i$-gon, all of whose sides have length $|\lambda_i|$. 
\item Then, for each $i$, re-glue the collection edges thus created so that each is glued to the adjacent edge that it wasn't previously glued to. 
\end{enumerate}
\begin{figure}[h]
\includegraphics[width=0.35\linewidth]{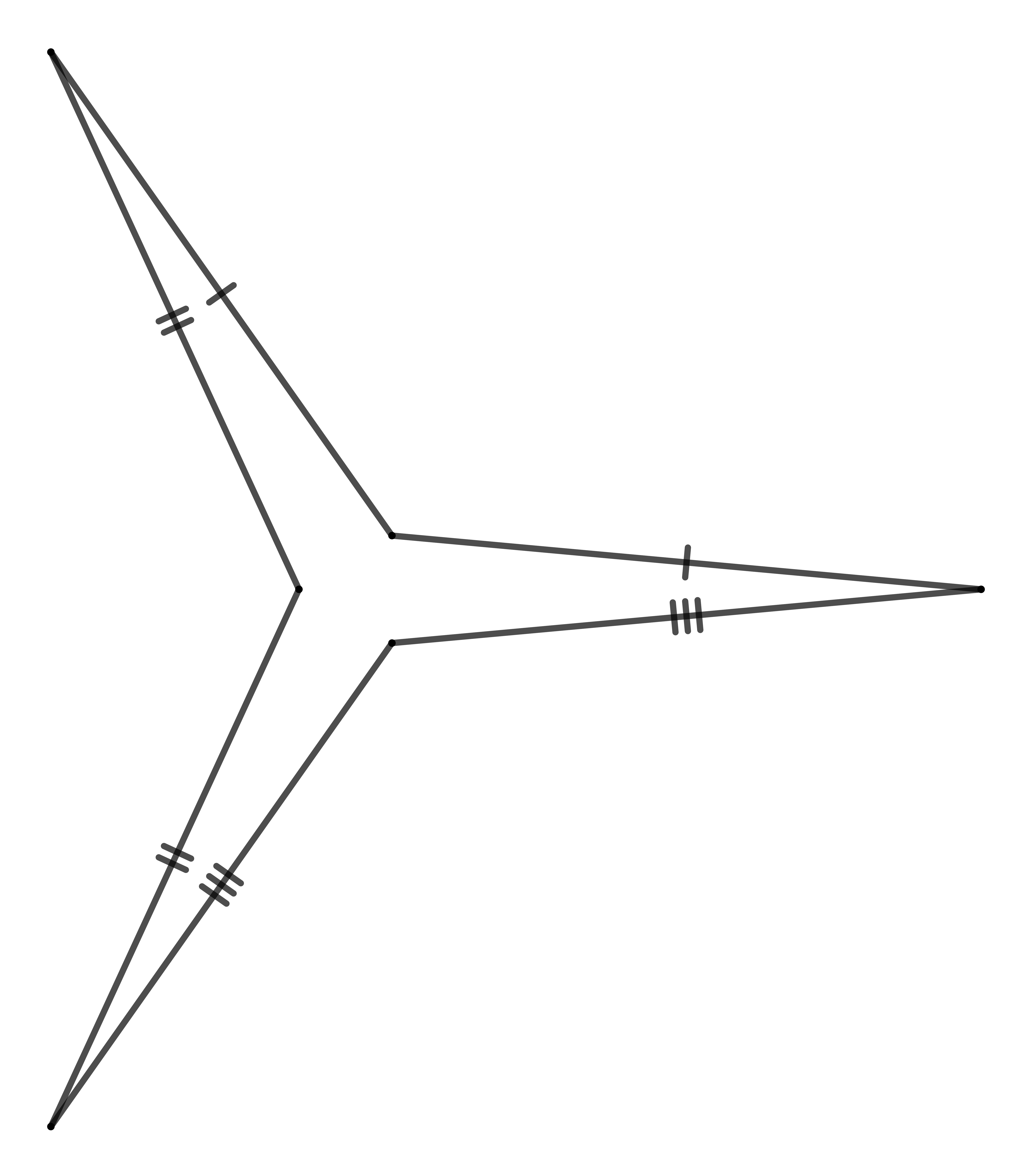}
\caption{A local surgery at a cone point of angle $6\pi$. }
\label{F:RelReglue}
\end{figure}
The cut and re-gluing surgeries corresponding to different $i$ are independent, and can be done simultaneously or sequentially. 

The relevance of this construction is given by the following observation. 

\begin{lem}
Along the path $\Sch_{\tau\lambda}(X,\omega), \tau\in [0,1]$, the absolute periods are constant, and the derivative of a cycle joining $z_i$ to $z_j$ is constant and equal to $\lambda_j-\lambda_i$.  
\end{lem}

%
%



\begin{cor}
For $\xi\in \ker(p)$, there is a unique $\lambda \in \bC^s$ with $\sum \lambda_i=0$ such that for $\tau$ sufficiently small, the surface $(X,\omega)+\tau\xi$ is equal to $\Sch_{\tau\lambda}(X,\omega)$. 

If $\xi$ is real, then so is $\lambda$. 
\end{cor}

\subsection{Real rel}\label{SS:RealRel}

When $\xi\in \ker(p)\cap H^1(X, \Sigma, \bR)$ is real, the construction above is somewhat nicer, because all the segments in $\Star(\lambda)$ are horizontal and hence parallel to each other. 

We consider the path $(X,\omega)-\tau\xi, \tau\geq 0$. (The minus sign will prove convenient in our application.) 
Either $(X,\omega)-\tau\xi$ is defined for all $\tau>0$, or it is defined for $\tau$ less than some constant $\tau_\xi$ such that a horizontal saddle connection on $(X,\omega)$ reaches zero length as $\tau\to \tau_\xi$. 

To this standard discussion, we add the following observation. 

\begin{lem}\label{L:RelLimit}
If $(X,\omega)-\tau\xi$ is only defined until time  $\tau_\xi<\infty$, then the path $(X,\omega)-\tau\xi$ converges as $\tau\to \tau_\xi$ in the WYSIWYG partial compactification. 

The limit can be obtained from $(X,\omega)$ by a cutting and regluing surgery.  Conversely,  $(X,\omega)-\tau\xi$ can be obtained from the limit by making cuts at the zeros of size proportional to $\tau-\tau_\xi$, and regluing these cuts in some pattern. 

The space $V$ of vanishing cycles is spanned by the horizontal saddle connections that reach zero length as $\tau\to \tau_\xi$, which are exactly the saddle connections $s$ where $\omega(s) = \tau_\xi \cdot \xi(s)$. 
\end{lem}

\begin{proof}
$(X,\omega)-\tau\xi$ is obtained by cutting a collection of line segments emanating from the zeros, as described above. The length of the segments is proportional to $\tau$. 

Define a translation surface $(X', \omega')$ by the same process, with cuts of lengths proportional to $\tau=\tau_\xi$. The $\tau=\tau_\xi$ case differs from the $\tau<\tau_\xi$ case only in that the endpoints of two different line segments being cut may coincide, or an endpoint of one of these line segments may now be a zero. 

This $(X', \omega')$ is the candidate limit for the path. By definition, it is obtained by local surgeries, and conversely $(X, \omega)$ can be recovered from $(X', \omega')$ by reversing these surgeries, cutting all the line segments produced in the re-gluing and then identifying them in their original gluing to obtain $(X,\omega)$. 

Similarly, $(X', \omega')$ can be obtained from $(X,\omega)-\tau\xi$ via cutting and re-gluing, where the cuts have size proportional to $\tau-\tau_\xi$. Conversely, we can obtain $(X,\omega)-\tau\xi$ from $(X', \omega')$ by cutting and re-gluing a collection of segments of size proportional to $\tau-\tau_\xi$, all of which start at the zeros and marked points of $(X', \omega')$. 

We thus see that there is an isometric map from the complement of a small neighborhood of the zeros on $(X', \omega')$ to $(X,\omega)-\tau\xi$, verifying the criterion for convergence from \cite[Definition 2.2]{MirWri}. 

For any relative homology class $s$, its holonomy on $(X,\omega)-\tau\xi$ is by definition $\omega(s)-\tau\xi(s)$. Thus, the horizontal saddle connections that reach zero length are exactly those where $\omega(s) = \tau_\xi \cdot \xi(s)$. That there are no other vanishing cycles follows as in the proof of Lemma \ref{L:VinCv}. 
\end{proof}

\subsection{Assumptions}\label{SS:AssForDD}
We will require some assumptions to define the double degeneration, which are motivated by the following two lemmas.

\begin{lem}\label{L:ScaleCert}
If the assumptions of Theorem \ref{T:GraphsFull} hold and $\rank(\cM_v)<\rank(\cM)$, then
\begin{enumerate}
\item if $\cM$ has no rel, then the rel vector in $T_{\Col_v(X, \omega)}(\cM_v)$ that certifies that $\Col_v(\bfC)$ is rel-scalable is unique, and
\item regardless of whether $\cM$ has rel, this vector can be chosen to be a complex multiple of a real vector. 
\end{enumerate}
\end{lem}

\begin{proof}
We begin by proving the first statement. Corollary \ref{C:RankMinus1} implies that $\mathrm{rank}(\cM_v) = \mathrm{rank}(\cM) - 1$. Since $\Col_v(\bfC)$ is rel-scalable, $\cM_v$ has at least one dimension of rel. This implies that, in the case that $\cM$ has no rel, $\cM_v$ has exactly one dimension of rel, since $\dim(\cM_v) < \dim(\cM)$. In particular, the vector certifying rel-scalability must be unique up to scaling. 

To prove the second statement, it suffices to show that the vector can be chosen to be real when $\bfC$ is horizontal. When $\bfC$ is horizontal, the real part of a vector certifying rel-scalability again certifies rel-scalability.
\end{proof}

We will now establish the following, which is a slight variant of \cite[Corollary 3.23]{ApisaWrightDiamonds}.

\begin{lem}\label{L:GenericallyParallel}
Suppose that $\bfC$ is an equivalence class  of generic cylinders on a surface $(X, \omega)$ in an invariant subvariety $\cM$. Let $v \in \TwistC$ define a cylinder degeneration. 

Suppose that the ratio of heights of any two cylinders in $\bfC_v$ is constant on a neighborhood of $(X, \omega)$ in $\cM$. Then all the saddle connections in $\Col_v(\bfC_v)$ are generically parallel on $\cM_v$. Moreover, $\Ann(V) \cap T_{(X, \omega)}(\cM)$ is codimension one in $T_{(X, \omega)}(\cM)$.
\end{lem}
\begin{rem}
By Corollary \ref{C:ConstantModuli}, if $\cM$ has no rel then it is automatically the case that the ratio of heights of any two cylinders in $\bfC$ is constant in any neighborhood where those cylinders persist. 
\end{rem}
\begin{proof}
%
Pick an arbitrary cylinder in $\bfC_v$. Let $\gamma$ denote its core curve, and let $s$ denote one of its cross curves. 

By Lemmas \ref{L:SCinRelHomology} and \ref{L:VinCv}, every vanishing cycle can be written as a linear combination of cross curves and core curves of cylinders in $\bfC_v$. Since the heights of any two cylinders in $\bfC_v$ have a constant ratio in a neighborhood of $(X, \omega)$, it follows that the holonomy of any such chain is a linear combination of the holonomies of $\gamma$ and $s$. Since $\gamma$ is not a vanishing cycle, there is some constant $c\in \bR$ so that $s' := c\gamma + s$ has generically the same holonomy as a vanishing cycle.

For any $1$-chain, $\alpha$ in $\overline{\bfC}_v$, the holonomy of $\alpha$ can be written as $a\gamma + bs'$ for some real constants $a$ and $b$. Since $s'$ is a vanishing cycle, the holonomy of $\Col_v(\alpha)$ can be written as $a$ times the holonomy of $\Col_v(\gamma)$. This shows that every saddle connection contained in $\Col_v(\bfC_v)$ is generically parallel to $\Col_v(\gamma)$. 

If $\alpha$ is a vanishing cycle, $\Col_v(\alpha)$ has zero holonomy, and so $a=0$. This shows that all vanishing cycles induce collinear functionals on $T_{(X,\omega)}(\cM)$, giving the codimension one statement. 
\end{proof}

\subsection{The definition}

Motivated  by the previous two lemmas, we make the following assumption for the rest of this section.

\begin{ass}\label{A:DD-Assumptions}
Suppose that $\Col_v(\bfC_v)$ is rel-scalable and that all the saddle connections parallel to $\Col_v(\bfC_v)$ are generically parallel to each other. Let $\eta$ denote a rel vector in $T_{\Col_v(X, \omega)}(\cM_v)$ that certifies that $\Col_v(\bfC_v)$ is rel-scalable and that is a complex multiple of a real vector. 
\end{ass}

\begin{rem}\label{R:DD-Assumptions}
When $\cM$ has no rel and all saddle connections parallel to $\bfC$ are generically parallel, Assumption \ref{A:DD-Assumptions} holds. The ``generically parallel" condition can always be obtained by an arbitrarily small perturbation.
\end{rem}

\begin{lem}\label{L:RelFlow}
The path $\Col_v(X, \omega) - t\eta$ is defined and remains in $\cM_v$ for $t \in [0, 1)$. 
It converges as $t\to 1$, and the space $V$ of vanishing cycles is spanned by the saddle connections parallel to $\Col_v(\bfC)$. 
\end{lem}
\begin{proof}
Rotating the surface if required, we can assume that $\bfC$ is horizontal and $\eta$ is real. The result then follows from Lemma \ref{L:RelLimit}, noting that $t_\eta=1$ since $\eta(s)=\omega(s)$ for all horizontal saddle connections. 
\end{proof} 

\begin{defn}\label{D:DD}
Let $\Col_v^{doub}(X, \omega)$ denote the limit as $t$ approaches $t_\eta = 1$ of $\Col_v(X, \omega) - t\eta$. Let $\cM_v^{doub}$ denote the component of the boundary containing $\Col_v^{doub}(X, \omega)$. 
\end{defn}

Note that these definitions depend on the choice of $\eta$. In this paper, they will be exclusively used in the case when $\cM$ has no rel, where Lemma \ref{L:ScaleCert} gives that $\eta$ is unique, so the dependence on $\eta$ is not reflected in the notation. 

\begin{war}\label{W:DD}
It is dangerously incomplete to describe the double degeneration simply by saying it contracts the graph $\Col_v(\bfC)$. First of all, there might be saddle connections generically parallel to this graph but not contained in this graph, and they get contracted as well. But more importantly, our proof of the rel-scalability of $\Col_v(\bfC)$ does not give an explicit description of a cohomology class $\eta$ certifying the rel-scalability, even when this $\eta$ is unique. So our understanding of the double degeneration is not very explicit. We have not ruled out the possibility that the local surgeries required to obtain $(X,\omega)-t\eta$ might require surgeries at zeros not on saddle connections generically parallel to $\Col_v(\bfC)$, so  singularities not on the graph $\Col_v(\bfC_v)$ might a priori ``move around" along the path $\Col_v(X, \omega) - t\eta$. 
\end{war}

\subsection{Basic results}
Keeping in mind that double degenerations are only defined under Assumption \ref{A:DD-Assumptions}, we  compare $\cM_v^{doub}$ to $\cM_v$ and $\cM$. 

\begin{lem}\label{L:DoubleBasic1}
$\dim \cM_v^{doub} = \dim\cM_v-1$ and $\rank\cM_v^{doub} = \rank\cM_v$. 
\end{lem}

\begin{proof}
By Lemma \ref{L:RelLimit}, the vanishing cycles are spanned by saddle connections parallel to $\Col_v(\bfC)$, which are all generically parallel by Assumption \ref{A:DD-Assumptions}. Hence $\dim \cM_v^{doub} = \dim\cM_v-1$. The fact that $\Col_v(\bfC)$ is rel-scalable, together with Lemma \ref{L:RankReducing}, gives that $\rank\cM_v^{doub} = \rank\cM_v$. 
\end{proof}

\begin{cor}\label{C:L:DDRelZero}
$\rank(\cM_v^{doub})=\rank(\cM)-1$. If $\rel(\cM)=0$ then $\rel(\cM_v^{doub})=0$. 
\end{cor}

\begin{proof}
The first statement follows from Lemma \ref{L:DoubleBasic1} and Corollary \ref{C:RankMinus1}. The second statement follows from the first, since $\cM_v^{doub}$ has dimension at least two less than $\cM$, and dimension is twice rank plus rel.  
\end{proof}

%

\begin{lem}\label{L:DefinitionOfLC}
Let $L_\bfC \subset H_1(X,\Sigma)$ be the span of all saddle connections contained in $\overline\bfC$ as well as all saddle connections parallel to $\bfC$. Then 
$$T_{\Col_v^{doub}(X, \omega)} (\cM_v^{doub}) = T_{(X,\omega)} (\cM) \cap \Ann(L_\bfC).$$
Moreover, the genus of $\Col_v^{doub}(X,\omega)$ is at least $d$ less than that of $(X,\omega)$, where $d$ is the dimension of the span of the core curves of $\bfC$ in homology. 
\end{lem}

\begin{proof}
The first statement follows because $L_\bfC$ is  the preimage under $(\Col_v)_*$ of the span (in relative homology) of the saddle connections parallel to $\Col_v(\bfC)$, which span the vanishing cycles for the degeneration from $\Col_v(X, \omega)$ to $\Col_v^{doub}(X, \omega)$. 

The second statement follows because all the core curves of cylinders are in the kernel of the induced map on homology for the composition of the collapse maps that map $(X, \omega)$ to $\Col_v^{doub}(X, \omega)$.  Alternatively, the second statement can be derived from the fact that $\Ann(L_\bfC)$ is the tangent space to the stratum of $\Col_v^{doub}(X,\omega)$, the fact that genus is the rank of the stratum,  Lemma \ref{L:BoundarySymplectic}, and Corollary \ref{C:RankCharacterization}.
\end{proof}

\section{Classification using a nested free cylinder}\label{S:nested}

In this section we develop some results of independent interest. 


\begin{defn}
Suppose that $C$ is a cylinder on a surface $(X, \omega)$ in an invariant subvariety $\cM$. 
We say that $C$ is \emph{nested} in another cylinder $H$ if $C\subset \overline{H}$ and $C$ crosses $H$ exactly once. 
\end{defn}

Nested cylinders are automatically simple. If $C$ is nested in $H$, then there is a saddle connection which appears in both the top and bottom boundaries of $H$ and is  a cross curve for $C$.  The property of being nested can typically be destroyed by perturbation, so when we discuss nested cylinders we will typically be considering ``non-generic" surfaces.  

\begin{thm}\label{T:nested}
Let $\cM$ be an invariant subvariety with no rel. If  a surface $(X,\omega)\in \cM$ has a nested free cylinder then $\cM$ is a stratum of Abelian differentials or a quadratic double.
\end{thm}  

Recall that ``free" is defined before Theorem \ref{T:Geminal}. After proving Theorem \ref{T:nested}, we will derive from it Proposition \ref{P:FindNested}. These results are  important ingredients in the next section. 

Proposition \ref{P:FindNested} also has the following surprisingly strong consequence. 

\begin{cor}\label{C:easy}
Suppose that $\cM$ has no rel and is contained in a genus $g$ stratum with $s$ zeros. Then if 
$$\rank(\cM)>\frac{g+s-1}2$$
then $\cM$ is either a component of a stratum or a quadratic double.
\end{cor}

The rank bound in Corollary \ref{C:easy} is strictly weaker than that in Theorem \ref{T:main} when $s>2$, equivalent when $s=2$, and slightly stronger when $s=1$. The $g=3$, $s=1$ case recovers the classification of rank two invariant subvarieties of $\cH(4)$, established in \cite{NW,ANW}. We prove Corollary \ref{C:easy} at the end of this section.

\subsection{The proof}
The arguments in this section use only the Cylinder Deformation Theorem and its partial converse Theorem \ref{C:ConstantModuli}. 

%
%

\begin{proof}[Proof of Theorem \ref{T:nested}:] 
Assume $\cM$ has a surface with a nested free cylinder.

\begin{lem}\label{L:ColOntoH}
$\cM$ contains a horizontally periodic surface $(X,\omega)$ with only a single horizontal cylinder $H$, such that $H$ contains a vertical free cylinder $V$, and such that any pair of horizontal saddle connections on $(X,\omega)$ that have the same length continue to have the same length in a neighborhood of $(X, \omega)$ in $\cM$.
\end{lem}
\begin{proof}
Start with any surface with a cylinder $V$ nested in a cylinder $H$. Without loss of generality, assume $H$ is horizontal and the surface doesn't contain any vertical saddle connections. Note that $H$ must be free, since deformations of $V$ can be used to change the circumference of $H$ without changing the circumference of any parallel cylinders. 

We now describe how to ``collapse the complement of $H$ onto $H$", via a procedure illustrated in Figure \ref{F:ColOntoH}.
\begin{figure}[h]
\includegraphics[width=.85\linewidth]{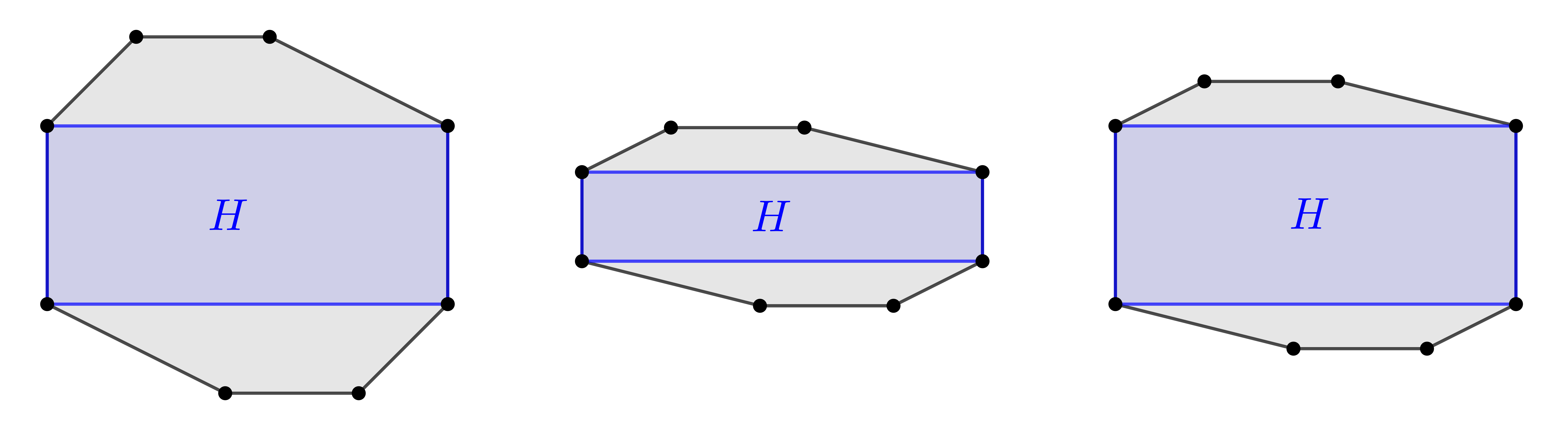}
\caption{The two step process illustrated with $\epsilon=\frac12$.}
\label{F:ColOntoH}
\end{figure}
Formally, this is a limit as $\epsilon\to 0$ of the surfaces obtained in two steps by first vertically scaling the surface by $\epsilon$ using the $GL(2, \bR)$ action, and then vertical stretching $H$ by $\epsilon^{-1}$ using a cylinder deformation, so that the height of $H$ stays constant. Since there are no vertical saddle connections, Masur's Compactness Criterion implies this process does not degenerate the surface. (In fact, the limit can be described concretely: it is obtained by gluing each point $p$ on the top of $H$ to the point on the bottom of $H$ where the straight line which leaves $p$ in the north direction first returns to $H$; this gluing is well defined on a co-finite set.) 

To conclude, perform a small generic real deformation so that any two horizontal saddle connections that are not $\cM$-parallel have different lengths, and then shear the surface so $V$ is vertical. 
\end{proof}

The rest of the proof takes place on a surface produced by Lemma \ref{L:ColOntoH}. For each horizontal saddle connection $a$ not contained in $V$, define $W_a$ to be the cylinder that crosses $H$ once, has $a$ as a cross curve, and does not intersect $V$. See Figure \ref{F:Wa}. 
\begin{figure}[h]
\includegraphics[width=.5\linewidth]{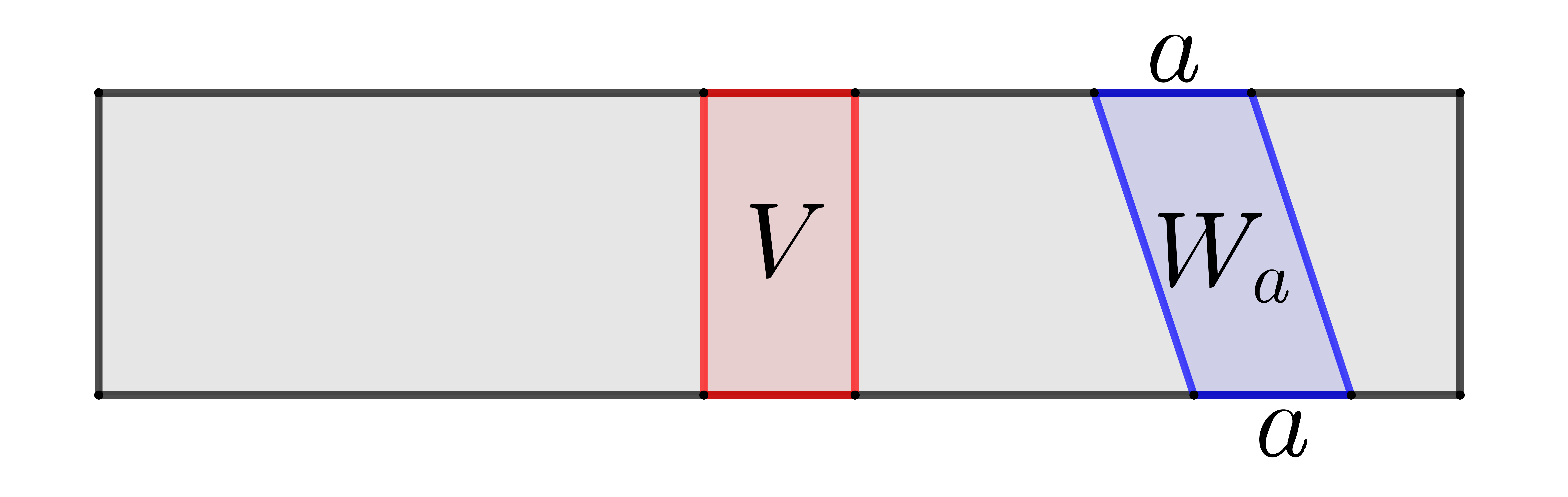}
\caption{}
\label{F:Wa}
\end{figure}

\begin{lem}\label{L:WaPrime}
For each $W_a$, there is at most one cylinder that is $\cM$-parallel to $W_a$. If such a cylinder exists, it has the same height as $W_a$. 
\end{lem}
\begin{proof}
This follows by ``over-collapsing $V$ to attack $W_a$", a procedure illustrated in Figure \ref{F:Vattack2}. The same idea was used previously in \cite{ApisaWrightDiamonds} and \cite{ApisaWrightGemini}. 
\begin{figure}[h]
\includegraphics[width=.455\linewidth]{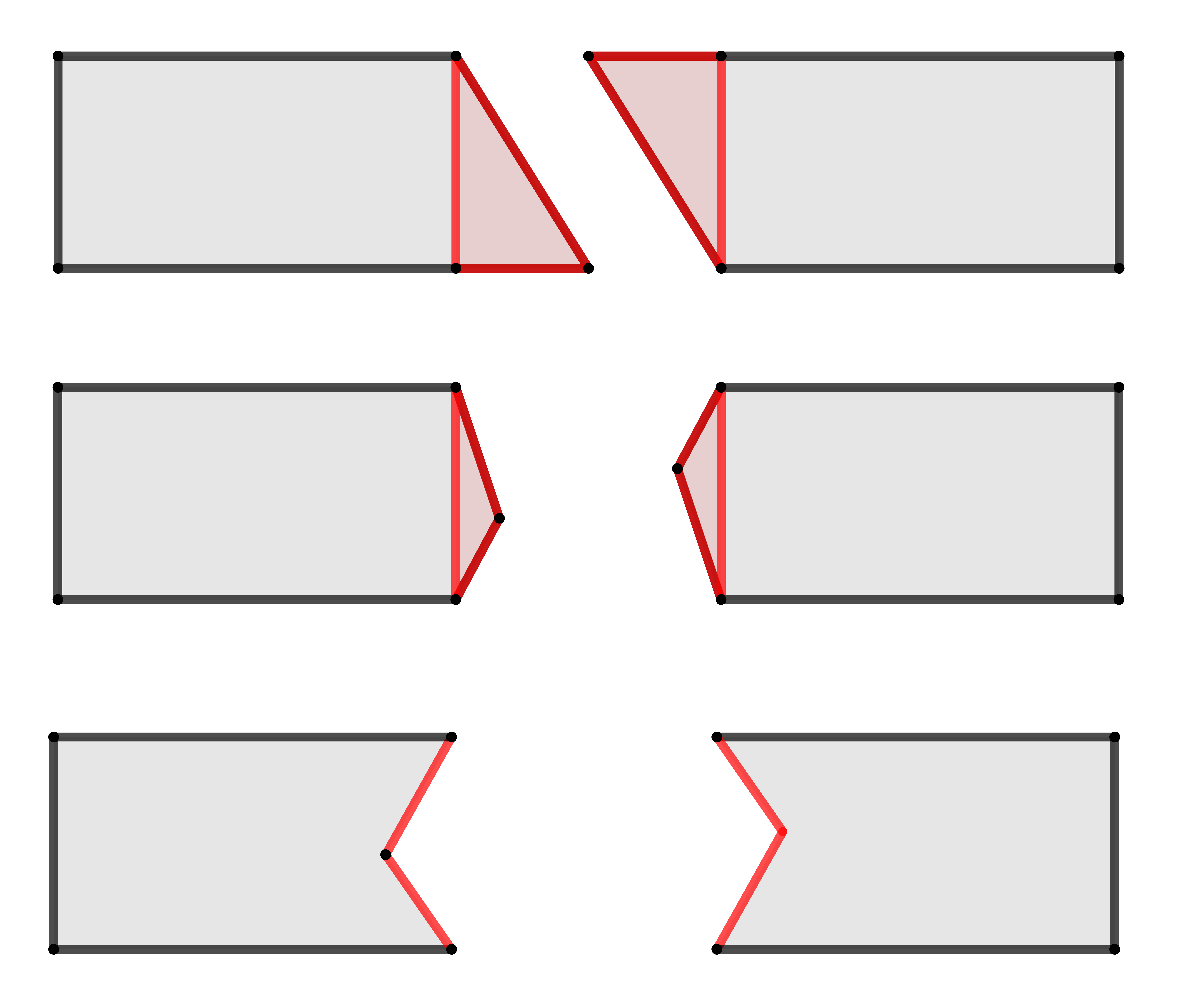}
\caption{}
\label{F:Vattack2}
\end{figure}

To accomplish this, start by shearing $V$ slightly so that it no longer contains a horizontal saddle connection. Then use the standard cylinder deformation of $V$ to reduce the height of $V$ without changing the horizontal foliation of the surface. Continuing in this direction, eventually $V$ reaches zero height, but because it does not contain horizontal saddle connections this does not degenerate the surface. Hence, this linear path in period coordinates can be continued past the point where $V$ reaches zero height, thus ``over-collapsing $V$". 

In this case the overcollapse can be understood concretely. We can think of $V$ has being composed of two triangles, and the deformation above simply changes these triangles as in Figure \ref{F:Vattack2}. The over-collapse can be obtained by removing two triangles from $H-V$, and making appropriate edge identifications. 

When these triangles first hit a cylinder $\cM$-parallel to $W_a$, it can hit at most two such cylinders. As the attack continues, these two cylinders change height (and indeed they lose an equal amount of height).   Corollary \ref{C:ConstantModuli} thus gives that there are at most two cylinders in the equivalence class of $W_a$. 

The same argument shows that if there are two cylinders in the equivalence class then their heights must be equal, since the moduli must stay rationally related throughout the attack. 
\end{proof}

If there is a cylinder $\cM$-parallel to $W_a$, we will denote it $W_a'$. If not, we will say that $W_a'$ does not exist. 

\begin{lem}\label{L:WaSnug}
If $W_a'$ exists, then every horizontal saddle connection that it passes through has the same length as $a$. 
\end{lem}

\begin{proof}
Otherwise, since $W_a$ and $W_a'$ have the same height, $W_a'$ passes through a horizontal saddle connection $c$ that is longer than $a$. 

Consider the cylinder deformation that horizontally compresses $W_c$ and, if it exists, $W_c'$; this deformation does not change the horizontal foliation. This deformation eventually reduces the length of $c$ to zero, and also changes the lengths of the horizontal saddle connections that $W_c'$ passes through. But, since $W_c'$ has the same height as $W_c$, and $a$ is smaller than $c$, $W_c'$ does not pass through $a$. So, the length of $a$ is unchanged. 

Our explicit definition of $W_a$ shows that $W_a$ persists along this deformation, and its height  does not go zero. But, the height  of $W_a'$ does go to zero, because $W_a'$ passes through $c$, and the length of $c$ goes to zero. This contradicts the fact that $W_a$ and $W_a'$ must always have the same height.
\end{proof}

\begin{lem}\label{L:WaNice}
If $W_a'$ exists, it is equal to $W_b$ for some horizontal saddle connection $b$. 
\end{lem}
This is equivalent to the statement that $W_a'$ passes through only one horizontal saddle connection, rather than more. If that saddle connection is $b$, then $W_a'=W_b$. Equivalently, $W_a'$ has the same circumference as $W_a$. Equivalently, the core curve of $W_a'$ intersects the core curve of $H$ only once. 

\begin{proof}
By Lemma \ref{L:WaSnug}, every horizontal saddle connection that $W_a'$ passes through is a core curve of $W_a'$. In particular, all these saddle connections are of the same length. Let $b_1, \ldots, b_k$ be the horizontal saddle connections that $W_a'$ passes through. We want to show $k=1$, so suppose to the contrary that $k>1$.

Without loss of generality (up to switching the words ``right" and ``left"), assume that $b_1$ is the closest of $a, b_1, \ldots, b_k$ to $V$ on the top boundary of $H$ in the left direction, i.e. the left-to-right distance from the right endpoint of $b_1$ to the left side of $V$ is minimal. Let $s$ be the closest of  $a, b_1, \ldots, b_k$ to $V$ on the bottom boundary of $H$ in the left direction. 

Because $W_a'$ does not cross $V$, it must pass up through $s$ and then through $b_1$, so $H\cap W_a'$ contains a parallelogram with $s$ on the bottom and $b_1$ on the top. Hence, because $k>1$, we conclude that $s\neq b_1$. One can also check that $s\neq a$, since $W_a$ and $W_a'$ are parallel and disjoint from $V$.  So, without loss of generality, assume $s=b_2$, as in Figure \ref{F:b1b2}. 
\begin{figure}[h]
\includegraphics[width=.5\linewidth]{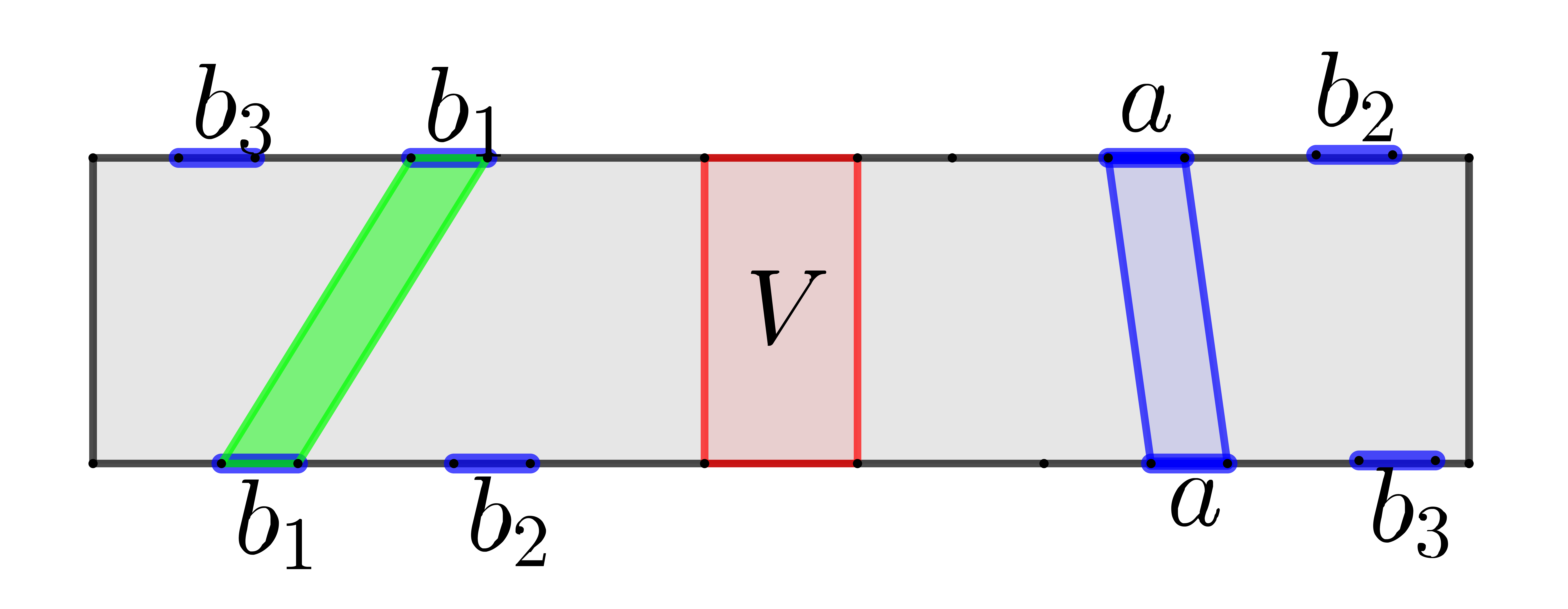}
\caption{The proof of Lemma \ref{L:WaNice}.}
\label{F:b1b2}
\end{figure}

Lemma \ref{L:WaPrime} gives that $W_{b_1}'$ has the same height as $W_{b_1}$. Our assumption on $(X,\omega)$ implies that the only horizontal saddle connections of the same length as $a$ are $ b_1, \ldots, b_k$. Hence Lemma \ref{L:WaSnug} implies that $W_{b_1}'$, if it exists, cannot pass through any horizontal saddle connections other than $a, b_1, \ldots, b_k$. However, $W_{b_1}'$ cannot pass through $b_2$, because $W_{b_1}'$ is disjoint from $V$, and by our choice of $b_1$ and $b_2$. 

Hence, considering the cylinder deformation in $W_{b_1}$, and $W_{b_1}'$ if it exists, we see that in fact $b_1$ and $b_2$ are not generically the same length. This is a contradiction, so we get the desired result. 
\end{proof}

\begin{lem}\label{L:Mstrat}
If, for every horizontal saddle connection $a$, $W_a'$ does not exist, then $\cM$ is a component of a stratum of Abelian differentials. 
\end{lem}
\begin{proof}
We will show that every deformation of $(X,\omega)$ in its stratum can be obtained by deforming the cylinders $H$, $V$, and the various $W_a$. Indeed, these deformation are sufficient to arbitrarily deform each edge in the rectangular presentation of $(X,\omega)$ illustrated above. 
%
\end{proof}

Therefore we assume that there is at least one horizontal saddle connection $s$ such that $W_s'$ exists. 

\begin{lem}
Under this assumption, the involution of the interior of $H$ that maps $V$ to itself extends to an involution of the surface.
\end{lem}
\begin{proof}
Let $a$ be any horizontal saddle connection for which $W_a'$ exists, and let $W_a'=W_b$. As in the proof of Lemma \ref{L:WaPrime}, by attacking from $V$ we see that $W_a$ and $W_b$ are interchanged by the involution. So in particular, $a$ and $b$ are interchanged by the involution. 

It now suffices to show that any saddle connection $c$ for which $W_c'$ does not exist is fixed by the involution.  This follows since, like $V$, $W_c$ is a nested free cylinder. So, letting $a$ and $b$ be horizontal saddle connections so that $W_a' = W_b$, the previous paragraph implies that the involution fixing $W_c$ must exchange $a$ and $b$ and hence coincide with the involution fixing $V$.
\end{proof}

We now claim that $\cM$ is equal to a quadratic double, having already shown in the previous lemma that it is contained in one. To do this, we will show that every deformation of $(X,\omega)$ on which the holonomy involution persists can be obtained using cylinder deformations of the pairs $\{W_a, W_a'\}$ when $W_a'$ exists, of individual $W_a$ when $W_a'$ does not exist, of $V$, and of $H$. Indeed, these deformation are sufficient to arbitrarily deform each fixed edge or exchanged pair of edges in the rectangular presentation of $(X,\omega)$ illustrated above. 
\end{proof}

\subsection{Background on cylindrical stability}\label{SS:CS}

Given a horizontally periodic surface $(X,\omega)\in\cM$, its \emph{twist space} is defined as the subspace of $T_{(X,\omega)}(\cM)$ of vectors that can be written as linear combinations of duals of core curves of horizontal cylinders. Its \emph{cylinder preserving space} is defined to be the subspace of $T_{(X,\omega)}(\cM)$ of cohomology classes evaluating to zero on all the core curves of horizontal cylinders on $(X,\omega)$. It is immediate from the definition that the twist space is contained in the cylinder preserving space. 

Following the terminology in \cite[Definition 2.4]{aulicino2016rank}, we say a  surface in $\cM$ is \emph{$\cM$-cylindrically stable}, or just cylindrically stable when $\cM$ is clear from context, if it is horizontally periodic and its twist space  is equal to its cylinder preserving space. 

We summarize a number of results that we will need on cylindrically stable surfaces. 

\begin{lem}\label{L:CS} 
Let $\cM$ be an invariant subvariety. 
\begin{enumerate}
\item\label{L:CS:LinCombo} A surface is cylindrically stable if and only if for every horizontal saddle connection there is a linear combination of core curves of horizontal cylinders that has the same period locally in $\cM$. 
\item\label{L:CS:TwistSpace} A surface is cylindrically stable if and only if its twist space has dimension $\rank(\cM)+\rel(\cM)$. No surface has twist space of larger dimension. 
\item\label{L:CS:rel0criterion} If $\cM$ has no rel, then a surface is cylindrically stable if and only if it has $\rank(\cM)$ horizontal equivalence classes. No surface has more horizontal equivalence classes. 
\item\label{L:CS:CoreCurves} If a surface is cylindrically stable, then the horizontal core curves span a subset of $T_{(X,\omega)}(\cM)^*$  of dimension $\rank(\cM)$.
\item\label{L:CS:rel0independent} If $\cM$ has no rel and a surface is cylindrically stable, then taking one core curve of each equivalence class gives a linearly independent subset of $T_{(X,\omega)}(\cM)^*$. 
\item\label{L:CS:perturb} If $\bk(\cM)=\bQ$, every surface  in $\cM$ can be perturbed to become cylindrically stable and square-tiled in such a way that all horizontal cylinders on the original surface stay horizontal in the perturbation. 
\end{enumerate}
\end{lem}

Here $\bk(\cM)$ denotes the smallest field which can be used to define the linear equations locally defining $\cM$ in period coordinates, as in \cite{Wfield}. 

\begin{proof}
This result can be thought of as a black box coming from \cite[Section 8]{Wcyl}, but nonetheless we give more specific references. 

The Cylinder Deformation Theorem and Theorem \ref{T:CDTConverse} imply that when $\cM$ has no rel, the dimension of the twist space of a horizontally periodic surface is equal to the number of horizontal equivalence classes. This observation shows that Claim \eqref{L:CS:rel0criterion} follows from Claim \eqref{L:CS:TwistSpace}. Notice too that Claim \eqref{L:CS:rel0independent} follows from Claims \eqref{L:CS:rel0criterion} and \eqref{L:CS:CoreCurves}. So we will prove the remaining claims.

Claim \eqref{L:CS:LinCombo} is immediate from \cite[Definition 8.1, Corollary 8.3]{Wcyl}, which gives that, the twist space can also be defined as the subspace of $T_{(X,\omega)}(\cM)$ of cohomology classes that are zero on all horizontal saddle connections.  

We will now prove Claim \eqref{L:CS:CoreCurves}. Let $T$ be the span of the core curves of the horizontal cylinders in $(T_{(X, \omega)}(\cM))^*$. By \cite[Lemma 8.8 and Corollary 8.11]{Wcyl}, $\dim T \geq \mathrm{rank}(\cM)$. Since the core curves of parallel cylinders span an isotropic subspace of $p(T_{(X, \omega)}(\cM))$, and any isotropic subspace has dimension at most $\mathrm{rank}(\cM)$, $\dim p(T) \leq \mathrm{rank}(\cM)$. Since the core curves of the cylinders are absolute homology classes, their span in $(p(T_{(X, \omega)}(\cM))^*$ is isomorphic to their span in $(T_{(X, \omega)}(\cM))^*$. Therefore, $\dim(T) = \mathrm{rank}(\cM)$.

For the first part of Claim \eqref{L:CS:TwistSpace}, note that if a surface is cylindrically stable, Claim \eqref{L:CS:CoreCurves} and \cite[Lemma 8.8]{Wcyl} imply that the cylinder preserving space has codimension $\rank(\cM)$ in $T_{(X, \omega)}(\cM)$. Hence it has dimension $\rank(\cM)+\rel(\cM)$. By  definition of cylindrically stable, the twist space is equal to the cylinder-preserving space, and hence must also have dimension $\rank(\cM)+\rel(\cM)$. Conversely, if the twist space has this dimension, then \cite[Lemma  8.10]{Wcyl} implies that the core curves span a subspace of $(T_{(X, \omega)}(\cM))^*$ of dimension $\rank(\cM)$, and hence that the cylinder preserving space has codimension $\rank(\cM)$.

The second part of Claim \eqref{L:CS:TwistSpace} is \cite[Corollary 8.11]{Wcyl} 

For Claim \eqref{L:CS:perturb}, recall that since $\bk(\cM)=\bQ$, square-tiled surfaces are dense in $\cM$; see for example \cite[Lemma 3.3]{ApisaWrightDiamonds}. Given this, the  claim follows from \cite[Sublemma 8.7]{Wcyl}, keeping in mind that perturbations can be arranged to be square-tiled and hence horizontally periodic.  The same argument is used in \cite[Lemma 6.9]{ApisaWrightDiamonds}.
\end{proof}

\subsection{An application}
We start with a somewhat cumbersome definition.

\begin{defn}\label{D:StrongCPath}
Suppose that $\bfC$ is a collection of parallel cylinders on a surface $(X, \omega)$ in an invariant subvariety $\cM$. A \emph{$\bfC$-path} is a continuous map $\gamma:[0, 1] \ra \cM$ such that $\gamma(0) = (X, \omega)$, and $\bfC$ persists at all points along this path, and the ratio of moduli and circumferences of any two cylinders in $\bfC$ is constant along the path. We will say that $\gamma(1)$ is \emph{$\bfC$-related to $(X, \omega)$}.
\end{defn}

\begin{rem}
If $\cM$ has no rel, Corollary \ref{C:ConstantModuli} gives that the condition that ``the ratio of moduli and circumferences of any two cylinders in $\bfC$ is constant along the path" is automatic whenever the cylinders in $\bfC$ are all pairwise $\cM$-equivalent.
\end{rem}

\begin{prop}\label{P:FindNested}
Suppose that $\bk(\cM)=\bQ$ and $\cM$ has no rel. Let $\bfC$ be a horizontal equivalence class on a surface $(X,\omega)$ in $\cM$. Then there is a $\bfC$-related cylindrically stable surface  $(X', \omega')$ on which the cylinders in $\bfC$ remain horizontal and where one of the following occurs:
\begin{enumerate}
    \item There is a free horizontal cylinder not in $\bfC$ that contains a nested free cylinder.
    \item None of the horizontal cylinders not in $\bfC$ are free.
\end{enumerate}
\end{prop}
\begin{rem}
Throughout this work, when we have considered collections $\bfC$ of cylinders, there has been a tacit understanding that $\bfC$ is nonempty. In this lemma, however, it is permissible to take $\bfC$ to be empty. We tacitly use this observation in the proof of Corollary \ref{C:easy}.
\end{rem}
\begin{proof}
Assume that every cylindrically stable $\bfC$-related surface $(X', \omega')$ on which $\bfC$ remains horizontal contains a free horizontal cylinder not in $\bfC$. Assume without loss of generality, perhaps after replacing $(X, \omega)$ with a  $\bfC$-related surface, that $(X, \omega)$ is square-tiled and that it contains as few horizontal free cylinders as possible, subject to the constraints that the cylinders in $\bfC$ remain horizontal and that the surface is cylindrically stable (such a surface exists by Lemma \ref{L:CS} \eqref{L:CS:perturb}).

By assumption there is a free horizontal cylinder $F$ on $(X, \omega)$ that does not belong to $\bfC$. Let $\gamma_F$ denote its core curve. Let $\bfA$ denote the collection of all horizontal cylinders on $(X, \omega)$. 

By Lemma \ref{L:CS} \eqref{L:CS:rel0independent}, there is a nonzero tangent vector $v \in T_{(X, \omega)}(\cM)$ that evaluates to zero on the core curves of every cylinder in $\bfA - \{ F \}$ but so that $v(\gamma_F)$ is a nonzero purely imaginary number. Since $\bk(\cM)=\bQ$, we may also assume that $\mathrm{Re}(v) = 0$ and that $\mathrm{Im}(v)$ is rational. 

By assuming that $v$ is sufficiently small, we can assume that the cylinders in $\bfA - \{F\}$ persist (and necessarily remain horizontal) on $(X, \omega) + v$. Since $v$ is rational, $(X, \omega) + v$ is square-tiled and hence horizontally periodic. Since $(X, \omega)$ had the fewest number of free horizontal cylinders, there must be a new free horizontal cylinder $F'$ on $(X, \omega) +  v$. 
 
By Corollary \ref{C:ConstantModuli}, $\{F'\}$ is an equivalence class, and the perturbation does not create new cylinders generically parallel to those in $\bfA - \{F\}$. By Lemma \ref{L:CS} \eqref{L:CS:rel0criterion}, since $(X,\omega)+v$ has at least $\rank(\cM)$ many equivalence classes, it is cylindrically stable. Since it cannot have more horizontal equivalence classes, horizontal cylinders on $(X, \omega) + v$ are exactly the ones in $(\bfA- \{ F\}) \cup \{ F'\}  $.
 
Assuming $v$ is sufficiently small, we can assume that $F$ has very small non-zero slope on $(X,\omega)+v$. Let $s$ be one of its boundary saddle connections. Since $F$ must be contained in $\overline{F}'$, we see that $F'$ must cross $s$, and hence that the height of $F'$ must be small. Since the area of $F'$ is at least that of $F$ on $(X,\omega)+v$, we get that $F'$ must have long circumference. In particular, we can assume that $F'$ has circumference longer than the sum of the circumferences of the cylinders in $\bfA - \{F\}$. That ensures that $F'$ contains a saddle connection on its top boundary that also appears on its bottom boundary. 
 
Therefore, the new horizontal cylinder $F'$ contains a nested cylinder, call it $T$. Let $\bfT$ be the equivalence class of $T$. 

\begin{lem}
$\bfT=\{T\}$. 
\end{lem}
\begin{proof}
The proof will use a variant of the technique used in the proof of Lemma \ref{L:ColOntoH} to obtain a surface where the cylinders in $\bfA-\{F\}$ stay horizontal, and where $\bfT$ becomes horizontal. Our assumption that $(X, \omega)$ had the minimum possible number of free horizontal cylinders will then give that $\bfT$ consists of a single free cylinder.

By replacing $(X,\omega)$ with the result of slightly shearing  $F'$, we can assume that $\overline{F'}$ does not contain any vertical saddle connections. Since $T$ is nested in $F'$ this implies in particular that the real part of the holonomy of the core curve of $T$ is nonzero. 

We will now build a sequence of surfaces depending on $\epsilon$ that have the same vertical foliation as $(X, \omega)$, via a two step process illustrated in Figure \ref{F:DontoT}. First  deform in the direction of $-i\sigma_{F'}$ until $F'$ has height $\epsilon$. Then  deform in the $i \sigma_{\bfT}$ direction  until the area of $T$ is one. Call this surface $(X_\epsilon, \omega_\epsilon)$.

\begin{figure}[h]\centering
\includegraphics[width=0.85\linewidth]{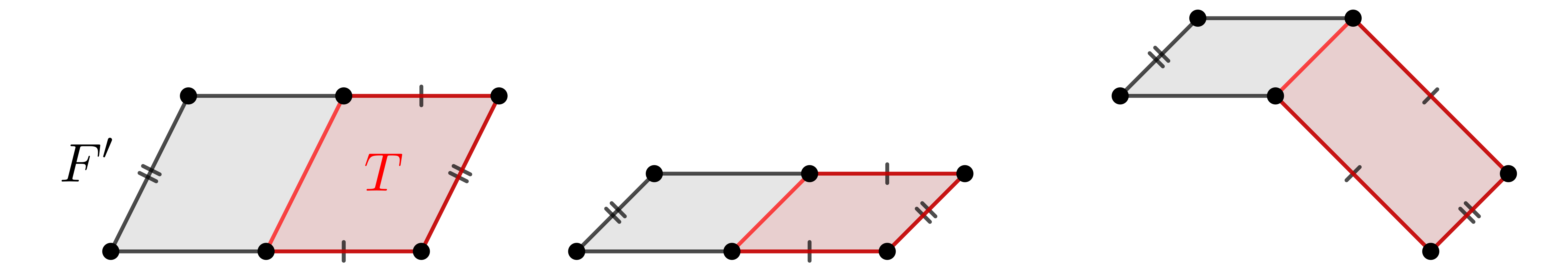}
\caption{The two step process for creating $(X_\epsilon, \omega_\epsilon)$.}
\label{F:DontoT}
\end{figure}

As $\epsilon$ tends to zero, the holonomy of the core curve of $T$ on $(X_\epsilon, \omega_\epsilon)$ tends to the real part of the holonomy of the core curve of $T$ on $(X, \omega)$,  and the area of $T$ is constant and equal to one. This shows that $T$ converges to a horizontal cylinder. Moreover, since we assumed that $F'$ had no vertical saddle connections it follows the limit of $(X_\epsilon, \omega_\epsilon)$ exists and is contained in $\cM$.  
Since the original surface was chosen to minimize the number of horizontal free cylinders, $\bfT$ must be a single free cylinder.
\end{proof}

Since $T$ is a free  cylinder nested in a free cylinder $F'$, this concludes the proof.
\end{proof}

\begin{proof}[Proof of Corollary \ref{C:easy}]
A horizontally periodic surface in a stratum has at most $g+s-1$ cylinders, for example by Lemma \ref{L:CS} \eqref{L:CS:TwistSpace}, so it is impossible for every horizontal equivalence class to have at least 2 cylinders when there are rank many equivalence classes. 
\end{proof}

\section{Finding useful cylinders}\label{S:hardest}

The goal of the section is the following result, which we will use to prove Proposition \ref{P:NoRelMain}. Given an equivalence class $\bfD$ of cylinders on $(X,\omega)\in \cM$, we will let $\what{\bfD}$ denote the union of $\bfD$ and all saddle connections parallel to $\bfD$, and we will let $\whatc{\bfD}=(X,\omega)\setminus \what{\bfD}$ denote its complement. 

\begin{thm}\label{T:FindD}
Assume that $\cM$ is an invariant subvariety with rank at least 3, no rel, and $\mathbf{k}(\cM ) = \mathbb{Q}$, and suppose that $\cM$ isn't a component of a stratum or a quadratic double. Suppose that $\bfC$ is an equivalence class of cylinders on $(X,\omega)$. Then, there exists a $\bfC$-related surface with generic equivalence classes $\bfD$ and $\bfD'$ such that
\begin{enumerate}
\item $\bfC$, $\bfD$ and $\bfD'$ are disjoint,
\item $\bfD$ has at least two cylinders, 
\item every saddle connection parallel to $\bfD$ is generically parallel to $\bfD$, and
\item  every component of $\whatc{\bfD}$ contains a cylinder from  $\bfC$ and one from  $\bfD'$. 
\end{enumerate}
\end{thm}

The main purpose of Theorem \ref{T:FindD} is to produce $\bfD$ with a favourable relationship with $\bfC$; the existence $\bfD'$ is of secondary importance.

We will derive Theorem \ref{T:FindD} from the following. 

\begin{thm}\label{T:InnerMost} 
Assume $\cM$ is an invariant subvariety with  rank at least 2 and no rel, and suppose that all parallel saddle connections on  $(X,\omega)\in \cM$ are generically parallel. Suppose that $\bfC, \bfD_1, \ldots, \bfD_{\rank(\cM)-1}$ are disjoint equivalence classes on $(X,\omega)$. Then, for some $1 \leq i \leq  \rank(\cM)-1$, every component of $\whatc{\bfD}_i$ contains a cylinder in $\bfC$ and a cylinder in each $\bfD_j, j\neq i$. 
\end{thm}
%
%
%

\begin{proof}[Proof of Theorem \ref{T:FindD} assuming Theorem \ref{T:InnerMost}] Since $\cM$ is neither a stratum nor a quadratic double, Theorem \ref{T:nested} gives that no surface in $\cM$ has a nested free cylinder. By Proposition \ref{P:FindNested}, we can therefore find a cylindrically stable  $\bfC$-related surface where the cylinders in $\bfC$ are horizontal and the only free horizontal cylinders (if there are any) belong to $\bfC$. Since an equivalence class with only one cylinder is a free cylinder, this implies that every equivalence class of horizontal cylinders aside from $\bfC$ contains at least two cylinders. 

Let $\bfC, \bfD_1, \bfD_2, \hdots, \bfD_{\mathrm{rank}(\cM)-1}$ denote the equivalence classes of horizontal cylinders on the newly constructed surface and perturb so that these equivalence classes persist and become generic and so that any two parallel saddle connections on the perturbed surfaces are actually generically parallel. By Theorem \ref{T:InnerMost}, there is some $i$ such that every component of $\whatc{\bfD}_i$ contains a cylinder from  $\bfC$ and from every $\bfD_j, j\neq i$. Setting $\bfD$ to be $\bfD_i$, and $\bfD'$ to be any $\bfD_j, j\neq i$, gives the result. 
\end{proof}

\subsection{Basic lemmas}

The following result will be a convenient tool for our analysis in this subsection. 

\begin{lem}\label{L:CorOfSW}
Let $(X,\omega)$ be a translation surface, and let $M$ be a minimal component for the horizontal straight line flow on $(X,\omega)$. Then, for all $\e>0$, there is a cylinder $C$ contained in $M$ whose $\e$-neighborhood contains $M$. 
\end{lem}

\begin{proof}[Sketch of proof]
A result of Smillie-Weiss implies that $M$ contains cylinders such that the imaginary part of the holonomy of the core curve is arbitrarily small \cite[Corollary 6]{SW2}; such cylinders are necessarily almost horizontal. Because $M$ is a minimal component, every sequence of cylinders in $M$ that become more and more horizontal must become more and more dense. 
\end{proof}

We also use the following direct corollary of the Cylinder Deformation Theorem, which is a variant of \cite[Proposition 3.2]{NW}.

\begin{cor}\label{C:CylinderProportion}
Let $\cM$ be an invariant subvariety. Let $\bfA$ and $\bfB$ be equivalence classes of cylinders. Then either $\bfA$ is disjoint from ${\bfB}$ and the saddle connections generically parallel to $\bfB$, or $\bfA$ intersects every cylinder in $\bfB$ and every saddle connection generically parallel to $\bfB$. 
\end{cor}

\begin{proof}
The result follows from deforming $\bfA$ using the Cylinder Deformation Theorem. 
\end{proof}

The next two lemmas are morally related to Theorem \ref{T:PrimeDecomp}, but the proofs we give will not make use of that result. 

\begin{lem}\label{L:Isolation}
Suppose $\bfD$ and $\bfD'$ are disjoint equivalence classes of cylinders on a surface $(X,\omega)$ in an invariant subvariety $\cM$. Suppose all parallel saddle connections on $(X,\omega)$ are generically parallel. 

Then every component of $\whatc{\bfD}$ that contains a saddle connection parallel to $\bfD'$ also contains a cylinder in $\bfD'$. 
\end{lem}

\begin{proof}
In the direction of $\bfD$, $(X, \omega)$ is partitioned into cylinders (which necessarily belong to $\bfD$) and minimal components for straight line flow.  Lemma \ref{L:CorOfSW} implies there is a cylinder $E$ which intersects $\bfD'$ but not $\bfD$. Let $\bfE$ be the equivalence class of $E$. Corollary \ref{C:CylinderProportion} gives that every cylinder of $\bfE$ must intersect a cylinder of $\bfD'$, and that $\bfE$ is disjoint from  $\what{\bfD}$. 

Deforming $\bfE$ thus only affects the components of $\whatc{\bfD}$ containing cylinders of $\bfD'$. Since deforming $\bfE$ changes $\bfD'$, it must also change all saddle connections generically parallel to $\bfD'$, giving the result. 
\end{proof}

\begin{lem}\label{L:AllOrNothing}
Suppose $\bfD$, $\bfD'$ and $\bfD''$ are disjoint equivalence classes of cylinders on a surface $(X,\omega)$ in an invariant subvariety $\cM$. Suppose all parallel saddle connections on $(X,\omega)$ are generically parallel. 

Then either every component of $\whatc{\bfD}$ that contains a cylinder in $\bfD'$ contains a cylinder in $\bfD''$, or no component of $\whatc{\bfD}$ contains both a cylinder in $\bfD'$ and a cylinder in $\bfD''$. 
\end{lem}

\begin{proof}
Suppose there is a component of $\whatc{\bfD}$ that contains both a cylinder in $\bfD'$ and a cylinder in $\bfD''$. Lemma \ref{L:CorOfSW} implies there is a cylinder $E$ which intersects $\bfD'$ and $\bfD''$ but not $\bfD$.

Let $\bfE$ be the equivalence class of $E$. Corollary \ref{C:CylinderProportion} gives that every cylinder of $\bfE$ must intersect both a cylinder in $\bfD'$ and one in $\bfD''$, that every cylinder of $\bfD' \cup \bfD$ intersects a cylinder of $\bfE$, and that $\bfE$ is disjoint from $\what{\bfD}$, giving the result.
\end{proof}

\subsection{Proof of Theorem \ref{T:InnerMost}}

For each $1\leq i \leq \rank(\cM)-1$, define $G_i$ to be the union of the components of $\whatc{\bfD}_i$ that contain a cylinder from $\bfC$. Note that because $G_i$ is bounded by saddle connections parallel to $\bfD_i$, we have  $G_i\neq G_j$ if $i\neq j$. 

\begin{lem}\label{L:Nesting}
For any $i\neq j$, if $\bfD_j$ is not contained in $G_i$, then $G_i \subset G_j$. 
\end{lem}

\begin{proof}
If $\bfD_j$ is not contained in $G_i$, it follows from Lemma \ref{L:AllOrNothing} that all cylinders of $\bfD_j$ are contained in components of $\whatc{\bfD}_i$ not containing cylinders of $\bfC$. Lemma \ref{L:Isolation} thus gives that  all the saddle connections  parallel to $\bfD_j$ are contained in components of $\whatc{\bfD}_i$ that do not contain cylinders of $\bfC$. 

Thus, each component of $\whatc{\bfD}_i$  that contains a cylinder of $\bfC$ is contained in such a component of $\whatc{\bfD}_j$. 
\end{proof}

For the remainder of the proof, fix $i$ such that $G_i$ is not contained in $G_j$ for any $j\neq i$. For this $i$, Lemma \ref{L:Nesting} gives the following immediate consequence. 

\begin{cor}
$G_i$ contains $\bfD_j$ for all $j\neq i$. 
\end{cor}

Lemma \ref{L:AllOrNothing} thus implies that every component of $\whatc{\bfD}_i$ that contains a cylinder from $\bfC$ also contains a cylinder from each $\bfD_j, j\neq i$. So the following lemma  completes the proof of Theorem \ref{T:InnerMost}.  

\begin{lem}
$(X,\omega) = \what{\bfD}_i \cup G_i$. 
\end{lem}

\begin{proof}
If not, there is a component $B$ of $\whatc{\bfD}_i$ that doesn't contain any cylinders of any of the $\bfD_j$ or of $\bfC$. 

Lemma \ref{L:CorOfSW} gives that there is a cylinder $E$ contained in $S$, which must be disjoint from all the  cylinders in $\bfD_j, j=1, \ldots, \rank(\cM)-1$ and  $\bfC$. Corollary \ref{C:CylinderProportion} thus gives that the equivalence class $\bfE$ of $E$ is also disjoint from all these cylinders. Thus, we have found $\rank(\cM)+1$ disjoint equivalence classes. This contradicts the fact that $\cM$ has no rel, because $p(T_{(X,\omega)}(\cM))$ is symplectic and of dimension $2\rank(\cM)$, and the standard shears in these equivalence classes span an isotropic subspace of dimension $\rank(\cM)+1$. 
\end{proof}

\section{Primality of the boundary}\label{S:Primality}

The goal of this short section is to find a simple criterion for when surfaces in a boundary component of a high rank invariant subvariety are connected. Throughout this section we will use the notation introduced in Section \ref{S:CylinderDegenerations}. 

\subsection{Cylinder degenerations are prime}

We begin with a general argument that shows that any cylinder degeneration produces a prime component of the boundary of $\cM$. This argument does not assume that $\cM$ has high rank. 

\begin{lem}\label{L:CylDegenPrimeBoundary}
Let $\bfC$ be a collection of generic cylinders on a surface $(X, \omega)$ in an invariant subvariety  $\cM$. Let $v \in \TwistC$ specify a cylinder degeneration. Then
$\cM_v$ is prime. 
\end{lem}
\begin{proof}
We begin with the following observation. 

\begin{sublem}\label{SL:sv}
Each component of $\Col_v(X,\omega)$ contains a finite set of parallel saddle connections, such that the sum of the holonomies of this finite set on one component is $\cM_v$-generically parallel to the analogous sum on any other component. 
\end{sublem}

\begin{proof}
Consider any saddle connection $s$ on the boundary of $\overline{\bfC}_v$. This $s$ remains a saddle connection along the collapse path, but on $\Col_v(X, \omega)$ may split into a finite set $s_v$ of saddle connections.

Along the collapse path $(X, \omega) - \overline{\bfC}$ is unchanged. Moreover, the cylinders in $\bfC - \bfC_v$ may have their heights change along the collapse path, but they persist on the boundary. In particular, this means that each component of $(X, \omega) - \overline{\bfC}_v$ is associated to a component of $\Col_v(X, \omega)$. Two components of $(X, \omega) - \overline{\bfC}_v$ could be associated to the same component of $\Col_v(X, \omega)$, and all components of $\Col_v(X, \omega)$ are associated to at least one component of $(X, \omega) - \overline{\bfC}_v$. 

If $s$ is contained in the boundary of a component of $(X, \omega) - \overline{\bfC}_v$, then $s_v$ is contained in the associated component of $\Col_v(X,\omega)$. 

Since the cylinders in $\bfC$ are generic, all $s$ arising this way are $\cM$-generically parallel, and hence the sums associated to each $s_v$ are $\cM_v$-generically parallel. 
\end{proof}

By Theorem \ref{T:PrimeDecomp}, if $\cM_v$ is not prime, then it is possible to produce a new surface in $\cM_v$ by applying an arbitrary element of $\mathrm{GL}(2, \mathbb{R})$ to the components of $\Col_v(X, \omega)$ in one prime factor while fixing the components of $\Col_v(X, \omega)$ in the other prime factors, contradicting Sublemma \ref{SL:sv}.
\end{proof}

\subsection{Prime boundary in high rank} We now turn to a somewhat miraculous property that distinguishes high rank invariant subvarieties.

\begin{lem}\label{L:ConnectedSurfacesinBoundary}
Suppose that $\cM'$ is a prime boundary component of a high rank invariant subvariety $\cM$ and that either
\begin{itemize}
\item $\mathrm{rank}( \cM' ) = \mathrm{rank}( \cM )$,  or 
\item $\mathrm{rank}( \cM' ) = \mathrm{rank}( \cM ) - 1$ and the surfaces in $\cM'$ have genus at least one less than those in $\cM$. 
\end{itemize}
Then the surfaces in $\cM'$ are connected.
\end{lem}
\begin{proof}
Let $(Y, \eta)$ be a surface in $\cM'$ and let $$(Y, \eta) = (Y_1, \eta_1) \cup \cdots \cup (Y_k, \eta_k)$$ be its decomposition into connected components. Suppose that the genus of $Y_i$ is $g_i$ and that the genus of a surface in $\cM$ is $g$. Our goal is to show that $k=1$.

Since $\cM$ is high rank, $\rank(\cM)\geq 1 + \frac{g}{2}$. Since $g\geq\sum g_i $, we have 
\[\rank(\cM)\geq 1+ \frac12 \sum_{i=1}^k g_i.\]
Since $\cM'$ is prime, the rank of any component of $\cM'$ is the same as the rank of $\cM'$ (by Lemma \ref{L:BoundarySymplectic}). The rank of the component of $\cM'$ containing $(Y_i, \eta_i)$ is bounded above by $g_i$ by definition of rank. Therefore, $\min( g_i )_{i=1}^k \geq \mathrm{rank}( \cM' )$ and so,
\[ \rank(\cM)\geq 1+ \frac{k}{2}\rank(\cM').\]
Set $r=\rank(\cM)$. If $\rank(\cM')=r$, this equation becomes 
$$-1\geq  \left(\frac{k}2-1\right)r,$$
which implies $k=1$. 

If $\rank(\cM') = r-1$, then, by assumption, the surfaces in $\cM'$ have genus at least one less than those in $\cM$ and so $g\geq1+\sum g_i$. This is an improvement over the estimate $g\geq\sum g_i $ used above. By the same reasoning as in the preceding paragraphs, we have  
$$r\geq \frac32+ \frac{k}2\left(r-1\right),$$
which implies 
$$r-1>  \frac{k}2\left(r-1\right),$$
which again implies $k=1$.
\end{proof}

\section{Proof of Proposition \ref{P:NoRelMain}}\label{S:Goldilocks}

In this section we will use Theorem \ref{T:FindD} to derive the following. 

\begin{thm}\label{T:GoodDD}
Suppose $\cM$ that has no rel, rank at least 3, and $\mathbf{k}(\cM) = \mathbb{Q}$. Assume that $\bfC$ is an equivalence class of  cylinders on $(X,\omega)\in \cM$ and that $(X, \omega)$ has no marked points.  Then there is a  $\bfC$-related $(X', \omega') \in \cM$ and an equivalence class $\bfD$ of generic cylinders on $(X', \omega')$ that is disjoint from $\bfC$ such that either a single or double degeneration of $\bfD$ has genus at least two less than $(X', \omega')$ and is contained in a prime component of the boundary. 

Moreover, on this single or double degeneration, $\bfC$ persists and there are no free marked points. 
\end{thm}

\begin{rem}\label{R:GenericParallelSC}
In the double degeneration case, it is implicit that all saddle connections on $(X', \omega')$ that are parallel to $\bfD$ are generically parallel, so that Assumption \ref{A:DD-Assumptions} is satisfied and the double degeneration is defined. 
%
\end{rem}


Theorem \ref{T:GoodDD} in turn easily implies Proposition \ref{P:NoRelMain}.



\begin{proof}[Proof of Proposition \ref{P:NoRelMain} given Theorem \ref{T:GoodDD}:]
%
%
%
In general, \cite[Theorem 1.5]{Wfield} gives 
$$\rank(\cM) \cdot \deg(\mathbf{k}(\cM))\leq 2g.$$
Hence the high rank assumption implies that  $\mathbf{k}(\cM) = \mathbb{Q}$.  
So the assumptions in Proposition \ref{P:NoRelMain} imply those in Theorem \ref{T:GoodDD}.

Since $\cM$ is not geminal, it contains a surface $(X, \omega)$ and an equivalence class $\bfC$ of cylinders that cannot be partitioned into free cylinders and pairs of twins (defined before Theorem \ref{T:Geminal}). Let $\cM'$ be the boundary of $\cM$ constructed in Theorem \ref{T:GoodDD} and let $(Y, \eta)\in \cM'$ be the single or double degeneration produced in Theorem \ref{T:GoodDD}, so  $\bfC$ persists on $(Y, \eta)$.

Note that $\rank(\cM')=\rank(\cM)-1$ by Corollaries \ref{C:RankMinus1} and \ref{C:L:DDRelZero}. (Lemma \ref{L:generic} gives that $\bfD$ is generic, which is required to apply Corollary \ref{C:RankMinus1}). The fact that the surfaces in $\cM'$ are connected is immediate from Lemma \ref{L:ConnectedSurfacesinBoundary}.

Finally, we will show that $\cM'$ is not geminal. This is expected because $\bfC$ certifies that $\cM$ is not geminal, and $\bfC$ persists on $(Y, \eta)\in \cM'$. We will give one of several possible technical justifications. 

We begin by claiming that the standard deformation of $\bfC$ is contained in the tangent space of $\cM'$, using the identification of the tangent space of $\cM'$ with a subspace of the tangent space of $\cM$.

\begin{lem}\label{L:StdSurvives}
$\sigma_{\bfC}\in T_{(Y, \eta)} (\cM')$.
\end{lem}
\begin{proof} 
In the single degeneration case, this is immediate since $\bfC$ and $\bfD$ are disjoint. (More formally, Lemma \ref{L:VinCv} gives that $\bfC$ does not intersect the vanishing cycles.) 

So suppose that we are in the case where we must double-degenerate $\bfD$. Lemma \ref{L:DefinitionOfLC} states that the tangent space to the component of the boundary containing the  double degeneration is identified with $T_{(X,\omega)} (\cM) \cap \Ann(L_\bfD)$, where 
 $L_\bfD \subset H_1(X,\Sigma)$ is the span of all saddle connections contained in $\overline\bfD$ as well as all saddle connections parallel to $\bfD$. 
 
By assumption, $\bfC$ is disjoint from $\bfD$ and hence, by Corollary \ref{C:CylinderProportion} and Remark \ref{R:GenericParallelSC}, from any saddle connection parallel to $\bfD$. 
Hence $\sigma_{\bfC}\in \Ann(L_\bfD)$ as desired. 
\end{proof}

Suppose to a contradiction that $\cM'$ is geminal. Since $\sigma_{\bfC}$ remains an element of $T_{(Y, \eta)}(\cM')$, no cylinder in $\Col(\bfC)$ can be a twin of a cylinder not in $\Col(\bfC)$, where $\Col(\bfC)$ denotes the cylinders on $(Y, \eta)$ that persist from $\bfC$. 

Moreover, since $\Twist( \Col(\bfC), \cM')$ is isomorphic to $\TwistC$, it follows that $\Twist( \Col(\bfC), \cM')$ is one-dimensional and hence $\Col(\bfC)$ consists of either one free cylinder or two twins. This implies in particular that $\sigma_{\bfC}$ is a multiple of $\sum_i \gamma_i^*$ where $\{ \gamma_i \}_i$ is the set of core curves of cylinders in $\Col(\bfC)$.  However, this cannot be the case since $\bfC$ contains two cylinders with distinct heights by assumption. This is a contradiction.
\end{proof}

\subsection{Lemmas for the single degeneration case}
In general, it is challenging to prove that genus decreases by more than one in (single) cylinder degenerations, and indeed there are surprisingly complicated situations where this does not occur. We can however verify a genus reduction of at least two in the following narrow circumstance. 

\begin{lem}\label{L:SingleDegenLoseTwo}
Let $\bfD$ be a generic equivalence class of cylinders on a surface $(X, \omega)$ in an invariant subvariety $\cM$. Let $v \in \Twist(\bfD, \cM)$ define a rank-reducing cylinder degeneration. 

Suppose additionally that there are homologous cylinders $H_1, H_2\notin \bfD$ such both components of the complement of the core curves of the $H_i$ contain cylinders of $\bfD$. 

Then the genus of $\Col_v(X,\omega)$ is at least two smaller than that of $(X,\omega)$.
\end{lem}

\begin{proof}
By Theorem \ref{T:GraphsFull}, $\Col_v(\bfD)$ is an acyclic graph. This cylinder degeneration can be achieved in two steps, by first collapsing the cylinders in $\bfD$ on one side of $H_1, H_2$, and then collapsing those on the other side. If either step didn't reduce the genus, then the portion of $\Col_v(\bfD)$ on the corresponding side of $H_1, H_2$ could not be acyclic.
\end{proof}

\begin{lem}\label{L:TwoSCnoFree}
Let $\bfD$ be a generic equivalence class of cylinders on a surface $(X, \omega)$ with no marked points in an invariant subvariety $\cM$ with no rel. Let $v \in \Twist(\bfD, \cM)$ define a rank-reducing cylinder degeneration. 
If $\Col_v(\bfC)$ contains at least two saddle connections, then $\Col_v(X,\omega)$ does not have any free marked points. 
\end{lem}
%
%
%

\begin{proof}
Theorem \ref{T:GraphsFull} gives that $\bfD_v=\bfD$, and hence Lemma \ref{L:GenericallyParallel} implies that all saddle connections in $\Col_v(\bfD_v)$ are generically parallel. 

Since $(X, \omega)$ does not have marked points, any marked points of $\Col_v(X,\omega)$ must lie on $\Col_v(\bfD)$. We first note that no saddle connection in $\Col_v(\bfD)$ can join a marked point to itself. If this were the case, this saddle connection would border a cylinder, which it would be generically parallel to, contradicting the fact that $\bfD_v=\bfD$. 

Hence, if any of the marked points were free, moving its position  would contradict the fact that all saddle connections in $\Col_v(\bfD_v)$ are generically parallel. 
\end{proof}

\subsection{Proof of Theorem \ref{T:GoodDD}}

Replacing $(X,\omega)$ with a  $\bfC$-related surface, let $\bfD$ and $\bfD'$ be the equivalence classes produced by Theorem \ref{T:FindD}.

\bold{The single degeneration case:} Suppose that $\bfD$ contains a pair of homologous cylinders $H_1, H_2$ such that both components of the complement of the core curves of the $H_i$ intersect $\whatc{\bfD}$. (This fails if one component is covered by cylinders in $\bfD$ and their boundary saddle connections.) In this case, the statement of Theorem \ref{T:FindD} gives that both components contain cylinders from  $\bfD'$. 

Lemma \ref{L:SingleDegenLoseTwo} gives that a single degeneration of $\bfD'$ loses two genus, and Lemma \ref{L:TwoSCnoFree} gives that this single degeneration does not contain free marked points. This proves  Theorem \ref{T:GoodDD} in this case. 

\bold{The double degeneration case:} Suppose that $\bfD$ does not contain such a pair of a pair of homologous cylinders $H_1, H_2$. In this case we will not make use of $\bfD'$. This case will however require an in-depth understanding of Section \ref{S:DoubleDegen}.  

Recall that Corollary \ref{C:CylinderProportion} gives that $\what\bfD$ is disjoint from $\bfC$. Fix $v \in \Twist(\bfD,\cM)$ that specifies a cylinder degeneration of $\cM$, and recall from Section \ref{S:DoubleDegen} that since $\cM$ has no rel this also uniquely specifies a double degeneration. 

\begin{lem}\label{L:Cstays}
Possibly after replacing $(X,\omega)$ with the result of a cylinder deformation in $\bfC$, we can assume that $\bfC$ persists on the double degeneration of $\bfD$, and each component of the double degeneration contains a cylinder of $\bfC$. 
\end{lem}

\begin{proof}
We first arrange for $\bfC$ to survive the double degeneration, using the following outline: the single degeneration of $\bfD$ doesn't affect $\bfC$, and, using a cylinder deformation of $\bfC$, we can assume that the heights of the cylinders in $\bfC$ are so large that they survive the surgeries that produce the double degeneration. We now give the details. 

Consider the one parameter path $(X_t, \omega_t)$ of translation surfaces where $\bfC$ is deformed by adding the multiple of the standard deformation $\sigma_\bfC$ that preserves and stretches out the direction parallel to $\bfD$. (For example, if $\bfD$ is vertical, this is accomplished by deforming in the direction of $i\sigma_\bfC$.) No new saddle connections parallel to $\bfD$ are created along the path, and the height of all cylinders in $\bfC$ can be assumed to be very large at the endpoint. 

The single degeneration $\Col_v(X_t, \omega_t)$ also varies continuously, and the unique vector certifying rel-scalability of $\Col_v(\bfD)$ is locally constant. 

Thus the  surgeries required to get from $\Col_v(\bfD)$ to the double degeneration, which are described in detail in Section \ref{S:DoubleDegen}, are locally constant. These surgeries involve making cuts of certain sizes at zeros and then re-gluing the boundary saddle connections of the cut surface in a prescribed pattern. 

We can assume at the end of the path that the height of each cylinder in $\bfC$ is much greater than the size of the surgeries required to pass from $\Col_v(\bfD)$ to the double degeneration. Thus, at the endpoint of the path, we can assume that each cylinder of $\bfC$ persists on the double degeneration. 

By the construction of the double degeneration described in Section \ref{S:DoubleDegen}, the components of the double degeneration  correspond bijectively to the components of $\whatc{\bfD}$. The statement of Theorem \ref{T:FindD} gives that all components of $\whatc{\bfD}$ contain cylinders of $\bfC$, and hence all components of the double degeneration contain cylinders of $\bfC$. 
\end{proof}

\begin{cor}\label{C:prime}
The invariant subvariety $\cM_v^{doub}$ containing the double degeneration of $\bfD$ is prime. 
\end{cor}

\begin{proof}
This follows from Theorem \ref{T:PrimeDecomp}, since if the component of the boundary containing the double degeneration were not prime, it would be a product, and that would contradict the fact that the cylinders in $\bfC$ are generically parallel to each other and appear on each component.
\end{proof}

\begin{lem}\label{L:DoubleDegenTwoGenus}
The double degeneration of $\bfD$ loses at least two genus. 
\end{lem}

\begin{proof}
If the core curves of $\bfD$ span a $k$-dimensional subspace in homology, then the double degeneration loses at least $k$ genus by Lemma \ref{L:DefinitionOfLC}. So if $k>1$, then the double degeneration of $\bfD$ loses at least two genus. 

Hence we suppose $k=1$, which means exactly that all cylinders in $\bfD$ are homologous. Recalling that $\bfD$ contains at least two cylinders, pick two distinct cylinders $D_1, D_2\in \bfD$. Let the core curve of $D_i$ be $\gamma_i$. 

By assumption, one of the two components of $(\gamma_1\cup \gamma_2)^c$ does not intersect $\whatc{\bfD}$, since otherwise we would be in the single degeneration case. Thus, there is a component that is covered by cylinders in $\bfD$ and their boundary saddle connections. Since all the cylinders in $\bfD$ are homologous, this means that there are cylinders $D, D' \in \bfD$ such that the top of $D$ is glued entirely to the bottom of $D'$.

Since there are no marked points, $\overline{D\cup D'}$  contains a subsurface of genus at least 2. Since double-degenerating $\bfD$ collapses this subsurface the result follows. 
\end{proof}

Lemma \ref{L:Cstays} gives that $\bfC$ persists on the double degeneration. Corollary \ref{C:prime} gives that  $\cM_v^{doub}$ is prime. Corollary \ref{C:L:DDRelZero} gives that $\cM_v^{doub}$ doesn't have any rel, and hence cannot have free marked points.  This concludes the proof of  Theorem \ref{T:GoodDD}.

\section{Typical rank-preserving degenerations}\label{S:Good}

Throughout this section $(X, \omega)$ will be a surface in an invariant subvariety $\cM$ and we will use the notation introduced in Section \ref{S:CylinderDegenerations}. In particular, if $\bfC$ is an equivalence class on $(X, \omega)$ and $v \in \TwistC$, we will use the notation $\bfC_v$, $t_v$, and $\Col_v$ introduced in Section \ref{S:CylinderDegenerations}.  

It will be necessary to place a mild genericity assumption on our degenerations, using the following definition.

\begin{defn}
A cylinder degeneration of $(X,\omega)\in \cM$ defined by $v \in \TwistC$ will be called \emph{typical} if $\bfC$ is generic and all  cylinders in $\bfC_v$ have constant ratio of heights on all perturbations of $(X,\omega)$ in $\cM$. \end{defn}

This definition is motivated by Lemma \ref{L:GenericallyParallel}, which states that, for typical degenerations, $\cM_v$ is codimension one and that all the saddle connections in $\Col_v(\bfC_v)$ are generically parallel to each other. 

The purpose of this section is to prove three results. The first produces typical degenerations, the second shows that such degenerations are often rank-preserving, and the third shows that typical rank-preserving degenerations do not create free marked points.

\begin{lem}\label{L:typical-exists}
For any generic equivalence class $\bfC$ on a surface $(X,\omega)$ in an invariant subvariety $\cM$, there exists  $v\in \mathrm{Twist}( \bfC, \cM)$ which defines a typical cylinder degeneration. 
\end{lem}

\begin{proof}
Assume without loss of generality that $\bfC$ is horizontal. For each pair of cylinders in $\bfC$, the subset of $v\in \TwistC$ where these two cylinders have constant ratio of heights along the path $(X,\omega)+tv$ is defined by a single linear equation, which is vacuous if the two cylinders have generically a constant ratio of heights and otherwise defines a hyperplane.  

Consider $v \in \TwistC$ not contained in any of the hyperplanes just described. By construction, for any deformation in a direction in  $\TwistC$, the cylinders in $\bfC_v$ have constant ratio of height, since otherwise $v$ would lie in one of the hyperplanes above. 

Proposition \ref{P:TwistDecomp} with $w=\omega$ now implies that all the cylinders in $\bfC_v$ have constant ratios of heights for any deformation in $\cM$. 

A final issue is that $v$ might not define a cylinder degeneration, if the path $(X,\omega)+tv, 0 \leq  t < t_v$ does not diverge in the stratum. This however can be easily corrected as follows. First replace $v$ by its imaginary part. Then add to $v$ a multiple $c\sigma_\bfC$ of the standard shear in $\bfC$, with $c\in \bR$ chosen so that $(X,\omega)+t_v c \sigma_\bfC$ has a vertical saddle connection in some cylinder of  $\bfC_v$. Both changes only affect the real part of $v$, and since $\bfC$ is horizontal $\bfC_v$ only depends on the imaginary part, so $\bfC_v$ remains the same. Because of the choice of $c$, now there is a saddle connection in $\bfC_v$ that has length going to zero along the cylinder degeneration path. 
%
%
%
\end{proof}

\begin{defn}\label{D:involved}
We will say that $\bfC$ is \emph{involved with rel} if some vector in $\ker(p) \cap T_{(X, \omega)}(\cM)$ evaluates non-trivially on a cross curve of a cylinder in $\bfC$.
\end{defn}

\begin{lem}\label{L:good-exists}
If $\bfC$ is generic and involved with rel, then every typical cylinder degeneration of $\bfC$ is rank-preserving. 
\end{lem}

\begin{proof}
First suppose that not all pairs of cylinders in $\bfC$ have generically constant ratio of height. In this case the definition of typical gives that $\bfC_v \neq \bfC$, and so Theorem \ref{T:GraphsFull} gives that the degeneration is rank-preserving. 

Next suppose that all pairs of cylinders in $\bfC$ have generically constant ratio of height. In this case $\bfC_v=\bfC$. Also $v$ must be a multiple of the standard deformation of $\bfC$, and hence there is a saddle connection $s$ in $\overline{\bfC}$ whose length goes to zero along the cylinder degeneration path. 

Let $w \in \ker(p) \cap T_{(X, \omega)}(\cM)$ evaluate non-trivially on a cross curve of a cylinder in $\bfC$. Because the cylinders in $\bfC$ have generically the same height, possibly after re-scaling we can assume that $w$ evaluated on any cross curve of any cylinder in $\bfC$ gives the height of that cylinder. Since $s$ must cross some of the curves in $\bfC$, we can conclude that $w(s)\neq 0$. Hence Lemma \ref{L:RankReducing} gives that the degeneration is rank-preserving.
\end{proof}

For the third main result of this section, it will be helpful to first note the following consequence of Theorem \ref{T:CDTConverse}. 

\begin{cor}\label{C:p0}
Let $\bfC$ be an equivalence class on a surface $(X,\omega)$ in an invariant subvariety $\cM$. 
If $v\in \TwistC$ is such that deforming $(X,\omega)$ in the direction of any complex multiple of $v$ does not change the area, then $p(v)=0$. 

In particular, if $\gamma_1$ and $\gamma_2$ are the core curves of cylinders in $\bfC$ with equal circumference and $\gamma_1^* - \gamma_2^* \in T_{(X,\omega)}(\cM)$, then $\gamma_1$ and $\gamma_2$ are homologous. 
\end{cor}

\begin{proof}
Using Theorem \ref{T:CDTConverse}, we can write $v= c \sigma_\bfC + r$, where $c\in \bC$ and $r$ is purely relative. Deforming in the direction of purely relative classes doesn't change the area of a translation surface, but there is a multiple of $\sigma_\bfC$ which corresponds to dilating (rather than shearing) $\bfC$, and this does change the area. This gives $c=0$, establishing the main claim. 

The second claim follows immediately from the first. 
\end{proof}

\begin{lem}\label{L:FreeMarkedPointAfterDegenerating}
If $v\in \TwistC$ defines a typical rank-preserving degeneration, and $(X, \omega)$ does not have marked points, then $\Col_v(X, \omega)$ does not have any free marked points.
\end{lem}

\begin{proof}
Suppose, in order to find a contradiction, that $\Col_v(X, \omega)$ does have a free marked point $p$. 

Without loss of generality, we assume that $\bfC$ is horizontal and that the only horizontal cylinders on $(X, \omega)$ are contained in $\bfC$. Since $(X, \omega)$ contains no marked points, the free point $p$ on $\Col_v(X, \omega)$ must be the endpoint of a saddle connection $s$ in $\Col_v(\bfC_v)$. 

We will show now that $s$ joins $p$ to itself. If $s$ joins $p$ to another zero or marked point, then $s$ cannot be generically parallel to another saddle connection, by definition of free marked point. Lemma \ref{L:GenericallyParallel} gives that the saddle connections in $\Col_v(\bfC_v)$ must be generically parallel to each other, so we conclude $\Col_v(\bfC_v) = \{ s\}$. Since $\cM$ and $\cM_v$ have the same rank, Theorem \ref{T:GraphsFull} gives that $\Col_v(\bfC_v)$ is a strongly connected directed graph. But a single saddle connection joining two distinct points is not strongly connected. This proves our claim that $s$ joins $p$ to itself.

\begin{figure}[h]\centering
\includegraphics[width=0.35\linewidth]{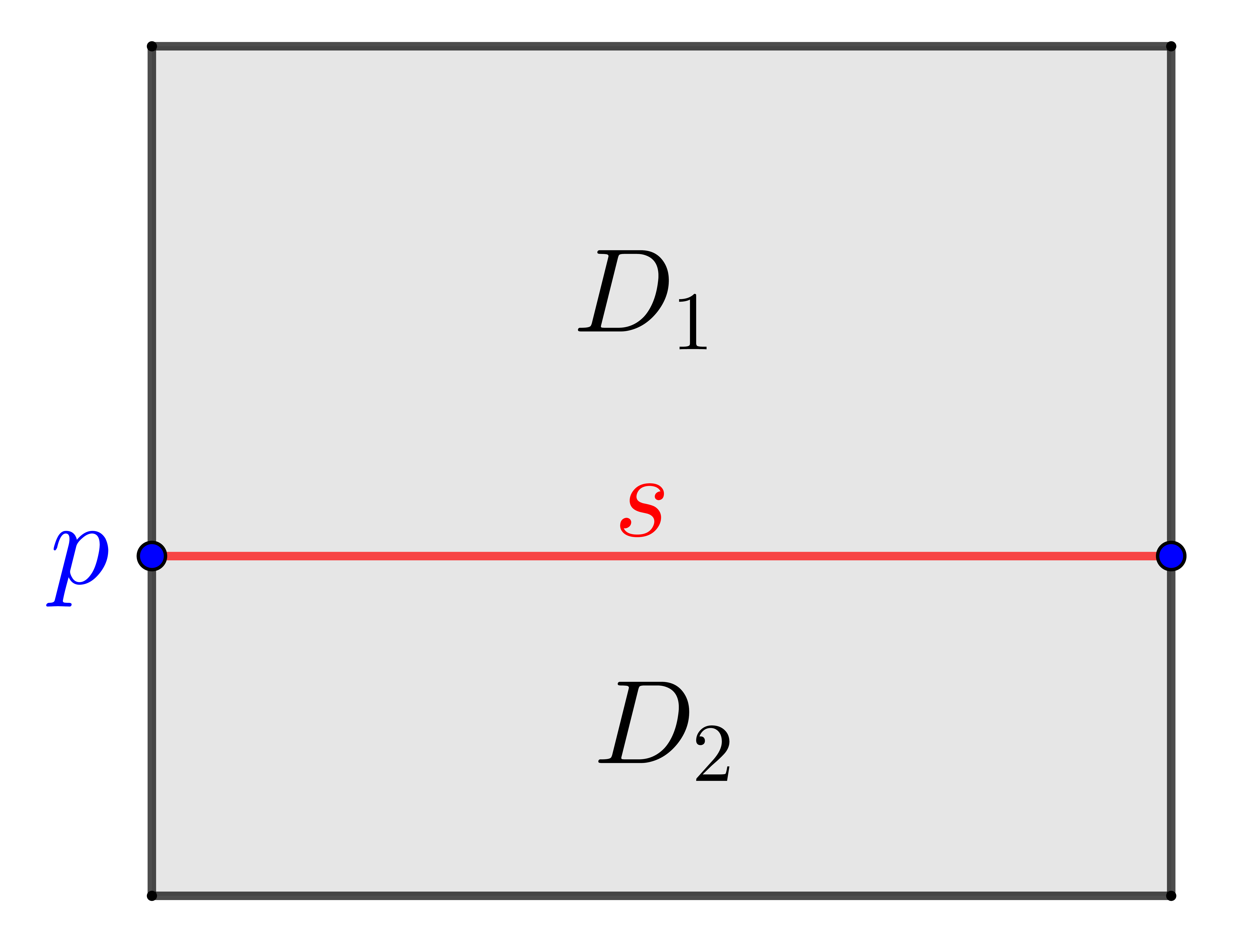}
\caption{The saddle connection $s$.}
\label{F:sD1D2}
\end{figure}

Any saddle connection joining a marked point to itself is a core curve of a cylinder on the corresponding surface without marked points. Thus, as in Figure \ref{F:sD1D2}, there are two cylinders $D_1$ and $D_2$ on $\Col_v(X, \omega)$, both of which have an entire boundary component consisting of $s$; these are the cylinders above and below $s$. Suppose without loss of generality that $D_2$ lies below $s$. 

Let $C_1$ and $C_2$ be the cylinders on $(X, \omega)$ with $\Col_v(C_i) = D_i$; these are simply the same cylinders viewed before the degeneration.

\begin{sublem}
$C_1$ and $C_2$ are homologous. 
\end{sublem}

\begin{proof}
Let $\gamma_i$ denote the core curve of $D_i$. Deforming in the direction of $\gamma_1^* - \gamma_2^*$ moves $p$ while keeping the rest of the surface unchanged. So since  $p$ is a free marked point, $\gamma_1^* - \gamma_2^*$ belongs to $T_{\Col_v(X, \omega)}(\cM_v)$. 

Notice that since the $D_i$ are horizontal, the $C_i$ are horizontal and hence contained in $\bfC$. (In particular, this shows that $\bfC \ne \bfC_v$.) We will continue to use $\gamma_i$ to denote the core curve of $C_i$. 

Since $\gamma_1^* - \gamma_2^*$ belongs to $\TwistC$, Corollary \ref{C:p0} gives that $\gamma_1$ and $\gamma_2$ are homologous. 
\end{proof}

%

 

Keeping in mind that $C_1$ and $C_2$ are homologous, let $\bfC_0$ be set of cylinders in $\bfC_v$ that lie above $C_2$ and below $C_1$, so that the closure $\overline{\bfC}_0$ covers the region above $C_2$ and below $C_1$. 

Let $D$ be the cylinder that lies directly above $C_2$. If $D$ and $C_2$ have the same circumference, then there must be a marked point separating them, which gives a contradiction. If not, we have the following. 

\begin{sublem}\label{SL:sbot}
If $D$ and $C_2$ do not have the same circumference, then there is oriented  closed loop in $\overline{\bfC}_v$ consisting of segments that travel along core curves of cylinders in $\bfC_v$, and vertical downwards oriented segments. 
\end{sublem}

\begin{figure}[h]\centering
\includegraphics[width=0.35\linewidth]{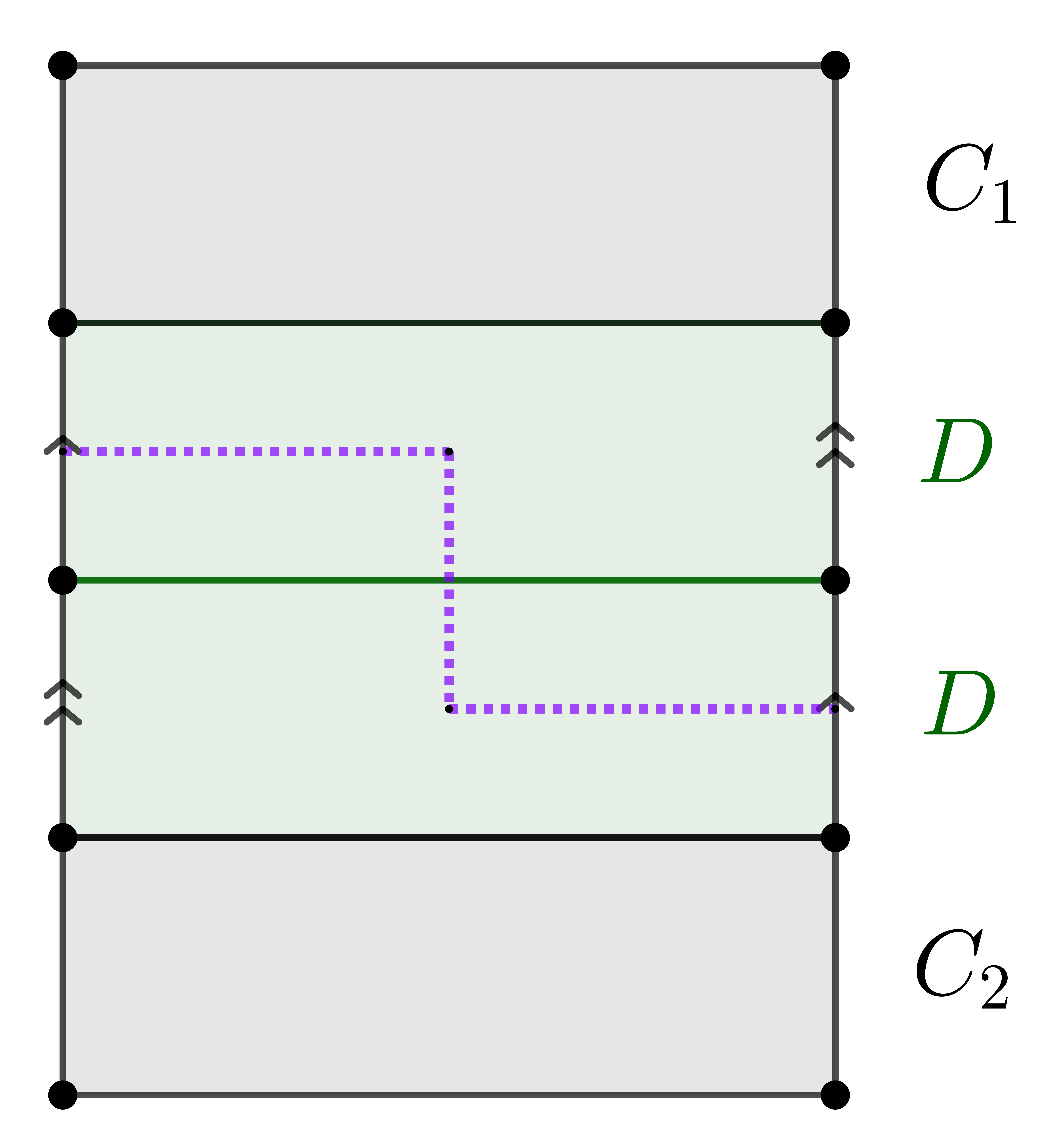}
\caption{The proof of Sublemma \ref{SL:sbot}}
\label{F:sbot}
\end{figure}

\begin{proof}
The assumption gives that there is a saddle connection on the bottom of $D$ that does not border $C_2$. On the other side of this saddle connection lies a cylinder  $D' \in \bfC_0$. See Figure \ref{F:sbot} for an example where $D'=D$. Because $C_1$ and $C_2$ are homologous, the bottom of $D'$ is glued entirely to cylinders in $\bfC_0 \cup \{C_2\}$. 

We can form a path by following the core curves of $D$, then going vertically down into $D'$, and then if necessary following the core curve of $D'$, and then going into a cylinder $D''$ below that. Eventually, this path must visit the same cylinder twice, and in this way we can obtain a closed loop as desired. See Figure \ref{F:sbot} for an example where $D'=D$. 
\end{proof}

The imaginary part of the holonomy of such a loop is non-zero, contradicting Lemma \ref{L:Loop2}. 
\end{proof}

%
%
%

\section{Proof of Proposition \ref{P:RelMain}}\label{S:RelMain}


Suppose that $\cM$ is an invariant subvariety of genus $g$ surfaces without marked points, and assume that $\cM$ has high rank, is not geminal, and has rel. Since $\cM$ is not geminal, there is a generic equivalence class $\bfC$ of cylinders on a surface $(X, \omega)$ in $\cM$ so that $\bfC$ cannot be partitioned into free cylinders and pairs of twins. 

We will give the proof of Proposition \ref{P:RelMain} in two cases. In both cases, we will use without further comment that a typical cylinder degeneration of an equivalence class involved in rel 
\begin{itemize}
\item is rank-preserving by Lemma \ref{L:good-exists}, 
\item is connected by Lemmas \ref{L:CylDegenPrimeBoundary} and \ref{L:ConnectedSurfacesinBoundary}, and 
\item does not have free marked points by Lemma \ref{L:FreeMarkedPointAfterDegenerating}. 
\end{itemize}
Hence, it suffices to find such a degeneration that is not geminal. 

The following lemma addresses the easier case, when the twist space of $\bfC$ doesn't contain the rel of $\cM$. 

\begin{lem}\label{L:Rel1:Twist}
If $\mathrm{Twist}(\bfC, \cM )$ does not contain $T_{(X,\omega)}(\cM) \cap \ker(p)$, then Proposition \ref{P:RelMain} holds.  
\end{lem}
\begin{proof}
By Lemma \ref{L:CS} \eqref{L:CS:perturb}, we can perturb $(X,\omega)$ in such a way that $\bfC$ remains horizontal and the surface becomes cylindrically stable. It follows from the definition of cylindrically stable that the twist space of $(X,\omega)$ then contains $T_{(X,\omega)}(\cM) \cap \ker(p)$, since the cylinder preserving space always contains $T_{(X,\omega)}(\cM) \cap \ker(p)$ by definition. Since $\mathrm{Twist}( \bfC, \cM )$ does not contain all the rel there is an equivalence class $\bfD$ of horizontal cylinders that is disjoint from $\bfC$ and involved in rel. 

After perturbing again, we may assume that $\bfD$ is generic. Using Lemma \ref{L:typical-exists}, pick $v \in \Twist(\bfD, \cM)$ that defines a typical cylinder degeneration. 

It suffices to show that $\cM_v$ is not geminal. Indeed, no cylinder in $\Col_v(\bfC)$ can be the twin of a cylinder not in $\Col_v(\bfC)$, since  the standard twist in $\bfC$ can also be performed in $\cM_v$, as in Lemma \ref{L:StdSurvives}. So, because $\bfC$ cannot be partitioned into free cylinders and pairs of twins, we get the same statement for the equivalence class of $\Col_v(\bfC)$.
\end{proof}

In light of Lemma \ref{L:Rel1:Twist}, we will suppose that  $\mathrm{Twist}(\bfC, \cM )$ contains all the rel in $T_{(X, \omega)}(\cM)$. So in particular, $\bfC$ is involved in rel. The rest of this section gives the proof of Proposition \ref{P:RelMain} in this case, by assuming that that all  typical degenerations of $\bfC$ are geminal and deriving a contradiction. 

Using Lemma \ref{L:typical-exists}, pick $v \in \mathrm{Twist}( \bfC, \cM )$ defining a typical cylinder degeneration. By assumption, $\cM_v$ is geminal, so since $\bfC_v$ is not all of $\bfC$ we can find a set  $\bfC_0$ of either one or two cylinders in $\bfC\setminus \bfC_v$ such that $\Col_v(\bfC_0)$ is either a single free cylinder or a pair of twins. 

In particular, $\mathrm{Twist}( \bfC_0, \cM )$ is one-dimensional and, using the identification of $T_{\Col_v(X, \omega)}(\cM_v)$ with a subspace of $T_{(X, \omega)}(\cM)$, spanned by the standard shear $\sigma$ of $\Col_v(\bfC_0)$. 

\begin{lem}\label{L:Rel:UnequalHeights}
There are two cylinders in $\bfC_0$ and they do not have generically equal heights.
\end{lem}
\begin{proof}
Suppose in order to derive a contradiction that the cylinders in $\bfC_0$ all have generically identical heights; this holds vacuously if $\bfC_0$ has only a single cylinder. For simplicity of notation assume additionally that $\bfC_0$ contains a vertical saddle connection; this can be arranged by shearing the surface.

Thus $-i\sigma$ defines a typical cylinder degeneration, and hence $\cM_{-i\sigma}$ is geminal. Hence $\Col_{-i\sigma}(\bfC)$ is partitioned into twins and free cylinders. Since $\Col_{-i\sigma}(\bfC - \bfC_0)$ consists of cylinders isometric to those in $\bfC - \bfC_0$, it follows that $\bfC$ has a decomposition into twins and free cylinders, which is a contradiction.
\end{proof}

By replacing $(X, \omega)$ with a perturbation and using Lemma \ref{L:Rel:UnequalHeights}, we may assume that the two cylinders in $\bfC_0$ have different heights. We will denote the two cylinders $C_1$ and $C_2$, where $C_1$ is shorter than $C_2$. Letting $\gamma_1$ and $\gamma_2$ denote their (consistently oriented) core curves, up to scaling, $\sigma = \gamma_1^* + \gamma_2^*$. 

Now that we have perturbed, $\bfC_{-i\sigma} = \{C_1\}$ and so $-i\sigma$ is typical. Since $C_2$ is not free, there is another cylinder $C_3 \in \bfC$ such that $\{\Col_{-i\sigma}(C_j)\}_{j=2}^3$ is a pair of twins. Letting $\gamma_3$ denote the (consistently oriented) core curve of $C_3$, we have that $\gamma_2^* + \gamma_3^*$ also belongs to $T_{(X, \omega)}(\cM)$. 

\begin{lem}\label{L:hom}
$C_1$ and $C_3$ are homologous to each other.
\end{lem}

\begin{proof}
Note that 
$$\gamma_1^* - \gamma_3^* = (\gamma_1^* + \gamma_2^*) - (\gamma_2^* + \gamma_3^*)\in \TwistC.$$
Since $C_2$ and $C_1$ become twins in a degeneration, they have the same circumference. Since $C_2$ and $C_3$ become twins in a degeneration, they also have the same circumference. So Corollary \ref{C:p0} gives the result.  
\end{proof}

For convenience, we will now assume that $C_1$ and $C_3$ contain vertical cross curves. This can be arranged for example by shearing the whole surface so a cross curve of $C_1$ becomes vertical, and then using a real multiple of the shear $\gamma_2^* + \gamma_3^*$ so that a cross curve of $C_3$ becomes vertical. Hence, if $w=i(\gamma_1^* - \gamma_3^*)$, then both $w$ and $-w$ define cylinder degenerations. Since $\bfC_w=\{C_3\}$ and $\bfC_{-w}=\{C_1\}$  both contain only a single cylinder, these cylinder degenerations are typical. 

We can thus observe that Theorems \ref{T:AAW} and \ref{T:Geminal} imply that both $\cM_{\pm w}$ are quadratic doubles. The idea of the remainder of the proof is to show that the holonomy involutions in these two degenerations are in some sense incompatible. This incompatibility is expected since (the images of) $C_1$ and $C_2$ are exchanged in one of these degenerations, while (the images of) $C_2$ and $C_3$ are exchanged in the other. 

Recall that a cylinder is said to be \emph{half-simple} if one of its boundaries is a single saddle connection and the other consists of two saddle connections of equal length. We will refer to these two boundaries as the \emph{simple} and \emph{half-simple} boundaries of the cylinder respectively. 

\begin{lem}\label{L:HighRankQDCyl}
In a high rank quadratic double, all generic cylinders are simple or half-simple. 
\end{lem}

\begin{proof}
The proof follows from the following two straightforward consequences of foundational results. 

\begin{sublem}\label{SL:HighRankZeros}
A quadratic double is high rank if and only if the corresponding stratum of quadratic differentials has at least 6 odd order zeros. 
\end{sublem}

\begin{proof}
This follows from Riemann-Hurwitz formula together with the formula for the rank of a stratum of quadratic differentials recorded in \cite[Lemma 4.2]{ApisaWright}. 
\end{proof} 

In \cite{MZ}, Masur and Zorich showed that the generic cylinders in strata of quadratic differentials other than $\cQ(-1^4)$ are one of five types. These cylinder types are reviewed in  \cite[Section 4.1]{ApisaWright}, where they were named simple cylinders, simple envelopes, half-simple cylinders, complex envelopes, and complex cylinders.

\begin{sublem}\label{SL:NoComplex}
Any stratum of quadratic differentials with at least 6 odd order zeros does not have generic cylinders that are complex cylinders or complex envelopes. 
\end{sublem}
\begin{proof}
This requires the work of Masur-Zorich \cite[Theorems 1 and 2]{MZ}; the exact consequence of their work we need is recalled in \cite[Theorem 4.8 (2)]{ApisaWrightDiamonds}. 

By this result, the complement of a complex envelope is a connected translation surface with boundary, and the complement of a complex cylinder is the union of two connected translation surfaces with boundary.  Therefore, any odd order zeros must be on the boundary of the complex cylinder or complex envelope. This shows that there are at most four odd order zeros in  strata with these types of cylinders. 
\end{proof}

A simple cylinder on a quadratic differential gives rise to  a twin pair of simple cylinders on the holonomy double cover; a half-simple cylinder gives rise to a twin pair of half simple cylinders; and a simple envelope gives rise to either a simple cylinder, or possibly a twin pair of simple or half-simple cylinders if there are marked points. These possibilities are illustrated in \cite[Figure 4.1]{ApisaWrightDiamonds}. Thus the lemma follows from the two sublemmas. 
\end{proof}

\begin{cor}
$C_i$ is either simple or half-simple. 
\end{cor}
\begin{proof}
We show the statement for $i=1$; the $i=2, 3$ cases are similar. 

By Lemma \ref{L:GenericallyParallel}, $\Col_w(C_1)$ is a generic cylinder. Hence Lemma \ref{L:HighRankQDCyl} gives that $\Col_w(C_1)$ is simple or half-simple. It follows that $C_1$ is simple or half simple, since $\Col_w(C_1)$  has at least as many saddle connections on each boundary component as $C_1$. 
\end{proof}

Perturb once more so that $(X, \omega)$ is horizontally periodic, while preserving the property that $\bfC$ is generic. Assume too after perturbation that any two horizontal cylinders that have identical heights in fact have generically identical heights. 

Assume without loss of generality that the bottom boundary of $C_2$ is simple. Let $D$ be the horizontal cylinder below $C_2$, so the single boundary saddle connection that is the bottom of $C_2$ appears in the top of $D$. 

Notice that since $(X, \omega)$ has no marked points, $D$ cannot have the same length as $C_2$ and hence $D \notin \{C_1, C_3\}$. In particular, $\Col_{\pm w}(D)$ remains a cylinder, and the single saddle connection that is the bottom of $\Col_{\pm w}(C_2)$ appears in the top of  $\Col_{\pm w}(D)$.

Let $J_{\pm w}$ denote the holonomy involution on $\Col_{\pm w}(X, \omega)$. Since the deformations specified by  $\pm w$ fix the complement of $\{C_1, C_3\}$ there are cylinders $D_{\pm}$ such that $\Col_{\pm w}(D_{\pm}) = J_{\pm w}(D)$. 

\begin{lem}
$D_+ = D_-$ and this cylinder is glued to the top boundary of $C_1$ and the top boundary of $C_3$.
\end{lem}
\begin{proof}
$D_+ = D_-$ since otherwise $\Col_w(D_+)$, $\Col_w(D_-)$, and $\Col_w(D)$ would be generically isometric, contradicting the fact that in a quadratic double it is possible to deform a cylinder and its image under the holonomy involution, without changing the rest of the surface. 

Recall that $\bfC_w=\{C_3\}$. Since the single saddle connection which is the bottom boundary of $C_2$ is glued to the top boundary of $D$, the corresponding statement is true for $\Col_w(C_2)$ and $\Col_w(D)$. Hence the top of $J_w(\Col_w(C_2))=\Col_w(C_1)$ consists of a single saddle connection, which is glued to the bottom of $J_w(\Col_w(D)) =\Col_w(D_+)$.

The top boundary of $C_1$ cannot border $C_2$ or $C_3$, since it consists of a single saddle connection, and this would result in a marked point. Hence, keeping in mind that $\bfC_w=\{C_3\}$, the fact that the top of $\Col_w(C_1)$ is contained in the bottom of $\Col_w(D_+)$ implies that the top of $C_1$ is contained in the bottom of $D_+$. 

The same argument with $-w$ instead of $w$ shows that the top boundary of $C_1$ is contained in the bottom boundary of $D_-$.
\end{proof}

We now have a contradiction since $D_+ = D_-$ appears on the top boundary of both $C_1$ and $C_3$, and yet these cylinders have homologous core curves by Lemma \ref{L:hom}.

\bibliography{mybib}{}
\bibliographystyle{amsalpha}
\end{document}